\pdfoutput=1
\documentclass[11pt,reqno,final]{amsart}

\usepackage[margin=23mm]{geometry}

\usepackage{setspace}
\usepackage{indentfirst}

\usepackage[T1]{fontenc}
\usepackage[utf8]{inputenc}

\usepackage[backend=biber,style=numeric,
url=false,doi=true,eprint=true,isbn=false,
hyperref=auto,backref=false]{biblatex}
\usepackage{amsmath}
\usepackage{amssymb}
\usepackage{stmaryrd}
\usepackage{mathrsfs}
\usepackage{tikz-cd}

\newtheorem{theorem}{Theorem}[section]
\newtheorem{lemma}[theorem]{Lemma}
\newtheorem{proposition}[theorem]{Proposition}
\newtheorem{corollary}[theorem]{Corollary}
\newtheorem{definition}[theorem]{Definition}
\newtheorem{remark}[theorem]{Remark}
\newtheorem{example}[theorem]{Example}

\newtheoremstyle{intro}% name
  {12pt}%      Space above, empty = `usual value'
  {12pt}%      Space below
  {\itshape}%  Body font
  {}%          Indent amount (empty = no indent, \parindent = para indent)
  {\bfseries}% Thm head font
  {.}%         Punctuation after thm head
  {.5em}%      Space after thm head: " " = normal interword space;
  {}%          Thm head spec

\theoremstyle{intro}
\newtheorem{introthm}{Theorem}

\newtheorem*{introques*}{Question}

\let\oldmarginpar\marginpar
\renewcommand\marginpar[1]{\-\oldmarginpar[\raggedleft\footnotesize #1]{\raggedright\footnotesize #1}}

\newsavebox{\overlongequation}
\newenvironment{dontbotheriftheequationisoverlong}
 {\begin{displaymath}\begin{lrbox}{\overlongequation}$\displaystyle}
 {$\end{lrbox}\makebox[0pt]{\usebox{\overlongequation}}\end{displaymath}}

\usepackage[pagebackref=false,unicode=true,bookmarks=true,
colorlinks=true,linktoc=page,linkcolor=red,citecolor=blue,
pdfstartview={FitH},final=true]{hyperref}
\hypersetup{hidelinks}

\usepackage[color]{showkeys}
\definecolor{refkey}{rgb}{1,0,0}
\definecolor{labelkey}{rgb}{0,0,1}

\allowdisplaybreaks[3]

\binoppenalty=\maxdimen
\relpenalty=\maxdimen
\DeclareMathOperator{\hypercohomology}{\mathbb{H}}
\DeclareMathOperator{\hochschildcohomology}{\textit{HH}}

\DeclareMathOperator{\pr}{pr}

\DeclareMathOperator{\Hom}{Hom}

\DeclareMathOperator{\Der}{Der}
\DeclareMathOperator{\tot}{tot}
\DeclareMathOperator{\id}{id}
\DeclareMathOperator{\rk}{rk}

\DeclareMathOperator{\pbw}{pbw}

\DeclareMathOperator{\Bott}{Bott}

\DeclareMathOperator{\CE}{CE}
\DeclareMathOperator{\DR}{DR}

\DeclareMathOperator{\poly}{poly}
\DeclareMathOperator{\ver}{ver}

\newcommand{\NN}{\mathbb{N}}
\newcommand{\NO}{\mathbb{N}_0}
\newcommand{\ZZ}{\mathbb{Z}}

\newcommand{\RR}{\mathbb{R}}
\newcommand{\CC}{\mathbb{C}}
\newcommand{\into}{\hookrightarrow}
\newcommand{\onto}{\twoheadrightarrow}
\newcommand{\isomorphism}{\cong}
\newcommand{\xto}[1]{\xrightarrow{#1}}

\newcommand{\dual}{^\vee}
\newcommand{\inv}{^{-1}}
\newcommand{\DD}{\mathfrak{D}}
\newcommand{\XX}{\mathfrak{X}}
\newcommand{\argument}{\mathord{\color{black!25}-}}

\newcommand{\degree}[1]{\left| #1 \right|}
\newcommand{\length}[1]{\left| #1 \right|}
\newcommand{\frakg}{\mathfrak{g}}
\newcommand{\lie}[1]{\mathcal{L}_{#1}}
\newcommand{\interior}[1]{\iota_{#1}}
\newcommand{\exterior}{d_{\DR}}
\newcommand{\hochschild}{d_{\mathscr{H}}}
\newcommand{\Tpoly}[1]{\mathcal{T}_{\poly}^{#1}}
\newcommand{\Tpolya}[1]{\Lambda^{#1+1}B}
\newcommand{\Dpoly}[1]{\mathcal{D}_{\poly}^{#1}}
\newcommand{\verticalTpoly}[1]{\mathscr{T}_{\poly}^{#1}}
\newcommand{\verticalDpoly}[1]{\mathscr{D}_{\poly}^{#1}}
\newcommand{\verticalX}{\XX_{\ver}}
\newcommand{\verticalD}{\DD_{\ver}}
\newcommand{\shuffle}[2]{\mathfrak{S}_{#1}^{#2}}
\newcommand{\sections}[1]{\Gamma(#1)}
\newcommand{\enveloping}[1]{\mathcal{U}(#1)}
\newcommand{\jet}[1]{\mathcal{J}({#1})}
\newcommand{\bracket}[2]{[#1,#2]}
\newcommand{\schouten}[2]{[#1,#2]}
\newcommand{\gerstenhaber}[2]{\llbracket #1,#2\rrbracket}
\newcommand{\duality}[2]{\left\langle#1\middle|#2\right\rangle}
\newcommand{\anchor}{\varrho}
\newcommand{\baranchor}{\bar{\anchor}}
\newcommand{\KK}{\Bbbk}
\newcommand{\cn}{\nabla^\lightning}	
\newcommand{\cnt}{\widetilde{\nabla}^\lightning}

\newcommand{\cA}{\mathcal{A}}
\newcommand{\cB}{\mathcal{B}}
\newcommand{\cL}{\mathcal{L}}
\newcommand{\cE}{\mathcal{E}}
\newcommand{\EE}{\mathcal{E}}
\newcommand{\cM}{\mathcal{M}}
\newcommand{\cF}{\mathcal{F}}
\newcommand{\cR}{\mathcal{R}}
\newcommand{\cQ}{\mathcal{Q}}

\newcommand{\po}{\pr_0}
\newcommand{\dabott}{d_A^{\Bott}}
\newcommand{\dace}{d_A}
\newcommand{\dau}{d_A^{\mathcal{U}}}
\newcommand{\UcL}[1]{\enveloping{\cL}^{\otimes #1+1}}
\newcommand{\etendu}[1]{#1_\natural}
\newcommand{\perturbed}[1]{\breve{#1}}
\newcommand{\dbott}{d_A^{\Bott}}
\newcommand{\dhoch}{\dau+\mathfrak{d}_\mathscr{H}}
\newcommand{\wa}{\sections{\Lambda^\bullet A\dual}}

\newcommand{\pbwb}{\widetilde{\pbw}}
\newcommand{\nablab}{\nabla^G}
\newcommand{\nablat}{\widetilde{\nabla}}

\newcommand{\perturbation}{\rho}
\newcommand{\suspended}{s}

\title[Polyvector fields and polydifferential operators associated with Lie pairs]
{Polyvector fields and polydifferential operators\\ associated with Lie pairs}

\thanks{Research partially supported by NSF grants DMS-1406668,
DMS-1707545 and DMS-2001599, and NSA grant H98230-14-1-0153.}

\author{Ruggero Bandiera}
\address{Dipartimento di Matematica, Sapienza Università di Roma}
\email{bandiera@mat.uniroma1.it}

\author{Mathieu Stiénon}
\address{Department of Mathematics, Pennsylvania State University}
\email{stienon@psu.edu}

\author{Ping Xu}
\address{Department of Mathematics, Pennsylvania State University}
\email{ping@math.psu.edu}

\begin{document}

%\listoflabels\newpage

\begin{abstract}
We prove that the spaces
$\tot\big(\wa\otimes_R\Tpoly{\bullet}\big)$
and $\tot\big(\wa\otimes_R\Dpoly{\bullet}\big)$
associated with a Lie pair $(L,A)$ each carry an $L_\infty$ algebra structure
canonical up to an $L_\infty$ isomorphism with the identity map as linear part.
These two spaces serve, respectively, as replacements
for the spaces of formal polyvector fields and formal polydifferential
operators on the Lie pair $(L,A)$.
Consequently, both $\hypercohomology^\bullet_{\CE}(A,\Tpoly{\bullet})$
and $\hypercohomology^\bullet_{\CE}(A,\Dpoly{\bullet})$
admit unique Gerstenhaber algebra structures.
Our approach is based on homotopy transfer
and the construction of a Fedosov dg Lie algebroid
(i.e.\ a dg foliation on a Fedosov dg manifold).
\end{abstract}

\maketitle

\tableofcontents

\section*{Introduction}

The algebraic structures of the spaces of polyvector fields
and of polydifferential operators on a manifold play a crucial role
in deformation quantization: Kontsevich's famous formality theorem
asserts that, for a smooth manifold $M$, the Hochschild--Kostant--Rosenberg map
extends to an $L_\infty$ quasi-isomorphism from the dgla of polyvector fields
on $M$ to the dgla of polydifferential operators on $M$
\cite{MR2062626,MR2699544,MR2102846,MR1914788,MR1944574}.

In this paper, we study the algebraic structures of ``polyvector fields''
and ``polydifferential operators'' on Lie pairs.
Throughout the paper, we use the symbol $\KK$ to denote
either of the fields $\RR$ and $\CC$.
A \emph{Lie algebroid} over $\KK$ is a $\KK$-vector bundle $L\to M$ together
with a bundle map $\anchor:L\to T_M\otimes_\RR\KK$ called \emph{anchor}
and a Lie bracket $\bracket{\argument}{\argument}$ on the sections of $L$ such that
$\anchor:\sections{L}\to\XX(M)\otimes_\RR\KK$ is a morphism of Lie algebras
and \[ \bracket{X}{fY}=f\bracket{X}{Y}+\anchor_{X}(f)Y ,\] for all $X,Y\in\sections{L}$ and $f\in C^\infty(M,\KK)$.
By a \emph{Lie pair} $(L,A)$, we mean an inclusion $A\into L$ of Lie algebroids
over a smooth manifold $M$.

Lie pairs arise naturally in a number of classical areas of mathematics such as Lie theory,
complex geometry, foliation theory, and Poisson geometry.
A complex manifold $X$ determines a Lie pair over $\CC$: viz.\ $L= T_X\otimes\CC$ and $A=T_X^{0,1}$.
A foliation $F$ on a smooth manifold $M$ determines a Lie pair over $\RR$: viz.\ $L=T_M$
and $A=T_F$ is the integrable distribution on $M$ tangent to the foliation $F$.
A manifold equipped with an action of a Lie algebra $\frakg$ gives rise to a Lie pair in a natural way
(see~\cite[Example~5.5]{MR1460632} and~\cite{MR1650045,MR3650387}).

Given a Lie pair $(L,A)$, the quotient $L/A$ is naturally an $A$-module \cite{MR3439229}.
When $L$ is the tangent bundle to a manifold $M$ and $A$ is an integrable distribution on $M$,
the infinitesimal $A$-action on $L/A$ is given by the Bott connection \cite{MR0362335}.

A Lie pair $(L,A)$ gives rise to two natural cochain complexes
\begin{equation}\label{ginger}
\big(\tot\big(\wa\otimes_R\Tpoly{\bullet}\big),d_A^{\Bott}\big)
\quad\text{and}\quad
\big(\tot\big(\wa\otimes_R\Dpoly{\bullet}\big),\dhoch\big)
\end{equation}
constructed as follows.
Denoting the algebra of smooth functions on the manifold $M$ by $R$,
we set $\Tpoly{\bullet}=\bigoplus_{k=-1}^{\infty}\Tpoly{k}$, where $\Tpoly{-1}=R$ and
$\Tpoly{k}=\sections{\Lambda^{k+1}(L/A)}$ for $k\geqslant 0$.
The Bott $A$-connection on $L/A$ makes every $\Tpoly{k}$ an ${A}$-module.
We can thus consider the complex of $A$-modules with trivial differential
\begin{equation*} \begin{tikzcd}
0 \arrow{r} & \Tpoly{-1} \arrow{r}{0} &
\Tpoly{0} \arrow{r}{0} & \Tpoly{1} \arrow{r}{0} &
\Tpoly{2} \arrow{r}{0} & \cdots
\end{tikzcd} \end{equation*}
Its Chevalley--Eilenberg cohomology
$\hypercohomology^\bullet_{\CE}(A,\Tpoly{\bullet})$
is the cohomology of the total cochain complex
\begin{equation}\label{lemon}
\big(\tot\big(\wa\otimes_R\Tpoly{\bullet}\big),\dabott\big)
.\end{equation}

Similarly, denoting the universal enveloping algebra of the Lie algebroid $L$
by $\enveloping{L}$, we set $\Dpoly{\bullet}=\bigoplus_{k=-1}^{\infty}\Dpoly{k}$,
where $\Dpoly{-1}=R$; $\Dpoly{0}=\frac{\enveloping{L}}{\enveloping{L}\sections{A}}$; and
$\Dpoly{k}$ with $k\geqslant 1$ is the tensor product $\Dpoly{0}\otimes_R\cdots\otimes_R\Dpoly{0}$
of $(k+1)$-copies of the left $R$-module $\Dpoly{0}$.
Multiplication in $\enveloping{L}$ from the left by elements of $\sections{A}$ (and $R$) induces
an $A$-module structure on the quotient $\frac{\enveloping{L}}{\enveloping{L}\sections{A}}$.
This action of $A$ on $\Dpoly{0}$ extends naturally to an action of $A$
on $\Dpoly{k}$ for each $k\geqslant 1$.
In fact, $\Dpoly{0}$ is a cocommutative coassociative coalgebra over $R$
whose comultiplication $\Delta:\Dpoly{0}\to\Dpoly{0}\otimes_R\Dpoly{0}$
is a morphism of $A$-modules. Therefore, the induced Hochschild complex
\[ \begin{tikzcd}
0 \arrow{r} & \Dpoly{-1} \arrow{r}{\hochschild} & \Dpoly{0} \arrow{r}{\hochschild} &
\Dpoly{1} \arrow{r}{\hochschild} & \Dpoly{2} \arrow{r}{\hochschild} & \cdots
\end{tikzcd} \]
is a complex of $A$-modules.
Its Chevalley--Eilenberg cohomology $\hypercohomology^\bullet_{\CE}(A,\Dpoly{\bullet})$
is the cohomology of the total cochain complex
\begin{equation}\label{lime}
\big(\tot\big(\wa\otimes_R\Dpoly{\bullet}\big),\dhoch\big)
,\end{equation}
where we use the abbreviated symbol $\mathfrak{d}_{\mathscr{H}}$
to denote the operator $\id\otimes\hochschild$
--- see Equation~\eqref{eq:Rome} for more details.

For instance, for the Lie pair $L=T_X\otimes\CC$ and $A=T^{0,1}_X$
arising from any complex manifold $X$,
the cochain complexes \eqref{lemon} and~\eqref{lime} are precisely the complexes
$\big(\Omega^{0,\bullet}(\Tpoly{\bullet}(X)),\bar{\partial}\big)$
and $\big(\Omega^{0,\bullet}(\Dpoly{\bullet}(X)),\bar{\partial}+\hochschild \big)$,
which are known to carry differential graded Lie algebra (a.k.a.\ dgla) structures.
The corresponding Chevalley--Eilenberg cohomology groups
$\hypercohomology_{\CE}^{\bullet}(A,\Tpoly{\bullet})$
and $\hypercohomology_{\CE}^{\bullet}(A,\Dpoly{\bullet})$
are isomorphic to the sheaf cohomology group $\hypercohomology^\bullet(X,\Lambda^\bullet T_X)$
and the Hochschild cohomology group $\hochschildcohomology^{\bullet}(X)$, respectively.

For a generic Lie pair $(L,A)$, however, there is no obvious way
to upgrade the cochain complexes \eqref{ginger} to dgla's (or $L_\infty$ algebras).
Here is an example. The cochain complex
$\big(\tot\Omega^{\bullet}_F(\Lambda^\bullet(T_M/T_F)),d_F^{\Bott}\big)$
associated with the Lie pair $(T_M,T_F)$ encoding a foliation $F$
on a smooth manifold $M$ may be thought of as the space
of formal polyvector fields on the leaf space of the foliation
\cite{MR3277952, MR3300319}, or more precisely,
on the differentiable stack \cite{MR2817778}
presented by the holonomy groupoid of the foliation $F$.
Similarly, denoting the associative algebra of differential operators on $M$ by $\DD(M)$,
the cochain complex
$\Big(\tot\Omega^{\bullet}_F\big(\bigotimes^\bullet_R
\big(\frac{\DD(M)}{\DD(M)\cdot\sections{T_F}}\big)\big),\dhoch\Big)$
may be thought of as the space of formal polydifferential operators
on the leaf space of the foliation, or more precisely,
on the differentiable stack presented by the holonomy groupoid of the foliation $F$.
Unless the foliation $F$ admits a transversal foliation \cite{arXiv:2103.08096},
there are no obvious dgla (or $L_\infty$ algebra) structures on these cochain complexes.

On the other hand, both $\hypercohomology^\bullet_{\CE}(A,\Tpoly{\bullet})$
and $\hypercohomology^\bullet_{\CE}(A,\Dpoly{\bullet})$
admit obvious associative algebra structures ---
the multiplications in cohomology proceed
from the wedge product in $\Tpoly{\bullet}$
and the tensor product of left $R$-modules in $\Dpoly{\bullet}$.

We are thus naturally led to the following central twofold question:

\begin{introques*}
\strut
\begin{enumerate}
\item Do the cohomology groups $\hypercohomology^\bullet_{\CE}(A,\Tpoly{\bullet})$
and $\hypercohomology^\bullet_{\CE}(A,\Dpoly{\bullet})$
admit canonical Gerstenhaber algebra structures?
\item Do the two cochain complexes
\[ \big(\tot\big(\wa\otimes_R\Tpoly{\bullet}\big),\dbott\big)
\quad\text{and}\quad
\big(\tot\big(\wa\otimes_R\Dpoly{\bullet}\big),\dhoch\big) \]
associated with a Lie pair $(L,A)$ admit $L_\infty$ algebra structures
compatible ``in a certain sense'' with their respective associative multiplications?
If so, are these $L_\infty$ structures canonical?
\end{enumerate}
\end{introques*}

To answer this question, we introduce the notion of Fedosov dg Lie algebroid,
we establish a pair of contractions, and we apply the homotopy transfer theorem
of $L_\infty$ algebras \cite{MR3276839,MR1932522,MR1950958}
(see also \cite{arXiv:1705.02880,MR2361936,arXiv:1807.03086, MR3318161,MR3323983}).
Roughly speaking, given a Lie pair $(L,A)$,
we construct a geometric object called Fedosov dg Lie algebroid,
which engenders a pair of natural dgla's whose respective cohomologies
carry natural Gerstenhaber algebra structures.
The pair of cochain complexes underlying these engendered dgla's are homotopy equivalent
(in a style reminiscent of Dolgushev's Fedosov resolutions \cite{MR2102846})
to the cochain complexes \eqref{ginger} associated with the Lie pair $(L,A)$.
The latter complexes then inherit $L_\infty$ structures by homotopy transfer.

Hereunder, we proceed to give a more detailed outline of the construction.

Given a Lie pair $(L,A)$ and having chosen some additional geometric data,
one can endow the graded manifold $\cM=L[1]\oplus L/A$ with a homological vector
field $Q$ encoding the formal geometry of the Lie pair.
The resulting dg manifold $(\cM,Q)$ is called a Fedosov dg manifold \cite{MR4150934}.
It turns out that there exists a natural dg integrable distribution
$\cF\subset T_{\cM}$ on $(\cM,Q)$.
In other words, the tangent dg Lie algebroid $T_{\cM}\to\cM$
arising from the Fedosov dg manifold $(\cM,Q)$
admits a natural dg Lie subalgebroid $\cF\to\cM$.
We call this dg Lie algebroid $\cF\to\cM$ a Fedosov dg Lie algebroid.

Lie algebroids being generalizations of tangent bundles,
the notions of polyvector fields and of polydifferential operators admit generalizations
to the broader context of Lie algebroids.
The spaces of (generalized) polyvectors fields and of (generalized) polydifferential operators
each admit a natural dgla structure and the cohomology of this dgla
is in fact a Gerstenhaber algebra \cite{MR1675117,MR1815717}.
The notions of polyvector fields and of polydifferential operators
can be extended further in an appropriate sense to the context of \emph{dg} Lie algebroids.
This yields again a pair of dgla's whose cohomologies are Gerstenhaber algebras.

More precisely, in the context of a dg Lie algebroid $\cL\to\cM$,
a $k$-vector field is a section of the vector bundle $\Lambda^k\cL\to\cM$
while a $k$-differential operator is an element of
$\big(\suspended\enveloping{\cL}\big)^{\otimes k}$,
the tensor product (as left $C^\infty (\cM )$-modules)
of $k$ copies of the suspended universal enveloping algebra
$\suspended\enveloping{\cL}$.

It is clear that
the differential $\cQ:\sections{\cL}\to\sections{\cL}$,
the homological vector field $Q:C^\infty(\cM)\to C^\infty(\cM)$,
and the Lie bracket on $\sections{\cL}$ encoding the dg Lie algebroid structure
of $\cL\to\cM$ extend naturally to a degree $(+1)$ differential
$\cQ: \sections{\Lambda^{k+1}\cL} \to \sections{\Lambda^{k+1}\cL}$
and a Schouten bracket
$\schouten{\argument}{\argument}:\sections{\Lambda^{u+1}\cL}\otimes
\sections{\Lambda^{v+1}\cL}\to\sections{\Lambda^{u+v+1}\cL}$
--- see Section~\ref{Najaf} for more details.
The resulting triple $\big(\tot_\oplus\sections{\Lambda^{\bullet+1}\cL},
\cQ,\schouten{\argument}{\argument}\big)$ is a dgla.

The universal enveloping algebra of a \emph{dg} Lie algebroid $\cL\to\cM$,
which is defined by adapting the construction of the universal enveloping algebra
of a Lie algebroid, is a dg Hopf algebroid $\enveloping{\cL}$ over the dgca $R=C^\infty(\cM)$.
For each $k\geqslant 0$, the dg structure on the dg Lie algebroid $\cL\to\cM$
determines a differential
$\cQ:\big(\suspended\enveloping{\cL}\big)^{\otimes k+1}
\to\big(\suspended\enveloping{\cL}\big)^{\otimes k+1}$
of degree $(+1)$.
A Hochschild coboundary differential
$\hochschild:\big(\suspended\enveloping{\cL}\big)^{\otimes k}
\to\big(\suspended\enveloping{\cL}\big)^{\otimes k+1}$
and a Gerstenhaber bracket
$\gerstenhaber{\argument}{\argument}:
\big(\suspended\enveloping{\cL}\big)^{\otimes u+1}
\otimes\big(\suspended\enveloping{\cL}\big)^{\otimes v+1}
\to\big(\suspended\enveloping{\cL}\big)^{\otimes u+v}$
can be defined explicitly in terms of the dg Hopf algebroid structure.
The resulting triple
$\big(\tot_\oplus\big(\suspended\enveloping{\cL}\big)^{\otimes \bullet+1},
\cQ+\hochschild,\gerstenhaber{\argument}{\argument}\big)$ is a dgla.

The ``polyvector fields'' and ``polydifferential operators''
associated with a \emph{Fedosov} dg Lie algebroid $\cF\to\cM$
may be viewed geometrically as polyvector fields and polydifferential operators
tangent to the dg foliation $\cF$ on the Fedosov dg manifold $(\cM,Q)$.
In fact, one can identify the ``polyvector fields''
$\big(\tot_\oplus\sections{\Lambda^{\bullet+1}\cF},\cQ\big)$
and ``polydifferential operators''
$\big(\tot_\oplus\big(\suspended\enveloping{\cF}\big)^{\otimes\bullet+1},
\cQ+\hochschild\big)$
associated with $\cF\to\cM$ to a pair of cochain complexes
\begin{equation}\label{drizzle}
\big(\tot(\sections{\Lambda^\bullet L\dual}\otimes_R\verticalTpoly{\bullet}),
\lie{Q}\big) \quad\text{and}\quad
\big(\tot(\sections{\Lambda^\bullet L\dual}\otimes_R\verticalDpoly{\bullet}),
\gerstenhaber{Q+m}{\argument}\Big)
,\end{equation}
where $\verticalTpoly{\bullet}$ denotes the formal polyvector fields and $\verticalDpoly{\bullet}$
the formal polydifferential operators tangent to the fibers of the vector bundle $L/A\to M$.

The next step and key ingredient of the construction consists
in establishing the following pair of contractions of Dolgushev--Fedosov type:
\begin{equation}\label{eq:A}
\begin{tikzcd}[cramped]
\Big(\tot\big(\wa\otimes_R
\Tpoly{\bullet}\big)
,d_A^{\Bott}\Big)
\arrow[r, " ", shift left] &
\Big(\tot\big(\sections{\Lambda^\bullet L\dual}\otimes_R
\verticalTpoly{\bullet}\big), \lie{Q}\Big)
\arrow[l, " ", shift left]
\arrow[loop, " ",out=5,in=-5,looseness = 3]
\end{tikzcd},
\end{equation}
and
\begin{equation}
\label{eq:B}
\begin{tikzcd}[cramped]
\Big(\tot\big(\wa\otimes_R\Dpoly{\bullet}\big),
\dhoch \Big)
\arrow[r, shift left, " "] &
\Big(\tot\big(\sections{\Lambda^\bullet L\dual}\otimes_R\verticalDpoly{\bullet}\big),
\gerstenhaber{Q+ m}{\argument}\Big)
\arrow[l, shift left, " "]
\arrow[loop, " ", out=5,in=-5,looseness = 3]
\end{tikzcd}.
\end{equation}

Finally, we use the homotopy transfer theorem for $L_\infty$ algebras
\cite{MR3276839,MR1932522,MR1950958}
--- see also \cite{arXiv:1705.02880,MR2361936,arXiv:1807.03086, MR3318161,MR3323983} ---
to push the $L_\infty$ structures carried by the complexes \eqref{drizzle}
(the r.h.s.\ of the contractions \eqref{eq:A} and~\eqref{eq:B})
to the complexes \eqref{ginger} (the l.h.s.\ of the contractions \eqref{eq:A} and~\eqref{eq:B}).
Furthermore, we prove that the resulting $L_\infty$ algebra structures on the complexes \eqref{ginger}
are unique up to $L_\infty$ isomorphisms having the identity map as linear part
and are therefore (essentially) independent of the choice of geometric data
made in the construction of the Fedosov dg Lie algebroid.
Moreover, we prove that the two cochain maps in the above contractions
\eqref{eq:A} and~\eqref{eq:B}
are compatible with the associative algebra structures
given by the wedge and cup products respectively.

Finally, combining these facts, we are able to prove the following theorem,
which is the main result of the paper.

\begin{introthm}\label{Macau}
Let $(L, A)$ be a Lie pair.
\begin{itemize}
\item[(1)] The cohomology groups $\hypercohomology^\bullet_{\CE}(A,\Tpoly{\bullet})$
and $\hypercohomology^\bullet_{\CE}(A,\Dpoly{\bullet})$
admit canonical Gerstenhaber algebra structures.
\item[(2a)] The spaces $\tot\big(\wa\otimes_R\Tpoly{\bullet}\big)$
and $\tot\big(\wa\otimes_R\Dpoly{\bullet}\big)$
admit $L_\infty$ algebra structures with the operators $d_A^{\Bott}$ and $\dhoch$
as their respective unary brackets.
\item[(2b)] These $L_\infty$ algebra structures are unique up to $L_\infty$ isomorphisms
having the identity map as linear part.
\item[(2c)] The binary brackets are compatible with the associative products
(viz.\ the wedge product and the cup product respectively)
in the sense that the graded Leibniz rule holds up to homotopy.
\end{itemize}
\end{introthm}

The above theorem is a synthesis of Propositions~\ref{Bari}, \ref{thm:Naples}
and~\ref{prop:uniqueness} from this paper.
We remark that, in Theorem~\ref{Macau}~(2c),
we only claim what is needed to ensure that the resulting cohomology groups
are Gerstenhaber algebras, but in fact the $L_\infty$
and associative algebra structures should be compatible in a much stronger and refined sense.
More precisely, the space of polyvector fields
$\tot\big(\wa\otimes_R\Tpoly{\bullet}\big)$
and that of polydifferential operators
$\tot\big(\wa\otimes_R\Dpoly{\bullet}\big)$
should both carry much richer algebraic structures,
such as the $\mathsf{Ger}_\infty$ algebras investigated by Tamarkin \cite{MR2699544}
or the $\mathsf{Br}_\infty$ algebras studied by Willwacher~\cite{MR3522653}.
In fact, it should again be possible to construct such structures explicitly
via homotopy transfer along Dolgushev--Fedosov contractions.
We will return to this issue in a forthcoming work.

When the Lie algebroid $L$ arises as the matched sum $A\bowtie B$
of a matched pair $(A, B)$ of Lie algebroids,
i.e.\ when the short exact sequence $0\to A \to L \to L/A \to 0$
admits a splitting $j: L/A\to L$ whose image $B:=j(L/A)$ is a Lie subalgebroid of $L$,
the $L_\infty$ algebra structures
on $\tot\big(\wa\otimes_R\Tpoly{\bullet}\big)$
and $\tot\big(\wa\otimes_R\Dpoly{\bullet}\big)$
in Theorem~\ref{Macau} turn out to be dgla's and admit a much simpler description
than in the case of a generic Lie pair.
Indeed, in the case of a matched pair, the dg manifold $(A[1]\oplus B,\dabott)$
is a dg Lie algebroid over the dg manifold $(A[1],\dace)$
whose associated cochain complexes of polyvector fields
and polydifferential operators are isomorphic to
$\big(\tot\sections{\Lambda^\bullet A\dual\otimes\Lambda^{\bullet+1}B},
\dabott\big)$ and
$\big(\tot\big(\wa\otimes_R\enveloping{B}^{\otimes\bullet+1}\big),\dhoch\big)$, respectively,
and are therefore naturally dgla's when endowed with the usual Schouten bracket
and the usual Gerstenhaber bracket, respectively.

\begin{introthm}
\label{thm:A}
If, in a Lie pair $(L,A)$, the Lie algebroid $L$ arises
as the matched sum $A\bowtie B$ of a matched pair $(A, B)$ of Lie algebroids
--- i.e.\ the short exact sequence $0\to A \to L \to L/A \to 0$ admits a splitting $j: L/A\to L$
whose image $B:=j(L/A)$ is a Lie subalgebroid of $L$ ---
then the $L_\infty$ algebra structures of Theorem~\ref{Macau} on
\[ \tot\big(\wa\otimes_R\Tpoly{\bullet}\big)
\quad\text{and}\quad \tot\big(\wa\otimes_R\Dpoly{\bullet}\big) \]
are actually dgla's and are respectively isomorphic to
\[ \big(\tot\sections{\Lambda^\bullet A\dual\otimes\Lambda^{\bullet+1}B},
\dabott,\schouten{\argument}{\argument}\big) \quad\text{and}\quad
\big(\tot\big(\wa\otimes_R\enveloping{B}^{\otimes\bullet+1}\big),
\dhoch,\gerstenhaber{\argument}{\argument}\big) ,\]
the dgla's of polyvector fields and of polydifferential operators
arising from the dg Lie algebroid $A[1]\oplus B\to A[1]$.
The isomorphisms are canonical.
Furthermore, the Gerstenhaber algebra structures on the corresponding cohomology groups
\[ \hypercohomology^\bullet_{\CE}(A,\Tpoly{\bullet})
\quad\text{and}\quad \hypercohomology^\bullet_{\CE}(A,\Dpoly{\bullet}) \]
are isomorphic to the canonical Gerstenhaber algebra structures on
\[ \hypercohomology^\bullet_{\CE}(A,\Lambda^{\bullet+1}B) \quad\text{and}\quad
\hypercohomology^\bullet_{\CE}\big(A,{\enveloping{B}}^{\otimes\bullet+1}\big) ,\]
respectively.
\end{introthm}

Finally, let us recall that the well-known Hochschild--Kostant--Rosenberg map
for ordinary smooth manifolds
admits a natural generalization as a morphism from
$\tot\big(\wa\otimes_R\Tpoly{\bullet}\big)$
to $\tot\big(\wa\otimes_R\Dpoly{\bullet}\big)$,
which is still a quasi-isomorphism of cochain complexes
and thus induces, on the cohomology level, an isomorphism
from $\hypercohomology^\bullet_{\CE}(A,\Tpoly{\bullet})$
to $\hypercohomology^\bullet_{\CE}(A,\Dpoly{\bullet})$.
However, there is a significant difference compared to the case of ordinary smooth manifolds:
the Hochschild--Kostant--Rosenberg map for Lie pairs does \emph{not} in general respect
the Gerstenhaber algebra structures on cohomology.
Nervertheless, it is always possible to remedy this defect:
the Hochschild--Kostant--Rosenberg morphism must be twisted.
Doing so involves techniques developed by Kontsevich
in the proof of his formality theorem \cite{MR2062626} --- see also \cite{MR2699544}.
Indeed, the present paper provides the foundation
for an ulterior paper \cite{MR3964152}
establishing a formality theorem for Lie pairs
and an ensuing Kontsevich--Duflo type theorem describing the precise relationship
between the Gerstenhaber algebra structures on
$\hypercohomology^\bullet_{\CE}(A, \Tpoly{\bullet})$
and $\hypercohomology^\bullet_{\CE}(A,\Dpoly{\bullet})$
revealed in Theorem~\ref{Macau}.
\subsection*{Acknowledgements}

We would like to thank Martin Bordemann, Damien Broka, Zhuo Chen,
Olivier Elchinger, Vasiliy Dolgushev, Camille Laurent-Gengoux, Hsuan-Yi Liao,
Kirill Mackenzie, Rajan Mehta,
Jim Stasheff, Luca Vitagliano, and Yannick Voglaire for fruitful discussions
and useful comments.
We are grateful to an anonymous referee for many insightful comments and suggestions
which led to sensible improvements in the presentation of our results.
Stiénon is grateful to Université Paris~7 for its hospitality
during his sabbatical leave in 2015--2016.

\section{Polydifferential operators and polyvector fields for Lie pairs}
\label{Boise}

\subsection{Chevalley--Eilenberg cohomology}

Let $A\to M$ be a Lie algebroid.
The Chevalley--Eilenberg cohomology $\hypercohomology_{\CE}^k(A,\EE^\bullet)$
in degree $k$ of a complex of left $\enveloping{A}$-modules
\[ \begin{tikzcd}
0 \arrow{r} & \EE^{-1} \arrow{r}{d} & \EE^{0} \arrow{r}{d}
& \EE^{1} \arrow{r}{d} & \EE^{2} \arrow{r}{d} & \cdots
\end{tikzcd} \]
is the total cohomology in degree $k$ of the double complex
\[ \begin{tikzcd}[row sep=small]
\vdots & \vdots & \vdots & \\
\sections{\Lambda^0 A\dual}\otimes_R\EE^{1} \arrow{u}{\id\otimes d} \arrow{r}{d_A^{\EE}} &
\sections{\Lambda^1 A\dual}\otimes_R\EE^{1} \arrow{u}{-\id\otimes d} \arrow{r}{d_A^{\EE}} &
\sections{\Lambda^2 A\dual}\otimes_R\EE^{1} \arrow{u}{\id\otimes d} \arrow{r}{d_A^{\EE}} & \cdots \\
\sections{\Lambda^0 A\dual}\otimes_R\EE^{0} \arrow{u}{\id\otimes d} \arrow{r}{d_A^{\EE}} &
\sections{\Lambda^1 A\dual}\otimes_R\EE^{0} \arrow{u}{-\id\otimes d} \arrow{r}{d_A^{\EE}} &
\sections{\Lambda^2 A\dual}\otimes_R\EE^{0} \arrow{u}{\id\otimes d} \arrow{r}{d_A^{\EE}} & \cdots \\
\sections{\Lambda^0 A\dual}\otimes_R\EE^{-1} \arrow{u}{\id\otimes d} \arrow{r}{d_A^{\EE}} &
\sections{\Lambda^1 A\dual}\otimes_R\EE^{-1} \arrow{u}{-\id\otimes d} \arrow{r}{d_A^{\EE}} &
\sections{\Lambda^2 A\dual}\otimes_R\EE^{-1} \arrow{u}{\id\otimes d} \arrow{r}{d_A^{\EE}} & \cdots
\end{tikzcd} \]

When we say that the above diagram is a double complex,
we mean in particular that each square of the grid commutes.
Hence the total cohomology is the cohomology of the complex
\[ \left(\bigoplus_{p+q=\bullet} \sections{\Lambda^p A\dual}
\otimes_R\EE^{q},d_A^{\EE}+ \id\otimes d\right) .\]
Recall that, the degree of the operator $d$ being $+1$, the usual sign convention
for the \emph{tensor product} of linear maps in the presence of gradings dictates that
\begin{equation}\label{eq:Rome}
\big(\id\otimes d\big)(\omega\otimes e)=(-1)^p\omega\otimes d(e),
\quad\forall\omega\in\sections{\Lambda^p A\dual},\ \forall e\in\EE^\bullet
.\end{equation}

\subsection{Polydifferential operators}

Given a Lie pair $(L,A)$, let $\Dpoly{-1}$ denote the algebra $R$
of smooth functions on the manifold $M$,
let $\Dpoly{0}$ denote the left $\enveloping{A}$-module
$\frac{\enveloping{L}}{\enveloping{L}\sections{A}}$,
let $\Dpoly{k}$ denote the tensor product
$\Dpoly{0}\otimes_R\cdots\otimes_R\Dpoly{0}$ of $(k+1)$ copies of
the left $R$-module $\Dpoly{0}$, and set
$\Dpoly{\bullet}=\bigoplus_{k=-1}^{\infty}\Dpoly{k}$.
Since $\Dpoly{0}$ is a left $\enveloping{A}$-module
and $\enveloping{A}$, as a Hopf algebroid, is endowed with a
comultiplication, $\Dpoly{k}$ is also naturally a left $\enveloping{A}$-module
for each $k\geqslant -1$ \cite{MR1815717}.

Furthermore, the comultiplication $\Delta:\enveloping{L}\to\enveloping{L}\otimes_R\enveloping{L}$
on the universal enveloping algebra $\enveloping{L}$
induces a comultiplication
\[ \Delta:\Dpoly{0}\to\Dpoly{0}\otimes_R\Dpoly{0} \]
since
\[ \Delta\big(\enveloping{L}\sections{A}\big)\subseteq \enveloping{L}\otimes_R
\big(\enveloping{L}\sections{A}\big) + \big(\enveloping{L}\sections{A}\big)
\otimes_R \enveloping{L} \]
--- see \cite[Sections 2.2 and 2.3]{arXiv:1408.2903}.

\begin{lemma}[\cite{arXiv:1408.2903}]
The $\enveloping{A}$-module $\Dpoly{0}$ is a cocommutative coassociative coalgebra over $R$
whose comultiplication $\Delta:\Dpoly{0}\to\Dpoly{0}\otimes_R\Dpoly{0}$ is
a morphism of $\enveloping{A}$-modules.
\end{lemma}

Following \cite[Equation (98)]{MR1815717}, introduce
the Hochschild differential $\hochschild:\Dpoly{k-1}\to\Dpoly{k}$
defined by
\begin{multline*}
\hochschild(u_1\otimes\cdots\otimes u_k) = 1\otimes u_1\otimes\cdots\otimes
u_k + \sum_{i=1}^{k} (-1)^i u_1\otimes\cdots\otimes u_{i-1}\otimes
\Delta(u_i) \otimes
u_{i+1}\otimes\cdots\otimes u_k \\
+ (-1)^{k+1} u_1\otimes\cdots\otimes u_k\otimes 1.
\end{multline*}
Since the comultiplication
$\Delta:\Dpoly{0}\to\Dpoly{0}\otimes_R\Dpoly{0}$ is cocommutative and
coassociative, $\hochschild$ is a coboundary operator,
i.e.\ $\hochschild^2=0$.
Moreover, since the comultiplication $\Delta$ is
a morphism of $\enveloping{A}$-modules,
$\hochschild:\Dpoly{k-1}\to\Dpoly{k}$
is a morphism of $\enveloping{A}$-modules.
Therefore, the Hochschild complex
\begin{equation*}\begin{tikzcd}
0 \arrow{r} & \Dpoly{-1} \arrow{r}{\hochschild} & \Dpoly{0} \arrow{r}{\hochschild} &
\Dpoly{1} \arrow{r}{\hochschild} & \Dpoly{2} \arrow{r}{\hochschild} & \cdots
\end{tikzcd}\end{equation*}
is a complex of $\enveloping{A}$-modules.

The Chevalley--Eilenberg cohomology $\hypercohomology^k_{\CE}(A,\Dpoly{\bullet})$
in degree $k$ of the Hochschild complex of the pair $(L,A)$
is the degree $k$ total cohomology of the double complex
\[ \begin{tikzcd}[row sep=small]
\vdots & \vdots & \vdots & \\
\sections{\Lambda^0 A\dual}\otimes_R\Dpoly{1} \arrow{u}{\id\otimes\hochschild} \arrow{r}{\dau} &
\sections{\Lambda^1 A\dual}\otimes_R\Dpoly{1} \arrow{u}{-\id\otimes\hochschild} \arrow{r}{\dau} &
\sections{\Lambda^2 A\dual}\otimes_R\Dpoly{1} \arrow{u}{\id\otimes\hochschild} \arrow{r}{\dau} &
\cdots \\
\sections{\Lambda^0 A\dual}\otimes_R\Dpoly{0} \arrow{u}{\id\otimes\hochschild} \arrow{r}{\dau} &
\sections{\Lambda^1 A\dual}\otimes_R\Dpoly{0} \arrow{u}{-\id\otimes\hochschild} \arrow{r}{\dau} &
\sections{\Lambda^2 A\dual}\otimes_R\Dpoly{0} \arrow{u}{\id\otimes\hochschild} \arrow{r}{\dau} &
\cdots \\
\sections{\Lambda^0 A\dual}\otimes_R\Dpoly{-1} \arrow{u}{\id\otimes\hochschild} \arrow{r}{\dau} &
\sections{\Lambda^1 A\dual}\otimes_R\Dpoly{-1} \arrow{u}{-\id\otimes\hochschild} \arrow{r}{\dau} &
\sections{\Lambda^2 A\dual}\otimes_R\Dpoly{-1} \arrow{u}{\id\otimes\hochschild} \arrow{r}{\dau} &
\cdots
\end{tikzcd} \]

The coboundary operator
$\dau:\sections{\Lambda^p A\dual}\otimes_R
\Dpoly{q}\to\sections{\Lambda^{p+1} A\dual}\otimes_R \Dpoly{q}$ is defined by
\[ \begin{split} \dau(\omega\otimes u_0\otimes\cdots\otimes u_q)
=&\ (d_A\omega)\otimes u_0\otimes\cdots\otimes u_q \\
&\ +\sum_{j=1}^{\rk(A)}\sum_{k=0}^{q} (\alpha_j\wedge\omega)
\otimes u_0\otimes \cdots\otimes u_{k-1} \otimes
a_j\cdot u_k\otimes u_{k+1} \otimes \cdots\otimes u_q
,\end{split} \]
for all $\omega\in\sections{\Lambda^p A\dual}$ and $u_0,u_1,\dots,u_q\in\Dpoly{0}$
--- for $q=-1$, we simply have $\dau=d_A$.
Here $(a_i)_{i\in\{1,\dots,r\}}$ is any local frame of $A$
and $(\alpha_j)_{j\in\{1,\dots,r\}}$ is the dual local frame of $A\dual$.
In other words, $\hypercohomology^k_{\CE}(A,\Dpoly{\bullet})$
is the cohomology of the total complex
\[ \big(\tot(\sections{\Lambda^{\bullet}A\dual}\otimes_R\Dpoly{\bullet}),\dhoch\big) ,\]
where we use the abbreviated symbol $\mathfrak{d}_{\mathscr{H}}$
to denote the operator $\id\otimes\hochschild$.
See Equation~\eqref{eq:Rome} for the sign convention
used in the definition of the map $\id\otimes\hochschild$.

However, unlike the universal enveloping algebra of a Lie algebroid,
$\Dpoly{0}$ is in general not a Hopf algebroid over $R$
--- in fact, $\Dpoly{0}$ is not even an associative algebra.
Therefore, a priori, the Hochschild cohomology is only a vector space.

\begin{remark}
In general
$\Dpoly{0}=\frac{\enveloping{L}}{\enveloping{L}\sections{A}}$
does not admit an associative product.
For a Lie pair $(T_M,T_F)$ encoding a foliation $F$
on a smooth manifold $M$, Vitagliano proved that
$\sections{\Lambda^{\bullet}A\dual}\otimes_R\Dpoly{0}$
can be thought of as the space of normal differential operators
of the foliation $F$ and admits an $A_\infty$-algebra structure
\cite{MR3313214}.
For a generic Lie pair $(L,A)$,
the existence of an $A_\infty$-algebra structure
on $\sections{\Lambda^{\bullet}A\dual}\otimes_R\Dpoly{0}$
was proved in~\cite{VitaglianoStienonXu}.
\end{remark}
There is a natural cup product
\begin{equation}\label{eq:cup01}
\big(\sections{\Lambda^{k}A\dual}\otimes_R\Dpoly{p}\big)
\otimes\big(\sections{\Lambda^{l}A\dual}\otimes_R\Dpoly{q}\big)
\xto{\smile}\sections{\Lambda^{k+l}A\dual}\otimes_R\Dpoly{p+q+1}
\end{equation}
on $\tot\big(\sections{\Lambda^{\bullet}A\dual}\otimes_R\Dpoly{\bullet}\big)$
defined by
\[ (\omega\otimes u)\smile(\theta\otimes v)
=(-1)^{l(p+1)}(\omega\wedge\theta)\otimes(u\otimes v) \]
for all $\omega\in\sections{\Lambda^{k}A\dual}$,
$\theta\in\sections{\Lambda^{l}A\dual}$,
$u\in\Dpoly{p}$ and $v\in\Dpoly{q}$.

The following proposition is easily verified.

\begin{lemma}
For any Lie pair $(L, A)$, the cochain complex
$\big(\tot(\sections{\Lambda^{\bullet}A\dual}\otimes_R\Dpoly{\bullet}),\dhoch\big)$,
equipped with the cup product \eqref{eq:cup01},
is a dg associative algebra.
Therefore, there is an induced associative algebra structure on the Hochschild cohomology
$\hypercohomology^\bullet_{\CE}(A,\Dpoly{\bullet})$.
\end{lemma}

\begin{remark}\label{graded_commutativity}
It is natural to expect that the induced associative product
on $\hypercohomology^\bullet_{\CE}(A,\Dpoly{\bullet})$ is
graded commutative, as in the case of the usual Hochschild cohomology
$H^\bullet(\Dpoly{\bullet}(M),\hochschild)$
associated to a smooth manifold $M$.
The group $H^\bullet(\Dpoly{\bullet}(M),\hochschild)$
is the cohomology of a subcomplex of the Hochschild cochain complex
$C^\bullet(C^\infty(M),C^\infty(M))$.
Its cup product is graded commutative up to a homotopy
given by the Gerstenhaber pre-Lie bracket
\cite[Theorem~3]{MR0161898}.
This pre-Lie bracket can be defined in terms of the comultiplication and
the multiplication on $\Dpoly{0}(M)$: see Equation~\eqref{eq:pre-Lie}
for the formula in a similar situation.
However, this approach does not work for a Lie pair $(L, A)$, since
$\Dpoly{0}=\frac{\enveloping{L}}{\enveloping{L}\sections{A}}$
does not admit an associative multiplication,
thus the usual proof for graded commutativity does not extend
to our situation. In what follows, we will get around this
difficulty by establishing an isomorphism of associative algebras
between $\hypercohomology^\bullet_{\CE}(A,\Dpoly{\bullet})$
and the Hochschild cohomology of the Fedosov dg Lie algebroid
of the Lie pair $(L,A)$
--- see Proposition~\ref{Dhaka} and Proposition~\ref{thm:Naples}.
The graded commutativity of the latter can be proved the usual way
--- see Proposition~\ref{pro:hongkong1}.
\end{remark}

\subsection{Polyvector fields}

Likewise, given a Lie pair $(L,A)$, let $\Tpoly{-1}$ denote the algebra $R$
of smooth functions on the manifold $M$
and set $\Tpoly{k}=\sections{\Lambda^{k+1}(L/A)}$ for $k\geqslant 0$.
Consider $\Tpoly{\bullet}=\bigoplus_{k=-1}^\infty\Tpoly{k}$ as a complex
of $\enveloping{A}$-modules with trivial differential:
\begin{equation*} \begin{tikzcd}
0 \arrow{r} & \Tpoly{-1} \arrow{r}{0} &
\Tpoly{0} \arrow{r}{0} & \Tpoly{1} \arrow{r}{0} &
\Tpoly{2} \arrow{r}{0} & \cdots
\end{tikzcd} \end{equation*}
Its Chevalley--Eilenberg cohomology $\hypercohomology^k_{\CE}(A,\Tpoly{\bullet})$
in degree $k$
is the degree $k$ total cohomology of the double complex
\[ \begin{tikzcd}[row sep=small]
\vdots & \vdots & \vdots & \\
\sections{\Lambda^0 A\dual}\otimes_R\Tpoly{1} \arrow{u}{0} \arrow{r}{d_A^{\Bott}} &
\sections{\Lambda^1 A\dual}\otimes_R\Tpoly{1} \arrow{u}{0} \arrow{r}{d_A^{\Bott}} &
\sections{\Lambda^2 A\dual}\otimes_R\Tpoly{1} \arrow{u}{0} \arrow{r}{d_A^{\Bott}} & \cdots \\
\sections{\Lambda^0 A\dual}\otimes_R\Tpoly{0} \arrow{u}{0} \arrow{r}{d_A^{\Bott}} &
\sections{\Lambda^1 A\dual}\otimes_R\Tpoly{0} \arrow{u}{0} \arrow{r}{d_A^{\Bott}} &
\sections{\Lambda^2 A\dual}\otimes_R\Tpoly{0} \arrow{u}{0} \arrow{r}{d_A^{\Bott}} & \cdots \\
\sections{\Lambda^0 A\dual}\otimes_R\Tpoly{-1} \arrow{u}{0} \arrow{r}{d_A^{\Bott}} &
\sections{\Lambda^1 A\dual}\otimes_R\Tpoly{-1} \arrow{u}{0} \arrow{r}{d_A^{\Bott}} &
\sections{\Lambda^2 A\dual}\otimes_R\Tpoly{-1} \arrow{u}{0} \arrow{r}{d_A^{\Bott}} & \cdots
\end{tikzcd} \]

The coboundary operator
$d_A^{\Bott}:\sections{\Lambda^p A\dual}\otimes
\Tpoly{q}\to\sections{\Lambda^{p+1} A\dual}\otimes
\Tpoly{q}$ is defined by
\[ \begin{split} d_A^{\Bott}(\omega\otimes b_0\wedge\cdots\wedge b_q)
=&\ (d_A\omega)\otimes b_0\wedge\cdots\wedge b_q \\
&\ +\sum_{j=1}^{\rk(A)}\sum_{k=0}^{q} (\alpha_j\wedge\omega)
\otimes b_0\wedge \cdots\wedge b_{k-1} \wedge
\nabla^{\Bott}_{a_j} b_k\wedge b_{k+1} \wedge\cdots\wedge b_q
,\end{split} \]
for all $\omega\in\sections{\Lambda^p A\dual}$ and $b_0,b_1,\dots,b_q\in\sections{L/A}$.
Here $(a_i)_{i\in\{1,\dots,r\}}$ is any local frame of $A$
and $(\alpha_j)_{j\in\{1,\dots,r\}}$ is the dual local frame of $A\dual$.

There is a natural wedge product
\begin{equation}\label{eq:wedge01}
\big(\sections{\Lambda^{k}A\dual}\otimes_R\Tpoly{p}\big)
\otimes\big(\sections{\Lambda^{l}A\dual}\otimes_R\Tpoly{q}\big)
\xto{\wedge}\sections{\Lambda^{k+l}A\dual}\otimes_R\Tpoly{p+q+1}
\end{equation}
on $\tot\big(\sections{\Lambda^{\bullet}A\dual}\otimes_R\Tpoly{\bullet}\big)$
defined by
\begin{equation}\label{eq:wedge1}
(\omega\otimes u)\wedge(\theta\otimes v)
=(-1)^{l(p+1)}(\omega\wedge\theta)\otimes(u\otimes v)
\end{equation}
for all $\omega\in\sections{\Lambda^{k}A\dual}$,
$\theta\in\sections{\Lambda^{l}A\dual}$,
$u\in\Tpoly{p}$ and $v\in\Tpoly{q}$.

We have the following

\begin{lemma}
For any Lie pair $(L,A)$, the cochain complex
$\big(\tot(\sections{\Lambda^{\bullet}A\dual}\otimes_R\Tpoly{\bullet}),\dbott\big)$,
equipped with the wedge product \eqref{eq:wedge1}, is a dg commutative algebra.
Therefore, the cohomology
$\hypercohomology^\bullet_{\CE}(A,\Tpoly{\bullet})$
is a graded commutative algebra.
\end{lemma}

\section{Fedosov dg Lie algebroids}

\subsection{Dg Lie algebroids and polyvector fields and polydifferential operators}
\label{Najaf}

A $\ZZ$-graded manifold $\cM$ with base manifold $M$ is
a sheaf $\cA$ of $\ZZ$-graded commutative
$\mathcal{O}_M$-algebras over $M$
such that there exists a $\ZZ$-graded vector space $V$,
a covering of $M$ by open submanifolds $U\subset M$, and
a collection of isomorphisms of $C^\infty(U,\KK)$-algebras
\[ {\cA}|_U\cong C^\infty(U,\KK)\otimes_\KK\widehat{S}(V\dual) ,\]
where $\widehat{S}(V\dual)$ denotes the $\KK$-algebra
of formal power series on $V$. Here $\mathcal{O}_M$ denotes
the sheaf of $\KK$-valued $C^\infty$ functions over $M$.
By $C^\infty(\cM)$, we denote the $\ZZ$-graded commutative algebra
$\sections{M,\cA}$ of global sections of $(M,\cA)$.
By a dg manifold, we mean a $\ZZ$-graded manifold $\cM$
endowed with a homological vector field,
i.e.\ a derivation $Q$ of degree $+1$ of $C^\infty(\cM)$
satisfying $\bracket{Q}{Q}=0$.

\begin{example}\label{example:Vaintrob_dg_vs_Lie_abd}
Let $A\to M$ be a Lie algebroid over $\KK$. Then $A[1]$ is a dg manifold
with the Chevalley--Eilenberg differential $d_{\CE}$ as homological vector field.
According to Va\u{\i}ntrob~\cite{MR1480150},
there is a bijection between the Lie algebroid structures on the vector bundle $A\to M$
and the homological vector fields on the $\ZZ$-graded manifold $A[1]$.
\end{example}

\begin{example}
Let $\frakg=\bigoplus_{i\in\ZZ}\frakg_i$ be a $\ZZ$-graded
finite dimensional vector space.
Then the graded manifold $\frakg[1]$ is a dg manifold,
if and only if the graded vector space $\frakg$ admits a structure of curved
$L_\infty$ algebra.
\end{example}

Below we recall some basic notations regarding dg vector bundles.
For details, see~\cite{MR2709144, MR2534186,MR3293862, PolishSurvey}.
A dg vector bundle is a vector bundle object in the category of dg manifolds.
Consider a vector bundle object $\cE\xto{\pi}\cM$
in the category of $\ZZ$-graded manifolds. Its space of sections
$\sections{\cE}$ is defined to be the direct sum
$\bigoplus_{j\in\ZZ}\sections{\cE}^j$,
where $\sections{\cE}^j$ consists of the sections of degree $j$,
i.e.\ the maps $l\in\Hom(\cM,\cE[-j])$
such that $(\pi[-j])\circ l=\id_\cM$.
Here $\pi[-j]:\cE[-j]\to\cM$ is the natural map induced from $\pi$
--- see~\cite{MR2709144, MR2534186} for more details.

\begin{remark}\label{observation_above}
When $\cE\to\cM$ is a dg vector bundle, the homological
vector fields on $\cE$ and $\cM$ naturally induce
a degree $(+1)$ operator $\cQ$ on $\sections{\cE}$, making
$\sections{\cE}$ a dg module over $C^\infty(\cM)$.
Since the space $\sections{\cE\dual}$ of linear functions on $\cE$
and the pull-back of $C^\infty (\cM)$ via $\pi$ together generate
$C^\infty(\cE)$, the converse is also true --- see~\cite{MR3319134}.
\end{remark}

\begin{example}\label{exp:dgtangent1}
Let $(\cM,Q)$ be a dg manifold.
The space $\XX(\cM)$ of vector fields
on $\cM$ (i.e.\ graded derivations of $C^\infty(\cM)$),
which can be regarded as the space of sections $\sections{T_\cM}$,
is naturally a dg module over the dg algebra $(C^\infty(\cM),Q)$
with the Lie derivative $\lie{Q}:\XX(\cM)\to\XX(\cM)$
playing the role of the operator $\cQ$.
As a consequence, $T_\cM $ is a dg manifold
--- the homological vector field on $T_\cM$ is called the
\emph{complete lift} of $Q$ as well as tangent lift
in~\cite{MR3319134} --- and $T_\cM\to\cM$ is a dg vector bundle.
\end{example}

The following lemma is standard \cite{MR3319134}.
\begin{lemma}\label{lem:dgvector}
Assume $\cE$ is a dg vector bundle over the dg manifold $(\cM,Q)$.
\begin{enumerate}
\item Then the dual bundle $\cE\dual$ is a dg vector bundle over
$(\cM,Q)$.
\item Furthermore, for all $k\geqslant 1$,
the exterior tensor power vector bundle $\Lambda^k\cE$
is a dg vector bundle over $(\cM,Q)$.
\end{enumerate}
\end{lemma}

Here and throughout the paper,
we use the shorthand notation $\Lambda^k\cE$ abusively
to actually denote $\big(S^k(\cE[-1])\big)[k]$.

\begin{proof}
By assumption, $\sections{\cE}$ is a dg module over
$\big(C^\infty(\cM),Q\big)$ with degree $(+1)$ differential
$\cQ:\sections{\cE}\to\sections{\cE}$.
Define a degree $(+1)$ operator
$\cQ:\sections{\cE\dual}\to\sections{\cE\dual}$ by
\[ \duality{\cQ(\xi)}{l}=
Q\duality{\xi}{l}-(-1)^{\degree{\xi}}\duality{\xi}{\cQ(l)} \]
for all homogeneous $\xi\in\sections{\cE\dual}$
and $l\in\sections{\cE}$.
It is simple to see that this operator makes
$\sections{\cE\dual}$ into a dg module
over $\big(C^\infty(\cM),Q\big)$.

Similarly, $\sections{\Lambda^{k}\cE}$ is a dg module
over $\big(C^\infty(\cM),Q\big)$ with the differential
$\cQ:\sections{\Lambda^{k}\cE}\to\sections{\Lambda^{k}\cE}$
of degree $(+1)$ defined by
\begin{equation}\label{eq:wedge}
\cQ(l_1\wedge\cdots\wedge l_k)
=\sum_{i=1}^k (-1)^{\degree{l_1}+\cdots\degree{l_{i-1}}} l_1\wedge\cdots\wedge
\cQ(l_i)\wedge\cdots\wedge l_k
\end{equation}
for all homogeneous $l_1,\dots,l_k\in\sections{\cE}$.

The conclusion thus follows.
\end{proof}

A dg Lie algebroid is a Lie algebroid object in the category of dg manifolds.
Equivalently, a dg Lie algebroid is a dg vector bundle $\cL\to\cM$
endowed with a $\ZZ$-graded Lie algebroid structure
satisfying the compatibility condition
\begin{equation}\label{eq:compatibility}
\schouten{d_{\cL}}{\cQ}=0,
\end{equation}
where $d_{\cL}$ is the Chevalley--Eilenberg differential
\begin{equation}\label{eq:CE}
d_{\cL}:\sections{\Lambda^\bullet\cL\dual}
\to\sections{\Lambda^{\bullet+1}\cL\dual}
\end{equation}
of the Lie algebroid $\cL\to\cM$,
$\cQ$ is the differential (of internal degree $(+1)$)
\begin{equation}\label{eq:intQ}
\cQ:\sections{\Lambda^\bullet\cL\dual}
\to\sections{\Lambda^\bullet\cL\dual}
\end{equation}
induced by the dg vector bundle structure on $\cL\to\cM$
(see Lemma~\ref{lem:dgvector}),
and $\schouten{\argument}{\argument}$ denotes the commutator.
For more details, we refer the reader to~\cite{MR2534186,MR2709144},
where dg Lie algebroids are called $Q$-algebroids.

\begin{example}\label{exp:dgtangent2}
As in Example~\ref{exp:dgtangent1}, let $(\cM,Q)$ be a dg manifold.
In addition to being a dg vector bundle,
$T_{\cM}\to\cM$ is also a Lie algebroid.
In this case, the Chevalley--Eilenberg differential \eqref{eq:CE}
is the de Rham differential
\[ \exterior:\Omega^\bullet(\cM)\to\Omega^{\bullet+1}(\cM) ,\]
while the internal differential \eqref{eq:intQ} is the Lie derivative
\[ \lie{Q} : \Omega^\bullet (\cM) \to \Omega^\bullet (\cM) .\]
Since $\schouten{\exterior}{\lie{Q}}=0$,
it follows that $T_{\cM}$ is indeed a dg Lie algebroid.
\end{example}

For an ordinary $\ZZ$-graded Lie algebroid,
one can speak about ``polyvector fields''
and ``polydifferential operators'' on the Lie algebroid.
For a dg Lie algebroid, the dg structure will induce degree (+1) differentials
on ``polyvector fields'' and ``polydifferential operators''.
For instance, the ``polyvector fields'' and ``polydifferential
operators'' for the tangent dg Lie algebroid $T_{\cM}$
of a dg manifold $(\cM,Q)$ as in Example~\ref{exp:dgtangent2}
are, respectively, the polyvector fields and the polydifferential operators on $\cM$,
while the induced degree (+1) differentials are
$\lie{Q}$ and $\gerstenhaber{Q}{\argument}$, respectively.
Here $\gerstenhaber{\argument}{\argument}$ stands for the Gerstenhaber
bracket on the polydifferential operators of $\cM$.

More precisely, a $k$-vector field on a dg Lie algebroid $\cL\to\cM$
is a section of the vector bundle $\Lambda^k\cL\to\cM$.
Since $\cL\to\cM$ is a dg vector bundle,
according to Lemma~\ref{lem:dgvector}, we have a degree $(+1)$ differential
$\cQ: \sections{\Lambda^{k+1}\cL} \to \sections{\Lambda^{k+1}\cL}$
--- see Equation~\eqref{eq:wedge}.
On the other hand, the Lie algebroid structure on $\cL$ yields a Schouten bracket
\[ \schouten{\argument}{\argument}: \sections{\Lambda^{u+1}\cL}\otimes
\sections{\Lambda^{v+1}\cL}\to \sections{\Lambda^{u+v+1}\cL} .\]

For $n\in\ZZ$, we set
\[ \tot_\oplus^n\sections{\Lambda^{\bullet+1}\cL}
=\bigoplus_{\substack{p+q=n\\ p,q\in\ZZ\\ q\geqslant -1}}
\big(\sections{\Lambda^{q+1}\cL}\big)^p ,\]
where $\big(\sections{\Lambda^{q+1}\cL}\big)^p$
denotes the subspace of $\sections{\Lambda^{q+1}\cL}$
consisting of homogeneous elements of degree $p+q$.

\begin{proposition}\label{pro:hongkong}
Let $\cL$ be a dg Lie algebroid over $\cM$.
\begin{enumerate}
\item When endowed with the differential $\cQ$,
the wedge product, and the Schouten bracket,
the space of `polyvector fields' $\tot_\oplus\sections{\Lambda^{\bullet+1}\cL}$
is a differential Gerstenhaber algebra\footnote{Here and in the sequel,
by a differential Gerstenhaber algebra, we mean a
Gerstenhaber algebra equipped with a degree $(+1)$ differential, which is
a derivation of both the associative multiplication and the Lie bracket.
Such structures were called strongly differential Gerstenhaber
algebras in~\cite{MR1675117,MR1361447}.} --- whence a dgla.
\item When endowed with the wedge product and the Schouten bracket,
the cohomology
$\hypercohomology^\bullet\big(\tot_\oplus\sections{\Lambda^{\bullet+1}\cL},\cQ)$
is a Gerstenhaber algebra.
\end{enumerate}
\end{proposition}

Similarly, a $k$-differential operator for a Lie algebroid $\cL$
is an element of $\suspended\enveloping{\cL}^{\otimes k}$,
the tensor product (as left $C^\infty(\cM)$-modules) of $k$ copies
of the suspended universal enveloping algebra $\suspended\enveloping{\cL}$.
Recall that the universal enveloping algebra $\enveloping{\cL}$
of a $\ZZ$-graded Lie algebroid $\cL\to\cM$ with anchor
$\anchor:\cL\to T_\cM$
is the quotient of the (reduced) tensor algebra
\begin{equation}\label{eq:tensor}
\bigoplus_{n=1}^{\infty}\big(\cR\oplus\sections{\cL}\big)^{\otimes n}
\end{equation}
of the $\KK$-module $\cR\oplus\sections{\cL}$ by the two-sided ideal
generated by the elements of the following four types:
\begin{align}
& X\otimes Y-(-1)^{\degree{X}\degree{Y}}Y\otimes X-\bracket{X}{Y} && f\otimes X-fX
\nonumber \\
& X\otimes g- (-1)^{\degree{g}\degree{X}}g\otimes X-\anchor_{X}(g) && f\otimes g-fg \label{eq:four}
\end{align}
for all homogeneous $X,Y\in\sections{\cL}$ and $f,g\in\cR$.
As earlier, the symbol $\cR$ denotes $C^\infty(\cM)$.

The universal enveloping algebra $\enveloping{\cL}$
is a coalgebra over $\cR$ \cite{MR1815717}.
Its comultiplication
\[ \Delta:\enveloping{\cL}\to\enveloping{\cL}\otimes_\cR\enveloping{\cL} \]
is an $R$-linear map of degree $0$ characterized by the identities
\begin{gather*}
\Delta(1)=1\otimes 1; \\
\Delta(b)=1\otimes b+b\otimes 1, \quad\forall b\in\sections{\cL}; \\
\Delta(u\cdot v)=\Delta(u)\cdot\Delta(v), \quad\forall u,v\in\enveloping{\cL} ,
\end{gather*}
where the symbol $\cdot$ denotes the multiplication in $\enveloping{\cL}$.
We refer the reader to~\cite{MR1815717} for the precise meaning of
(the r.h.s.\ of) the last equation above.
More explicitly, we have
\begin{multline}\label{eq:325}
\Delta(b_1 b_2\cdots b_n) =1\otimes(b_1 b_2\cdots b_n)
+\sum_{p=1}^{n-1}\sum_{\sigma\in\shuffle{p}{n-p}}
\pm (b_{\sigma(1)}\cdots b_{\sigma(p)})\otimes
(b_{\sigma(p+1)}\cdots b_{\sigma(n)}) \\
+(b_1 b_2\cdots b_n)\otimes 1
,\end{multline}
where $\pm$ denotes the Koszul sign of the $(p,n-p)$-shuffle\footnote{A
$(p,q)$-shuffle is a permutation $\sigma\in S_{p+q}$ of the set
$\{1,2,\cdots,p+q\}$ satisfying $\sigma(1)<\sigma(2)<\cdots<\sigma(p)$
and $\sigma(p+1)<\sigma(p+2)<\cdots<\sigma(p+q)$.
The subset of $S_{p+q}$ consisting of all $(p,q)$-shuffles is denoted $\shuffle{p}{q}$.}
$\sigma$ of the $n$-tuple of homogeneous elements
$b_1,\dots,b_n$ of $\sections{\cL}$.

Now assume that $\cL\to\cM$ is a dg Lie algebroid.
The differential $\cQ:\sections{\cL}\to\sections{\cL}$
and the homological vector field $Q:C^\infty(\cM)\to C^\infty(\cM)$
induce a differential of degree $(+1)$ on the (reduced) tensor algebra
\eqref{eq:tensor} by way of the Leibniz rule. From the compatibility
condition \eqref{eq:compatibility}, it is simple to see
that the two-sided ideal generated by the elements
\eqref{eq:four} is stable under this induced differential
on the universal enveloping algebra
\[ \cQ:\enveloping{\cL}\to\enveloping{\cL} ,\]
which we denote by the same symbol $\cQ$ by abuse of
notation. This differential
is compatible with both the algebra and coalgebra structures
on $\enveloping{\cL}$ so that $\enveloping{\cL}$ is a dg Hopf
algebroid over the dga $\cR=C^\infty(\cM)$.
As a consequence, we obtain a differential
$\cQ:\suspended\enveloping{\cL}^{\otimes k+1}
\to\suspended\enveloping{\cL}^{\otimes k+1}$
of degree $(+1)$ for each $k\geqslant -1$.
Here $\suspended\enveloping{\cL}^{\otimes 0}=\suspended\cR$ and
$\suspended\enveloping{\cL}^{\otimes k+1}$ (with $k\geqslant 0$)
denotes the tensor product
$\suspended\enveloping{\cL}\otimes_\cR\cdots\otimes_\cR\suspended\enveloping{\cL}$
of $(k+1)$-copies of the left $\cR$-module $\suspended\enveloping{\cL}$.

A Hochschild coboundary differential
\begin{equation}\label{eq:hochschild}
\hochschild:\suspended\enveloping{\cL}^{\otimes k}
\to\suspended\enveloping{\cL}^{\otimes k+1}
\end{equation}
and Gerstenhaber bracket
\begin{equation}\label{eq:Gbraket}
\gerstenhaber{\argument}{\argument}:
\suspended\UcL{p}\otimes\suspended\UcL{q}\to\suspended\UcL{p+q}
\end{equation}
can be defined by the following explicit algebraic identities:
\begin{multline*}
\hochschild(u_1\otimes\cdots\otimes u_k)
= (\pm) 1\otimes u_1\otimes\cdots\otimes u_k
+ \sum_{i=1}^{k} (\pm) u_1\otimes\cdots\otimes\Delta(u_i)\otimes\cdots\otimes u_k \\
+ (\pm) u_1\otimes\cdots\otimes u_k\otimes 1
\end{multline*}
and
\begin{equation}\label{hazmat}
\gerstenhaber{\phi}{\psi} = \phi\star\psi - (\pm) \psi\star\phi
,\end{equation}
where $\phi\star\psi\in\suspended\UcL{p+q}$ is defined by
\begin{equation}\label{eq:pre-Lie}
\phi\star\psi = \sum_{k=0}^{p} (\pm)
u_0\otimes\cdots\otimes u_{k-1}\otimes
(\hat{\Delta}^q u_k)\cdot \psi
\otimes u_{k+1}\otimes\cdots\otimes u_p
\end{equation}
if $\phi=u_0\odot u_1\odot\cdots\odot u_p$ for some
$u_0,u_1,\dots,u_p\in\suspended\enveloping{\cL}$
and $\psi\in\suspended\enveloping{\cL}^{\otimes q+1}$.
We refer the reader to~\cite{MR1815717} for the precise meaning
of the product $(\hat{\Delta}^q u_k)\cdot\psi$
in $\suspended\UcL{q}$ appearing in the last equation above.
Here $\hat{\Delta}:\suspended\enveloping{\cL}\to
\suspended\enveloping{\cL}\otimes_{\cR}\suspended\enveloping{\cL}$
is the map induced by the comultiplication $\Delta$ on $\enveloping{\cL}$.

Finally, the tensor algebra of $\suspended\enveloping{\cL}$ over $\cR$
carries an obvious cup product
\begin{equation}\label{eq:cup}
\suspended\enveloping{\cL}^{\otimes p}\otimes
\suspended\enveloping{\cL}^{\otimes q}
\xto{\smile}\suspended\enveloping{\cL}^{\otimes p+q}
,\end{equation}
the tensor product over $\cR$ itself:
\[ \phi\smile\psi=\phi\otimes\psi .\]

For $n\in\ZZ$, we set
\[ \tot_\oplus^n\suspended\enveloping{\cL}^{\otimes\bullet+1}
=\bigoplus_{\substack{p+q=n\\ p,q\in\ZZ\\ q\geqslant -1}}
\big(\suspended\enveloping{\cL}^{\otimes q+1}\big)^p ,\]
where $\big(\suspended\enveloping{\cL}^{\otimes q+1}\big)^p$
denotes the subspace of $\suspended\enveloping{\cL}^{\otimes q+1}$
consisting of elements of degree $p+q$.

\begin{proposition}\label{pro:hongkong1}
Let $\cL$ be a dg Lie algebroid over $\cM$.
\begin{enumerate}
\item When endowed with the differential $\cQ+\hochschild$
and the Gerstenhaber bracket \eqref{eq:Gbraket},
$\tot_\oplus\suspended\enveloping{\cL}^{\otimes\bullet+1}$ is a dgla.
\item When endowed with the cup product
(i.e.\ the tensor product $\otimes_\cR$)
and the Gerstenhaber bracket, the Hochschild cohomology
$\hypercohomology^\bullet\big(\tot_\oplus\suspended\enveloping{\cL}^{\otimes \bullet+1},\cQ+\hochschild\big)$
is a Gerstenhaber algebra.
\end{enumerate}
\end{proposition}
\begin{proof}
This can be proved directly by adapting the ordinary Hochschild
cohomology theory of associative algebras \cite{MR0161898}.
Since $\enveloping{\cL}$ is a dg Hopf algebroid over $\cR$,
all relevant formulae in~\cite{MR0161898} concerning the algebraic
structures on the Hochschild cochain complex of an associative
algebras hold in our context, with the differential being
$\cQ+\hochschild$, and the pre-Lie bracket and
the cup product being given, respectively, by
Equation~\eqref{eq:pre-Lie} and Equation~\eqref{eq:cup}.
We leave the details to the reader.
\end{proof}

\begin{remark}\label{rk:compatible}
Contrary to Proposition~\ref{pro:hongkong},
here $\tot_\oplus\suspended\enveloping{\cL}^{\otimes\bullet+1}$
is not a differential Gerstenhaber algebra,
for the Lie bracket and the associative multiplication
are only compatible up to homotopy. Likewise, the associative
multiplication is graded commutative \emph{on the cohomology level},
for the cup product on cochains is graded commutative
only up to homotopy.
This is reminiscent of the ordinary Hochschild cohomology theory
of associative algebras --- see~\cite{MR0161898}.
\end{remark}

\subsection{Fedosov dg manifolds}\label{Fedosov_dg_abd}

In this section, we recall the basic construction of Fedosov dg manifolds
of a Lie pair. For details, see~\cite{MR4150934}.

Let $(L,A)$ be a Lie pair.
We use the symbols $B$ to denote the quotient vector bundle $L/A$ and $r$ to denote its rank.

Consider the endomorphism $\delta$ of the vector bundle
$\Lambda^\bullet L\dual\otimes\hat{S}B\dual$ defined by
\[ \delta(\omega\otimes\chi^J)=\sum_{m=1}^r
\big(q^\top(\chi_m)\wedge\omega\big)\otimes J_m\,\chi^{J-e_m} ,\]
for all $\omega\in\Lambda L\dual$ and $J\in\NN^r$.
Here $\{\chi_k\}_{k=1}^r$ denotes an arbitrary local frame
for the vector bundle $B\dual$, the symbol $q^\top$ denotes
the vector bundle morphism $q^\top:B\dual\to L\dual$ dual
to the quotient morphism $q:L\to B$, the symbol $e_m$ denotes
the multi-index $(0,\cdots,0,1,0,\cdots,0)$
having its single nonzero entry in the $m$-th position, and
\[ \chi^J=\underset{J_1 \text{ factors}}{\underbrace{\chi_1\odot\cdots\odot\chi_1}}
\odot \underset{J_2 \text{ factors}}{\underbrace{\chi_2\odot\cdots\odot\chi_2}}
\odot \cdots \odot \underset{J_r \text{ factors}}{\underbrace{\chi_r\odot\cdots\odot\chi_r}} \]
if $J=(J_1,J_2,\cdots,J_r)$.

The operator $\delta$ is a derivation of degree $+1$ of the bundle of
graded commutative algebras ${\Lambda^\bullet L\dual
\otimes\hat{S} B\dual }$ and satisfies $\delta^2=0$.
The resulting cochain complex
\[ \begin{tikzcd}[column sep=small]
\cdots \arrow[r] & \Lambda^{n-1} L\dual \otimes\hat{S} B\dual \arrow[r, "\delta"] &
\Lambda^{n} L\dual \otimes\hat{S} B\dual \arrow[r, "\delta"] &
\Lambda^{n+1} L\dual \otimes\hat{S} B\dual \arrow[r] & \cdots
\end{tikzcd} \]
admits a contraction onto the complex with trivial differential
\[ \begin{tikzcd}
\cdots \arrow[r] & \Lambda^{n-1} A\dual \arrow[r, "0"] &
\Lambda^{n} A\dual \arrow[r, "0"] & \Lambda^{n+1} A\dual \arrow[r] & \cdots
\end{tikzcd} \]
Indeed, for every choice of splitting $i\circ p+j\circ q=\id_L$ of the short
exact sequence
\begin{equation}
\label{eq:9}
\begin{tikzcd}
0 \arrow[r] & A \arrow[r, "i"] & L \arrow[l, "p", bend left, dashed]
\arrow[r, "q"] & B \arrow[r] \arrow[l, "j", bend left, dashed] & 0
\end{tikzcd}
\end{equation}
and its dual
\[ \begin{tikzcd}
0 \arrow[r] & B\dual \arrow[r, "q^\top"] & L\dual
\arrow[l, "j^\top", bend left, dashed] \arrow[r, "i^\top"]
& A\dual \arrow[r] \arrow[l, "p^\top", bend left, dashed] & 0
\end{tikzcd} ,\]
the chain maps
\[ \sigma:\Lambda^\bullet L\dual \otimes\hat{S} B\dual \to\Lambda^\bullet A\dual \]
and
\[ \tau:\Lambda^\bullet A\dual \to\Lambda^\bullet L\dual \otimes\hat{S} B\dual \]
respectively defined by
\begin{equation}\label{Eq:sigma}
\sigma(\omega\otimes\chi^J)
=\begin{cases} i^\top(\omega) & \text{if }\length{J}=0 \\
0 & \text{otherwise,} \end{cases}
\end{equation}
for all $\omega\in\Lambda^\bullet(L\dual)$,
and \[ \tau(\alpha)=p^\top(\alpha)\otimes 1 ,\] for all $\alpha\in\Lambda^\bullet(A\dual)$,
satisfy
\[ \sigma\tau=\id \qquad \text{and} \qquad \id-\tau\sigma= h\delta+\delta h ,\]
where the homotopy operator
\[ h:\Lambda^{\bullet} L\dual \otimes\hat{S} B\dual
\to\Lambda^{\bullet-1} L\dual \otimes\hat{S} B\dual \]
is defined by
\[ h(\omega\otimes\chi^J)
=\begin{cases}\frac{1}{v+\length{J}}\sum_{k=1}^r
(\iota_{j(\partial_k)}\omega)\otimes\chi^{J+e_k} & \text{if } v\geqslant 1 \\
0 & \text{if } v=0 \end{cases} \]
for all $\omega\in p^\top(\Lambda^u A\dual)\otimes q^\top(\Lambda^v B\dual)$.
Here $\{\partial_k\}_{k=1}^r$ denotes the local frame for $B$ dual to $\{\chi_k\}_{k=1}^r$.
Notice that $h\tau=0$, $\sigma h=0$, and $h^2 =0$,
i.e.\ the triple of maps $(\tau,\sigma,h)$ make up a contraction
of $\Lambda^\bullet L\dual\otimes \hat{S}B\dual$ onto $\Lambda^\bullet A\dual$.
We remark that the operator $h$ is \emph{not} a derivation of the
algebra $\sections{\Lambda^\bullet L\dual\otimes\hat{S}B\dual}$.
However, the contraction $(\tau,\sigma,h)$ is compatible with
the graded commutative algebra structures
on $\sections{\Lambda^\bullet L\dual\otimes\hat{S}B\dual}$
and $\wa$ in the following sense:

\begin{lemma}\label{lem:semifull}
The triple $(\tau,\sigma,h)$ is a semifull algebra contraction
--- see Definition~\ref{def:semifull} ---
of $\sections{\Lambda^\bullet L\dual\otimes\hat{S}B\dual}$
onto $\wa$.
Furthermore, $\sigma$ and $\tau$ are morphisms of graded algebras.
\end{lemma}

\begin{proof}
The fact that $\tau$ and $\sigma$ are algebra morphisms follows
directly from the definitions. Moreover, the last four identities
in Definition~\ref{def:semifull} follow at once from the fact that
$\sigma$ is an algebra morphism (and the identities $\sigma h=0$
and $\sigma\tau=\operatorname{id}$).
Denoting by $\mu$ the product
on $\sections{\Lambda^{\bullet}L\dual\otimes\hat{S}B\dual}$,
the remaining identities to prove are
(recall that $\mu$ is graded commutative)
\[ h\mu(h\otimes h)=0,\qquad h\mu(h\otimes\tau)=0,\qquad h\mu(\tau\otimes\tau)=0 .\]	
To prove these, we introduce a second operator of degree $(-1)$
\[ \eta:\Lambda^{\bullet} L\dual \otimes\hat{S} B\dual
\to\Lambda^{\bullet-1} L\dual \otimes\hat{S} B\dual \]
defined by
\begin{equation}\label{derivationeta}
\eta(\omega\otimes\chi^J)
=\sum_{k=1}^r (\iota_{j(\partial_k)}\omega)\otimes\chi^{J+e_k}
\end{equation}
for all $\omega\in\Lambda^\bullet(L\dual)$.
Notice the similarity with the definition of $h$.
It is easy to check that $\eta$ is an algebra derivation
and $\ker(\eta)=\ker(h)$.
Using these two facts, we see that
\[ h^2=0 \quad\Rightarrow\quad \eta h=0 \quad\Rightarrow\quad
\eta\mu(h\otimes h)=\mu(\eta\otimes\id+\id\otimes\eta)(h\otimes h)
=0 \quad\Rightarrow\quad h\mu(h\otimes h)=0. \]
The remaining identities are proved in a similar way.
\end{proof}

The notions of $L$-connection on $B$ extending the Bott $A$-connection
and of torsion-free $L$-connection on $B$ were defined in
\cite{MR3439229,arXiv:1408.2903}.
A torsion-free $L$-connection on $B$ is necessarily an
extension of the Bott $A$-connection \cite[Lemma~5.2]{arXiv:1408.2903}.
According to~\cite[Lemma~4.5]{MR4150934}, an $L$-connection $\nabla$ on $B$
is torsion-free if and only if \[ \delta d_L^{\nabla}+d_L^{\nabla}\delta=0 .\]

\begin{theorem}[{\cite[Proposition~4.6]{MR4150934}}] \label{strawberry}
Let $(L,A)$ be a Lie pair with quotient $B=L/A$.
We interpret the sections of the bundle $L\dual\otimes\hat{S}B\dual\otimes B$
as derivations of the algebra
$\sections{\Lambda^\bullet L\dual\otimes\hat{S}B\dual}$ in the natural way.
Given a splitting of the short exact sequence \eqref{eq:9}
and a torsion-free $L$-connection $\nabla$ on $B$,
there exists a \emph{unique} derivation
\[ X^\nabla\in\sections{L\dual\otimes\hat{S}^{\geqslant 2}B\dual\otimes B}, \]
satisfying $\big(h\otimes\id_B\big)(X^\nabla)=0$ and such that the derivation
\[ Q:\sections{\Lambda^\bullet L\dual\otimes\hat{S}B\dual}
\to\sections{\Lambda^{\bullet+1}L\dual\otimes\hat{S}B\dual} \]
defined by
\begin{equation}\label{eq:Q}
Q=-\delta+d_L^\nabla+X^\nabla
\end{equation}
satisfies $Q^2=0$.
\end{theorem}

As a consequence, $(\cM=L[1]\oplus B,Q=-\delta+d_L^\nabla+X^\nabla)$ is a dg manifold,
which we call a \emph{Fedosov dg manifold} associated with the Lie pair $(L,A)$.
The Fedosov dg manifold $(\cM,Q)$ of Theorem~\ref{strawberry} was also obtained independently by
Batakidis--Voglaire~\cite{MR3724780} in the case of matched pairs.

\begin{remark}
The Kapranov dg manifold $A[1]\oplus L/A$ introduced in~\cite{arXiv:1408.2903}
is a dg submanifold of the Fedosov dg manifold $(L[1]\oplus L/A, Q)$.
\end{remark}

In order to study the dependence of the above construction on the involved choices,
it is useful to review a different description of the Fedosov dg manifold,
which can also be found in~\cite{MR4150934}.
As shown in~\cite{MR2989383,arXiv:1408.2903} (see also \cite[\S 3.4]{MR4150934}),
the choice of a splitting $j:B\to L$ of the short exact sequence
$0\to A\to L\to B\to 0$ and of an $L$-connection $\nabla$ on $B$
determines a Poincar\'e--Birkhoff--Witt isomorphism
of filtered $C^\infty(M)$-coalgebras (PBW map in short)
\begin{equation}
\label{eq:pbw}
\operatorname{pbw}:\sections{SB}\to\frac{\enveloping{L}}{\enveloping{L}\sections{A}}.\end{equation}
Since $\enveloping{L}\sections{A}$ is a left ideal of the algebra $\enveloping{L}$,
there is a natural left $\sections{L}$-action on the quotient
$\frac{\enveloping{L}}{\enveloping{L}\sections{A}}$,
and an induced flat $L$-connection $\cn$ on $SB$:
\begin{equation}
\label{eq:cn}
\cn_l(s) = \operatorname{pbw}\inv(l\cdot\operatorname{pbw}(s))
\end{equation}
for all $l\in\sections{L}$ and $s\in\sections{SB}$. Moreover, for every $l\in\sections{L}$,
the covariant derivative $\cn_l$ is a coderivation of the $C^\infty(M)$-coalgebra $\sections{SB}$.

Dualizing, we obtain an $L$-connection on $\hat{S}(B\dual)$, which we continue to denote by $\cn$.
Furthermore, for every $l\in\sections{L}$ the covariant derivative
$\cn_l$ is a derivation of the $C^\infty(M)$-algebra
$\sections{\hat{S}(B\dual)}$.
Finally, this latter fact implies that the induced Chevalley--Eilenberg differential
\begin{equation}\label{eq:CEL}
d_L^{\cn}:\sections{\Lambda^\bullet L\dual\otimes\hat{S}(B\dual)}
\to\sections{\Lambda^{\bullet+1}L\dual\otimes\hat{S}(B\dual)}
\end{equation}
is a derivation of the algebra
$\sections{\Lambda^\bullet L\dual\otimes\hat{S}(B\dual)}$,
and can thus be regarded as a homological vector field
on the graded manifold $L[1]\oplus B$.
One of the main results of~\cite{MR4150934} is the following theorem.

\begin{theorem}[{\cite[Theorem~4.7]{MR4150934}}]\label{thm:EW}
Given a Lie pair $(L,A)$, together with a splitting
of the short exact sequence $0\to A\to L\to B\to 0$
and a torsion-free $L$-connection on $B$,
the dg manifold $(L[1]\oplus B,d_L^{\cn})$
constructed above and the dg manifold $(L[1]\oplus B,Q)$
constructed via the Fedosov iteration in Theorem~\ref{strawberry}
coincide, i.e.\ $d_L^{\cn}=Q$.
\end{theorem}

The Fedosov dg manifolds obtained from different choices
of a splitting and a connection are isomorphic to one another.
The isomorphism can be made explicit via the associated PBW maps.
This shall be needed in Section~\ref{sec:uniqueness},
where we will establish the uniqueness claim (2b)
of Theorem~\ref{Macau} from the introduction.

We consider two different choices $j_1,\nabla_1$ and $j_2,\nabla_2$
of a splitting $B\to L$ and a torsion-free $L$-connection on $B$
as before; the two induced homological vector fields $Q_1$ and $Q_2$
on $L[1]\oplus B$; and the two induced Poincar\'e--Birkhoff--Witt
isomorphisms $\pbw_1$ and $\pbw_2$.
The composition of the latter
\[ \psi:=\pbw_1\inv\circ \pbw_2:\sections{SB}\to\sections{SB} \]
is an automorphism of the $C^\infty(M)$-coalgebra $\sections{SB}$
intertwining the two induced $L$-module structures.
Likewise, the dual map
$\psi\dual:\sections{\hat{S}(B\dual)}\to\sections{\hat{S}(B\dual)}$
is an automorphism of the $C^\infty(M)$-algebra
$\sections{\hat{S}(B\dual)}$
intertwining the two induced $L$-module structures.
Finally, it follows immediately that
\[ \id\otimes\psi\dual :
\big(\sections{\Lambda^\bullet L\dual\otimes\hat{S}(B\dual)},Q_1\big)
\to \big(\sections{\Lambda^\bullet L\dual\otimes\hat{S}(B\dual)},Q_2\big) \]
is an isomorphism of dg manifolds $(L[1]\oplus B,Q_2)\to(L[1]\oplus B,Q_1)$.

\subsection{Fedosov dg Lie algebroids}

Let $(L,A)$ be a Lie pair over a base manifold $M$.
Given a splitting $j:B\to L$ of the short exact sequence of vector bundles
$0\to A \to L \to B \to 0$ and a torsion-free $L$-connection $\nabla$
on $B$,
one constructs a Fedosov dg manifold $(\cM, Q)$, where $\cM=L[1]\oplus B$,
as in Theorem~\ref{strawberry}.

Let $R=C^\infty(M)$.
Let $\cF\to\cM$ denote the pullback of the vector bundle $B\to M$
through the surjective submersion $\cM\to M$.
It is a graded vector bundle whose total space $\cF$ is the graded manifold with base $M$
associated with the graded vector bundle $L[1]\oplus B\oplus B\to M$.
Its space of sections $\sections{\cF\to\cM}$ is canonically identified
with $C^\infty(\cM)\otimes_{R}\sections{B}\cong
\sections{\Lambda^\bullet L\dual\otimes\hat{S}(B\dual)\otimes B}$.
It is naturally a vector subbundle of $T_{\cM}\to\cM$;
the inclusion $\sections{\cF\to\cM}\into\XX(\cM)$ takes the section
$(\lambda\otimes\chi^J)\otimes\partial_k\in C^\infty(\cM)\otimes_{R}\sections{B}$
of the vector bundle $\cF\to\cM$ to the derivation
$\mu\otimes\chi^M\mapsto\lambda\wedge\mu\otimes M_k\chi^{M+J-e_k}$ of $C^{\infty}(\cM)$.

Alternatively, denote by $T_{\ver}B\to B$ the formal vertical tangent bundle of the vector
bundle $B\to M$, which consists of all formal vertical tangent vectors of $B$.
Its space of sections $\verticalX(B):=\sections{B;T_{\ver}B}$
is naturally isomorphic to $\sections{\hat{S}(B\dual)\otimes B}$.
Indeed $T_{\ver}B$ is a double vector bundle
\cite{arXiv:math/9808081}, which is isomorphic to $B\oplus B$.
Consider the projection $\pr:\cM=L[1]\oplus B\to B$.
Then $\cF$ is isomorphic to the pull back bundle $\pr^* T_{\ver}B$.

\begin{proposition}\label{pro:Rome}
The subbundle $\cF\subset T_\cM$ is a dg integrable distribution
(or a dg foliation) of the dg manifold $(\cM,Q)$,
i.e.\ $\cF$ is a dg Lie subalgebroid of the tangent dg Lie algebroid
$T_{\cM}\to\cM$.
\end{proposition}

\begin{proof}
It is simple to see that $\cF\to\cM$ is a Lie subalgebroid of $T_\cM\to\cM$.
Hence it suffices to show that $\cF$ admits a dg manifold structure
such that $\cF\to\cM$ is a dg subbundle of $T_\cM\to\cM$ ---
the compatibility condition \eqref{eq:compatibility}
holds automatically since $T_\cM\to\cM$ is a dg Lie algebroid
(see Example~\ref{exp:dgtangent2}).
According to Remark~\ref{observation_above}, it suffices to prove that
$\sections{\cM;\cF}$ is a dg module over $\big(C^\infty(\cM),Q\big)$.
It is clear that
$\sections{\cM;\cF}\isomorphism\sections{\Lambda^\bullet L\dual}
\otimes_R\verticalX(B)$.
From Equation~\eqref{eq:Q}, it follows that
$\sections{\Lambda^\bullet L\dual}\otimes_R\verticalX(B)$
is stable under the Lie derivative $\lie{Q}$.
Moreover, we have
\[ \lie{Q}(\xi\cdot(\eta\otimes X))=Q(\xi)\cdot(\eta\otimes X)
+(-1)^{\degree{\xi}}\xi\cdot\lie{Q}(\eta\otimes X) ,\]
for all homogeneous $\xi\in\sections{\Lambda^\bullet L\dual\otimes\hat{S}B\dual}$,
$\eta\in\sections{\Lambda^\bullet L\dual}$, and $X\in\verticalX(B)$.
Therefore, $\sections{\cM;\cF}$ is a dg module over $\big(C^\infty(\cM),Q\big)$.
\end{proof}

Any dg Lie algebroid constructed in this manner
is called a \emph{Fedosov dg Lie algebroid}
associated with the Lie pair $(L,A)$.

Next, we will identify the space of polyvector fields
on the Fedosov dg Lie algebroid $\cF$ over $\cM$.

Set
\begin{equation}\label{eq:Xpoly}
\Tpoly{k}=\sections{\Lambda^{k+1}B}
\end{equation}
and let $\verticalTpoly{k}$ denote
$\sections{B;\Lambda^{k+1}T_{\ver}B}$,
the space of formal vertical $(k+1)$-vector fields on $B$.
It is clear that
\begin{equation}\label{eq:Xpoly1}
\verticalTpoly{k}\isomorphism\sections{\hat{S}(B\dual)}\otimes_R \Tpoly{k}
\end{equation}
and
\[ \sections{\cM;\Lambda^{k+1}\cF}\isomorphism
\sections{\Lambda^{\bullet}L\dual}\otimes_R\verticalTpoly{k}\isomorphism
\sections{\Lambda^{\bullet}L\dual\otimes\hat{S}B\dual}\otimes_R\Tpoly{k} .\]

Since $\cF$ is a dg Lie subalgebroid of $T_\cM$,
the subspace $\sections{\cM;\Lambda^{k+1}\cF}\isomorphism
\sections{\Lambda^{\bullet}L\dual\otimes\hat{S}B\dual}\otimes_R\Tpoly{k}$
of the space $T_{\poly}^{k}(\cM)$ of $(k+1)$-vector fields
on $\cM=L[1]\oplus B$ is stable under $\lie{Q}$,
we obtain a cochain complex
\[ \begin{tikzcd}
\cdots \arrow[r] &
\sections{\Lambda^{u}L\dual \otimes \hat{S}B\dual}\otimes_R\Tpoly{k}
\arrow[r, "\lie{Q}"] &
\sections{\Lambda^{u+1}L\dual \otimes \hat{S}B\dual}\otimes_R\Tpoly{k}
\arrow[r] & \cdots
\end{tikzcd} \]
for each $k\geqslant -1$.

Applying Proposition~\ref{pro:hongkong} to the
Fedosov dg Lie algebroid $\cF\to\cM$, we obtain the following

\begin{proposition}\label{pro:Lyon}
The total complex
$\Big(\tot\big(\sections{\Lambda^\bullet L\dual\otimes\hat{S}B\dual}
\otimes_R\Tpoly{\bullet}\big),\lie{Q}\Big)$
admits a differential Gerstenhaber algebra,
whence a dgla structure.
\end{proposition}

Finally, we consider the space of polydifferential operators
on the Fedosov dg Lie algebroid $\cF$ over $\cM$.

Let $\verticalDpoly{k}$ denote the space of formal vertical
$(k+1)$-differential operators on the vector bundle $B$
and set $\verticalDpoly{\bullet}=\bigoplus_{k=-1}^{\infty}\verticalDpoly{k}$.
There exists a canonical isomorphism
\begin{equation}\label{eq:varphi}
\begin{tikzcd}
\sections{\hat{S}(B\dual)\otimes
\underset{k+1\text{ factors}}{\underbrace{S(B)\otimes\cdots\otimes S(B)}}}
\arrow{r}{\varphi}[swap]{\isomorphism} &
\verticalDpoly{k}
\end{tikzcd} .
\end{equation}

In terms of local dual frames $\{\chi_i\}_{i=1,\ldots,r}$
and $\{\partial_j\}_{j=1,\ldots,r}$ for $B\dual$ and $B$ respectively,
and the corresponding local frames $\{\chi^I\}_{I\in\mathbb{N}^r}$
and $\{\partial_J\}_{J\in\mathbb{N}^r}$ for $ \hat{S}(B\dual)$
and $S(B)$ respectively, the isomorphism $\varphi$ sends
$\chi^I\otimes\partial_{J_0}\otimes\cdots\otimes\partial_{J_k}
\in\sections{\hat{S}(B\dual)\otimes
\underset{k+1\text{ factors}}{\underbrace{S(B)\otimes\cdots\otimes S(B)}}}$
to the $(k+1)$-differential operator
\[ \sections{\hat{S}(B\dual)}^{\otimes k+1}\ni
\chi^{I_0}\otimes\cdots\otimes\chi^{I_k}\longmapsto
\chi^I\cdot\partial_{J_0}(\chi^{I_0})
\cdots\partial_{J_k}(\chi^{I_k})
\in\sections{\hat{S}(B\dual)} .\]

The algebra $C^\infty(L[1]\oplus B)$ is a
module over its subalgebra
$\sections{\Lambda^\bullet L}\isomorphism\sections{\Lambda^\bullet L\dual\otimes S^0(B\dual)}$.
The subspace of $D_{\poly}^{\bullet}(L[1]\oplus B)$ comprised of all
$\sections{\Lambda^\bullet L\dual}$-multilinear polydifferential operators
is easily identified to
$\tot\big(\sections{\Lambda^\bullet L\dual}\otimes_R\verticalDpoly{\bullet}\big)$.
It is simple to see that the universal enveloping algebra
$\enveloping{\cF}$ of the dg Lie algebroid $\cF\to\cM$
is naturally identified
with $\sections{\Lambda^\bullet L\dual}\otimes_R\verticalDpoly{0}$,
which is a dg Hopf algebroid over
$\cR=C^\infty (\cM)\isomorphism\sections{\Lambda^\bullet L\dual \otimes\hat{S} B\dual }$.
Moreover,
$\enveloping{\cF}$ is a dg Hopf subalgebroid of $D_{\poly}^{0}(L[1]\oplus B)$.
Notice that
\begin{equation}\label{Halifax}
\suspended\enveloping{\cF}^{\otimes k+1}\isomorphism \sections{\Lambda^\bullet L\dual}
\otimes_R\verticalDpoly{k}
.\end{equation}

Since $\cF$ is a dg Lie subalgebroid of $T_\cM$,
the subspace
\[ \tot_\oplus\suspended\enveloping{\cF}^{\bullet+1}\isomorphism
\tot\big(\sections{\Lambda^\bullet L\dual}\otimes_R\verticalDpoly{\bullet}\big) \]
of $D_{\poly}^{\bullet}(\cM)$
is stable under the Hoch\-schild coboundary operator
$\gerstenhaber{Q+m}{\argument}$.
Here $m=1\otimes 1$ is the element of $\suspended\enveloping{\cF}^{\otimes 2}$
arising from the multiplication of $C^\infty(\cM)$.

The Lie bracket \eqref{eq:Gbraket} and the cup product \eqref{eq:cup}
on $\tot_\oplus\suspended\enveloping{\cF}^{\bullet+1}$ carry over to a
$\sections{\Lambda^\bullet L\dual}$-linear
Lie bracket and cup product on $\tot\big(\sections{\Lambda^\bullet L\dual}
\otimes_R\verticalDpoly{\bullet}\big)$ through the identification \eqref{Halifax}.

Applying Proposition~\ref{pro:hongkong1} to the
Fedosov dg Lie algebroid $\cF\to\cM$, we obtain the following

\begin{proposition}\label{Bujumbura}
\strut
\begin{enumerate}
\item The triple $\big( \tot\big(\sections{\Lambda^\bullet L\dual}
\otimes_R\verticalDpoly{\bullet}\big), \gerstenhaber{Q+m}{\argument},
\gerstenhaber{\argument}{\argument} \big)$ is a dgla.
\item The cohomology group
$\hypercohomology^\bullet \Big(\tot\big(\sections{\Lambda^\bullet L\dual}
\otimes_R\verticalDpoly{\bullet}\big), \gerstenhaber{Q+ m}{\argument}\Big)$
equipped with the induced Lie bracket and cup product
is a Gerstenhaber algebra.
\end{enumerate}
\end{proposition}

\section{\texorpdfstring{$L_\infty$}{L∞} algebra structures}
\label{Chatanooga}

In this section, we endow the spaces of polyvector fields and polydifferential
operators of a Lie pair --- see Section~\ref{Boise} --- with $L_\infty$ algebra structures,
which are canonical up to $L_\infty$ isomorphism.

\subsection{Dolgushev--Fedosov contraction and \texorpdfstring{$L_\infty$}{L∞} algebra structure on the space of polyvector fields of a Lie pair}

The following lemma is straightforward.

\begin{lemma}
The subspace $\sections{\Lambda^\bullet L\dual}\otimes_R\verticalTpoly{k}$
of the space $T_{\poly}^{k}(L[1]\oplus B)$ of $(k+1)$-vector fields on $L[1]\oplus B$
is stable under $\lie{\delta}$.
\end{lemma}

Also, note that the following diagram commutes:
\[ \begin{tikzcd}
\sections{\Lambda^{i}L\dual}\otimes_R\verticalTpoly{k}
\arrow{r}{\lie{\delta}}
\arrow[leftrightarrow]{d}{\isomorphism}
&
\sections{\Lambda^{i+1}L\dual}\otimes_R\verticalTpoly{k}
\arrow[leftrightarrow]{d}{\isomorphism}
\\
\sections{\Lambda^{i}L\dual\otimes\hat{S}(B\dual)}
\otimes_R \Tpoly{k}
\arrow{r}[swap]{\delta\otimes\id}
&
\sections{\Lambda^{i+1}L\dual\otimes\hat{S}(B\dual)}
\otimes_R \Tpoly{k}
\end{tikzcd} \]

Since the vector field $\delta$ on $L[1]\oplus B$
is homological, we obtain the cochain complex
\[ \begin{tikzcd}
\cdots \arrow{r} &
\sections{\Lambda^{i}L\dual}\otimes_R\verticalTpoly{k}
\arrow{r}{\lie{-\delta}} &
\sections{\Lambda^{i+1}L\dual}\otimes_R\verticalTpoly{k}
\arrow{r} & \cdots
\end{tikzcd} \]
which admits the descending filtration
\[ \mathscr{F}_m=\bigoplus_{i=0}^{\rk(L)} \sections{\Lambda^i L\dual
\otimes\hat{S}^{\geqslant m-i} B\dual }\otimes_R \Tpoly{k} .\]

We shall denote by $\etendu{\tau}$, $\etendu{\sigma}$ and
$\etendu{h}$ the maps defined by the following commutative diagrams
(where $\tau$, $\sigma$, and $h$ are the maps introduced
in Section~\ref{Fedosov_dg_abd})
\begin{equation}\label{eq:etendusigma}
\begin{tikzcd}[row sep=tiny]
\sections{\Lambda^{i}L\dual}\otimes_R\verticalTpoly{k}
\arrow{dr}{\etendu{\sigma}}
\arrow[leftrightarrow]{dd}{\isomorphism} & \\
& \sections{\Lambda^{i}A\dual}\otimes_R\Tpoly{k} \\
\sections{\Lambda^{i}L\dual\otimes\hat{S} B\dual }
\otimes_R \Tpoly{k}
\arrow{ur}[swap]{\sigma\otimes\id} &
\end{tikzcd}
\end{equation}

\[ \begin{tikzcd}[row sep=tiny]
& \sections{\Lambda^{i}L\dual}\otimes_R\verticalTpoly{k}
\arrow[leftrightarrow]{dd}{\isomorphism} \\
\sections{\Lambda^{i}A\dual}\otimes_R\Tpoly{k}
\arrow{ur}{\etendu{\tau}}
\arrow{dr}[swap]{\tau\otimes\id} & \\
& \sections{\Lambda^{i}L\dual\otimes\hat{S} B\dual }
\otimes_R \Tpoly{k}
\end{tikzcd} \]

\begin{equation}\label{eq:etenduh}
\begin{tikzcd}
\sections{\Lambda^{i}L\dual}\otimes_R\verticalTpoly{k}
\arrow{r}{\etendu{h}}
\arrow[leftrightarrow]{d}{\isomorphism}
&
\sections{\Lambda^{i-1}L\dual}\otimes_R\verticalTpoly{k}
\arrow[leftrightarrow]{d}{\isomorphism}
\\
\sections{\Lambda^{i}L\dual\otimes\hat{S}(B\dual)}
\otimes_R\Tpoly{k}
\arrow{r}[swap]{h\otimes\id}
&
\sections{\Lambda^{i-1}L\dual\otimes\hat{S}(B\dual)}
\otimes_R\Tpoly{k}
\end{tikzcd}
\end{equation}

Adapting the proof of~\cite[Proposition~4.3]{MR4150934}, we obtain
\begin{proposition}\label{Shiraz}
The complex
$\big( \sections{\Lambda^{\bullet}L\dual}
\otimes_R\verticalTpoly{k}, \lie{-\delta} \big)$
contracts onto
$\big( \sections{\Lambda^{\bullet}A\dual}
\otimes_R\Tpoly{k}, 0 \big)$.
More precisely, we have the filtered contraction
\begin{dontbotheriftheequationisoverlong} \begin{tikzcd}
\cdots \arrow{r} & \sections{\Lambda^{n-1} L\dual}\otimes_R\verticalTpoly{k}
\arrow{d}{\etendu{\sigma}}
\arrow{r}{\lie{-\delta}} &
\sections{\Lambda^{n} L\dual}\otimes_R\verticalTpoly{k}
\arrow{d}{\etendu{\sigma}}
\arrow{r}{\lie{-\delta}}
\arrow[dashed]{ddl}[near end]{\etendu{h}} &
\sections{\Lambda^{n+1} L\dual }\otimes_R\verticalTpoly{k}
\arrow{d}{\etendu{\sigma}} \arrow{r}
\arrow[dashed]{ddl}[near end]{\etendu{h}} & \cdots \\
\cdots \arrow{r} &
\sections{\Lambda^{n-1} A\dual}\otimes_R\Tpoly{k}
\arrow{d}[swap]{\etendu{\tau}} \arrow{r}[near start]{0} &
\sections{\Lambda^{n} A\dual}\otimes_R\Tpoly{k}
\arrow{d}[swap]{\etendu{\tau}} \arrow{r}[near start]{0} &
\sections{\Lambda^{n+1} A\dual}\otimes_R\Tpoly{k}
\arrow{d}[swap]{\etendu{\tau}} \arrow{r} & \cdots \\
\cdots \arrow{r} &
\sections{\Lambda^{n-1} L\dual }\otimes_R\verticalTpoly{k}
\arrow{r}{\lie{-\delta}} &
\sections{\Lambda^{n} L\dual }\otimes_R\verticalTpoly{k}
\arrow{r}{\lie{-\delta}} &
\sections{\Lambda^{n+1} L\dual }\otimes_R\verticalTpoly{k}
\arrow{r} & \cdots
\end{tikzcd} \end{dontbotheriftheequationisoverlong}
where $\etendu{\tau}$, $\etendu{\sigma}$ and
$\etendu{h}$ are defined by the above commutative diagrams \eqref{eq:etendusigma}-\eqref{eq:etenduh}.
\end{proposition}

\begin{lemma}\label{lem:polyvsemifull}
The contraction $(\tau_\natural,\sigma_\natural,h_\natural)$
in Proposition~\ref{Shiraz} is a semifull algebra contraction
--- on both sides, the associative multiplication is the wedge product.
Moreover, the maps $\etendu{\tau}$ and $\etendu{\sigma}$
preserve the wedge products.
\end{lemma}

\begin{proof}
This follows immediately from the definitions
and the corresponding statements for $(\tau,\sigma,h)$
--- see Lemma~\ref{lem:semifull}.
\end{proof}

Consider the homological vector field $Q$ on $L[1]\oplus B$
introduced in Theorem~\ref{strawberry}:
\[ Q=-\delta+\perturbation \qquad\text{with}\qquad
\perturbation=d_L^\nabla+X^\nabla
\qquad\text{and}\qquad
X^\nabla\in\sections{L\dual\otimes
\hat{S}^{\geqslant 2}(B\dual)\otimes B} .\]

\begin{proposition}\label{Bagnolet}
There exists a contraction
\begin{equation}\label{eq:Bagnolet}
\begin{tikzcd}[cramped]
\Big(\tot\big(\wa\otimes_R
\Tpoly{\bullet}\big)
,d_A^{\Bott}\Big)
\arrow[r, " \etendu{\perturbed{\tau}}", shift left] &
\Big(\tot\big(\sections{\Lambda^\bullet L\dual}\otimes_R
\verticalTpoly{\bullet}\big), \lie{Q}\Big)
\arrow[l, " \etendu{\sigma}", shift left]
\arrow[loop, "\etendu{\perturbed{h}}",out=5,in=-5,looseness = 3]
\end{tikzcd}
\end{equation}
More precisely, for every $k\geqslant -1$, we have the (filtered) contraction
\begin{dontbotheriftheequationisoverlong} \begin{tikzcd}
\cdots \arrow{r} & \sections{\Lambda^{n-1}L\dual}\otimes_R\verticalTpoly{k}
\arrow{d}{\etendu{\sigma}}
\arrow{r}{\lie{Q}} &
\sections{\Lambda^{n}L\dual}\otimes_R\verticalTpoly{k}
\arrow{d}{\etendu{\sigma}}
\arrow{r}{\lie{Q}}
\arrow[dashed]{ddl}[near end]{\etendu{\perturbed{h}}} &
\sections{\Lambda^{n+1}L\dual}\otimes_R\verticalTpoly{k}
\arrow{d}{\etendu{\sigma}} \arrow{r}
\arrow[dashed]{ddl}[near end]{\etendu{\perturbed{h}}} & \cdots \\
\cdots \arrow{r} &
\sections{\Lambda^{n-1}A\dual}\otimes_R\Tpoly{k}
\arrow{d}[swap]{\etendu{\perturbed{\tau}}} \arrow{r}{d_A^{\Bott}} &
\sections{\Lambda^{n}A\dual}\otimes_R\Tpoly{k}
\arrow{d}[swap]{\etendu{\perturbed{\tau}}} \arrow{r}{d_A^{\Bott}} &
\sections{\Lambda^{n+1}A\dual}\otimes_R\Tpoly{k}
\arrow{d}[swap]{\etendu{\perturbed{\tau}}} \arrow{r} & \cdots \\
\cdots \arrow{r} &
\sections{\Lambda^{n-1}L\dual}\otimes_R\verticalTpoly{k}
\arrow{r}{\lie{Q}} &
\sections{\Lambda^{n}L\dual}\otimes_R\verticalTpoly{k}
\arrow{r}{\lie{Q}} &
\sections{\Lambda^{n+1}L\dual}\otimes_R\verticalTpoly{k}
\arrow{r} & \cdots
\end{tikzcd} \end{dontbotheriftheequationisoverlong}
where
\begin{equation}\label{eq:perturbedetendutau}
\etendu{\perturbed{h}}=\sum_{l=0}^\infty
(\etendu{h}\lie{\perturbation})^l \etendu{h},
\quad\text{ and }\quad \etendu{\perturbed{\tau}}=\sum_{l=0}^\infty
(\etendu{h}\lie{\perturbation})^l\etendu{\tau}.
\end{equation}
Moreover, the cochain maps $\etendu{\perturbed{\tau}}$
and $\etendu{\sigma}$ intertwine the wedge products
on their domain and codomain.
\end{proposition}

As an immediate consequence of Proposition~\ref{Bagnolet},
by considering the bigradings on both sides of~\eqref{eq:Bagnolet},
we obtain the following

\begin{corollary}\label{Smith}
For every $k\geqslant -1$, we have a contraction
\[ \begin{tikzcd}[cramped]
\Big(\wa\otimes_R\Tpoly{k}\big),\dabott \Big)
\arrow[r, shift left, "\etendu{\perturbed{\tau}}"] &
\Big(\sections{\Lambda^\bullet L\dual}\otimes_R\verticalTpoly{k}\big),\lie{Q}\Big)
\arrow[l, shift left, "\etendu{\sigma}"]
\arrow[loop, "\etendu{\perturbed{h}}", out=5,in=-5,looseness = 3]
\end{tikzcd} \]
\end{corollary}

The case $k=-1$ was established in~\cite[Proposition~5.4]{MR4150934}.

The proof of Proposition~\ref{Bagnolet} requires the following technical results.

\begin{lemma}\label{Cergy}
Let $\po$ denote the canonical projection
$\hat{S}(B\dual)\otimes B\onto
S^0(B\dual)\otimes B$.
For all $a\in\sections{A}$ and $j\in\{1,\dots,r\}$, we have
\[ \po\big(\schouten{\cn_a}{\partial_j}\big)
=\nabla^{\Bott}_a(\partial_j) .\]
Recall that $\{\partial_j\}_{j=1,\dots,r}$ is a local frame
for the vector subbundle $B\isomorphism S^0(B\dual)\otimes B$
of $\hat{S}(B\dual)\otimes B$. Here we think of $\partial_j$ as a local
section of $\hat{S}(B\dual)\otimes B$.
The sections of the vector bundle $\hat{S}(B\dual)\otimes B$
may be interpreted as fiberwise formal vertical vector fields on $B$
--- they act as derivations of the algebra $\sections{\hat{S}(B\dual)}$
of fiberwise formal functions on $B$ in a natural fashion.
\end{lemma}

\begin{proof}
We have seen that, for all $a\in\sections{A}$,
the operator $\cn_a$ is a derivation of
$\sections{\hat{S}(B\dual)}$, which stabilizes
the filtration $\sections{\hat{S}^{\geqslant n}(B\dual)}$.
Therefore, there exist local sections
$\theta_k^M$ of $A\dual$ such that
\[ \cn_a \chi_k=
\sum_{\substack{M\in\NO^r \\ \length{M}\geqslant 1}}
\interior{a}\theta_k^M\cdot \chi^M .\]
It follows that $\cn_a$ may be regarded
as a section of $\hat{S}^{\geqslant 1}(B\dual)\otimes B$:
\[ \cn_a =
\sum_{k=1}^r \bigg(\sum_{\substack{M\in\NO^r \\ \length{M}\geqslant 1}}
\interior{a}\theta_k^M\cdot \chi^M\bigg) \partial_k .\]

On one hand, it follows from
\begin{multline*} \schouten{\cn_a}{\partial_j}
= \cn_a\circ\partial_j - \partial_j\circ
\cn_a =
\sum_{k=1}^r \sum_{\length{M}\geqslant 1}
\interior{a}\theta_k^M\cdot\chi^M\partial_k\circ\partial_j
-\sum_{k=1}^r \sum_{\length{M}\geqslant 1}
\interior{a}\theta_k^M\cdot
\partial_j\circ(\chi^M\partial_k) \\
=-\sum_{k=1}^r \sum_{\length{M}\geqslant 1}
\interior{a}\theta_k^M\cdot M_j \chi^{M-e_j}\cdot\partial_k
\end{multline*}
that
\[ \po\big(\schouten{\cn_a}{\partial_j}\big)
= -\sum_{k=1}^r \interior{a}\theta_k^{e_j} \cdot \partial_k .\]
On the other hand, it follows from
\[ 0=\anchor_{a}\underset{\delta_{k,j}}{\underbrace{\duality{\chi_k}{\partial_j}}}
=\duality{\cn_a\chi_k}{\partial_j}
+\duality{\chi_k}{\cn_a\partial_j} \]
and the fact that $\cn_a$ stabilizes the
subspace $\sections{S^1(B)}$ of $\sections{S(B)}$
that
\[ \begin{split}
\cn_a\big(\partial_j\big)
=&\ \sum_k \duality{\chi_k}{
\cn_a \partial_j}\partial_k \\
=&\ -\sum_k \duality{\cn_a \chi_k}{\partial_j} \partial_k \\
=&\ -\sum_k \sum_{\length{M}\geqslant 1}
\interior{a}\theta_k^M\cdot\duality{\chi^M}{\partial_j}
\cdot\partial_k \\
=&\ -\sum_k \interior{a}\theta_k^{e_j}\cdot\partial_k
.\end{split} \]
Finally, for all $a\in\sections{A}$
and $b\in\sections{B}$, we have
$\cn_a(b)=\nabla^{\Bott}_a(b)$ as
\begin{multline*}
\pbw(\cn_a b-\nabla^{\Bott}_a b)
= a\cdot\pbw(b)-\pbw\big(q\bracket{a}{j(b)}\big) \\
= a\cdot j(b)-j\circ q(\bracket{a}{j(b)})
= j(b)\cdot a+ \underset{\in\sections{A}}
{\underbrace{p(\bracket{a}{j(b)})}} =0
\end{multline*}
in $\frac{\enveloping{L}}{\enveloping{L}\sections{A}}$.
The proof is complete.
\end{proof}

\begin{lemma}\label{Drancy}
$\etendu{\sigma}\circ \lie{\perturbation}\circ \etendu{\tau}=d_A^{\Bott}$
\end{lemma}

\begin{proof}
Let $(l_k)_{k\in\{1,\dots,\rk{L}\}}$ denote any local frame of $L$
and let $(\lambda_k)_{k\in\{1,\dots,\rk{L}\}}$ denote the dual local frame of $L\dual$.
Likewise, let $(a_k)_{k\in\{1,\dots,\rk{A}\}}$ denote any local frame of $A$
and let $(\alpha_k)_{k\in\{1,\dots,\rk{A}\}}$ denote the dual local frame of $A\dual$.
For all $\omega\in\wa$, $n\in\NN$, and
$j_0,\dots,j_n\in\{1,\dots,r\}$, we have
\begin{align*}
&\ \etendu{\sigma}\Big(
\schouten{\perturbation}
{\etendu{\tau}\big(\omega\otimes\partial_{j_0}\wedge
\cdots \wedge\partial_{j_n}\big)}\Big) \\
=&\ \etendu{\sigma}
\Big( \schouten{\perturbation}
{p^\top\omega\otimes 1\otimes\partial_{j_0}\wedge
\cdots\wedge\partial_{j_n})} \Big) \\
=&\ \etendu{\sigma} \Big(
d_L(p^\top\omega)\otimes 1\otimes\partial_{j_0}\wedge
\cdots\wedge\partial_{j_n}) \\
&\qquad +\sum_k\lambda_k\wedge p^\top\omega\otimes
\schouten{\nabla_{l_k}-\interior{l_k}\Xi}{1\otimes\partial_{j_0}\wedge\cdots\wedge\partial_{j_n}}
\Big) \\
=&\ \sigma\big(d_L(p^\top\omega)\otimes 1\big)
\otimes \partial_{j_0}\wedge\cdots
\wedge\partial_{j_n} \\
&\qquad +\sum_k \etendu{\sigma}
\Big( p^\top\alpha_k\wedge p^\top\omega
\otimes \schouten{\cn_{a_k}}
{1\otimes\partial_{j_0}\wedge\cdots\wedge
\partial_{j_n}} \Big) \\
=&\ d_A\omega\otimes \partial_{j_0}\wedge\cdots
\wedge\partial_{j_n} \\
&\qquad +\sum_k \etendu{\sigma}
\bigg(p^\top(\alpha_k\wedge\omega)\otimes
\bigg\{ \sum_{t=0}^n
1\otimes\partial_{j_0}\wedge\cdots\wedge
\schouten{\cn_{a_k}}{\partial_{j_t}}\wedge\cdots\wedge\partial_{j_n}
\bigg\} \bigg) \\
=&\ d_A\omega\otimes\partial_{j_0}\wedge\cdots
\wedge\partial_{j_n} \\
&\qquad +\sum_k\sum_{t=0}^n
\alpha_k\wedge\omega\otimes
\partial_{j_0}\wedge\cdots\wedge
\po\schouten{\cn_{a_k}}
{\partial_{j_t}}\wedge\cdots\wedge\partial_{j_n}
.\end{align*}
It follows from Lemma~\ref{Cergy} that
\[ \po\schouten{\cn_{a_k}}
{\partial_{j_t}} = \cn_{a_k}(\partial_{j_t})
= \nabla^{\Bott}_{a_k}(\partial_{j_t}) .\]
Hence, we conclude that $\etendu{\sigma}\circ\lie{\perturbation}
\circ\etendu{\tau}=d_A^{\Bott}$.
\end{proof}

\begin{proof}[Proof of Proposition~\ref{Bagnolet}]
We proceed by homological perturbation
--- see Lemma~\ref{HPL} and also~\cite[Lemma~A.1]{MR4150934}
and~\cite[\S1]{MR1109665}.
Starting from the filtered contraction of Proposition~\ref{Shiraz}, it suffices to perturb
the coboundary operator $\lie{-\delta}$ by the operator $\lie{\perturbation}$.
One checks that
$\etendu{\sigma}\lie{\perturbation}\etendu{h}=0$.
It follows that
\[ \etendu{\perturbed{\sigma}}:=\sum_{l=0}^\infty
\etendu{\sigma}(\lie{\perturbation}\etendu{h})^l
=\etendu{\sigma} \]
and, making use of Lemma~\ref{Drancy},
\[ \vartheta:=\sum_{l=0}^\infty
\etendu{\sigma}(\lie{\perturbation} \etendu{h})^l\lie{\perturbation}\etendu{\tau}
=\etendu{\sigma}\lie{\perturbation}\etendu{\tau}
=d_A^{\Bott} .\]
The result follows immediately since $-\delta+\perturbation=Q$.

Finally, the claim that $\sigma_{\natural}$ is compatible
with the wedge products is contained in Lemma~\ref{lem:polyvsemifull},
while the same statement for $\perturbed{\tau}_\natural$
follows from Lemmas~\ref{lem:polyvsemifull}, \ref{lem:semifullstable}
and~\ref{lem:semifulltau}.
\end{proof}

The next proposition gives an alternative characterization of the
map $\etendu{\perturbed{\tau}}$ as the solution of an initial value problem.

\begin{proposition}\label{Mandalay}
Given $x\in\wa\otimes_R\Tpoly{\diamond}$
and $y\in\sections{\Lambda^\bullet L\dual}
\otimes_R\verticalTpoly{\diamond}$,
we have
\[ \etendu{\perturbed{\tau}}(x)=y
\qquad\text{if and only if}\qquad
\begin{cases}
\etendu{h}(y) = 0 \\
\etendu{h}\lie{Q}(y) = 0 \\
\etendu{\sigma}(y) = x
\end{cases}
\quad\text{if and only if}\qquad
\begin{cases}
\lie{\eta}(y) = 0 \\
\lie{\bracket{\eta}{Q}}(y) = 0 \\
\etendu{\sigma}(y) = x
\end{cases} \]
\end{proposition}
The derivation $\eta$ was defined in Equation~\eqref{derivationeta}.

\begin{proof}
Assume $\etendu{\perturbed{\tau}}(x)=y$.
From $\etendu{h}\etendu{\tau}=0$ and $\etendu{h}\etendu{h}=0$,
we get \[ \etendu{h}\etendu{\perturbed{\tau}}=
\etendu{h}\sum_{l=0}^\infty
(\etendu{h}\lie{\perturbation})^l\etendu{\tau}
=\etendu{h}\etendu{\tau}+\etendu{h}(\etendu{h}\lie{\perturbation})\sum_{l=0}^\infty
(\etendu{h}\lie{\perturbation})^l\etendu{\tau}
=0 .\]
It follows that
\[ \etendu{h}(y)=\etendu{h}\etendu{\perturbed{\tau}}(x)=0
\qquad\text{and}\qquad
\etendu{\sigma}(y)=\etendu{\sigma}\etendu{\perturbed{\tau}}(x)=x .\]
Furthermore, since $\etendu{\perturbed{\tau}}$ is a chain map,
we have
\[ \etendu{h}\lie{Q}(y)=
\etendu{h}\lie{Q}\etendu{\perturbed{\tau}}(x)=
\etendu{h}\etendu{\perturbed{\tau}}d_A^{\Bott}(x)=0 .\]

Conversely, assuming $\etendu{h}(y)=0$;
$\etendu{h}\lie{Q}(y)=0$; and $\etendu{\sigma}(y)=x$,
it follows from
\[ \etendu{\perturbed{\tau}}\etendu{\sigma}-\id=
\etendu{\perturbed{h}}\lie{Q}+\lie{Q}\etendu{\perturbed{h}} \]
that
\[ \etendu{\perturbed{\tau}}(x)-y=
\etendu{\perturbed{\tau}}\etendu{\sigma}(y)-y=
\etendu{\perturbed{h}}\lie{Q}(y)+\lie{Q}\etendu{\perturbed{h}}(y)=
\sum_{l=0}^\infty(\etendu{h}\lie{\perturbation})^l
\etendu{h}\lie{Q}(y)+\lie{Q}
\sum_{l=0}^\infty(\etendu{h}\lie{\perturbation})^l
\etendu{h}(y)=0 \] and we can conclude that
$\etendu{\perturbed{\tau}}(x)=y$.

Finally, it is not difficult to show that
$\ker(\etendu{h})=\ker(\lie{\eta})$.
It follows that
\[ \begin{cases} \etendu{h}(y)=0 \\
\etendu{h}\lie{Q}(y)=0 \end{cases}
\quad\text{if and only if}\qquad
\begin{cases} \lie{\eta}(y)=0 \\
\lie{\eta}\lie{Q}(y)=0 \end{cases}
\quad\text{if and only if}\qquad
\begin{cases} \lie{\eta}(y)=0 \\
\lie{\bracket{\eta}{Q}}(y)=0 \end{cases}
\qedhere \]
\end{proof}

It follows from the homotopy transfer theorem for $L_\infty$ algebras
\cite{MR3276839,MR1932522,MR1950958,arXiv:1705.02880,MR2361936,arXiv:1807.03086,MR3318161,MR3323983}
applied to the contraction in Proposition~\ref{Bagnolet}
that the dgla structure carried by
$\tot\big(\sections{\Lambda^\bullet L\dual}\otimes_R\verticalTpoly{\bullet}\big)$
induces an $L_\infty$ algebra structure on
$\tot\big(\wa\otimes_R\Tpoly{\bullet}\big)$.
Moreover, since the retraction $\etendu{\sigma}$ preserves
the wedge products
according to Proposition~\ref{Bagnolet},
we immediately obtain the following

\begin{proposition}\label{Bari}
Given a Lie pair $(L,A)$, each choice of a splitting $j:B\to L$ of the short
exact sequence of vector bundles $0\to A \to L \to B \to 0$
and of a torsion-free $L$-connection $\nabla$ on $B$ determines
\begin{enumerate}
\item an $L_\infty$ algebra structure on
$\tot\big(\wa\otimes_R\Tpoly{\bullet}\big)$
with the operator $d_A^{\Bott}$ as unary bracket
\item and a Gerstenhaber algebra structure on
$\hypercohomology^\bullet_{\CE}(A,\Tpoly{\bullet})$,
the cohomology of the total complex
\[ \Big(\tot\big(\wa
\otimes_R\Tpoly{\bullet}\big),d_A^{\Bott}\Big) ,\]
where the Lie bracket is induced by the binary bracket
of the $L_\infty$ algebra structure on
$\tot\big(\wa
\otimes_R\Tpoly{\bullet}\big)$
and the multiplication by the wedge product \eqref{eq:wedge01}.
\end{enumerate}
\end{proposition}

\begin{remark}
One can prove that the $L_\infty$ algebra structure
on $\tot\big(\wa\otimes_R\Tpoly{\bullet}\big)$
is compatible with the wedge product in the sense that
all $L_\infty$ multibrackets are multi-derivations
with respect to the wedge product.
In other words, in the terminology of~\cite{MR4091493},
$\tot\big(\wa\otimes_R\Tpoly{\bullet}\big)$
is a $(+1)$-shifted derived Poisson algebra.\footnote{In the context of $\ZZ_2$-grading,
$(+1)$-shifted derived Poisson algebras are also called homotopy Schouten algebras
by Khudaverdian--Voronov~\cite{MR2757715}.
Note that, $0$-shifted derived Poisson algebras
were studied by Oh--Park~\cite{MR2180451}
and Cattaneo--Felder~\cite{MR2304327},
who called them $P_\infty$ algebras.}
\end{remark}

\subsection{Dolgushev--Fedosov contraction and \texorpdfstring{$L_\infty$}{L∞} algebra structure on the space of polydifferential operators of a Lie pair}

Denote the space of polydifferential operators on
$L[1]\oplus B$ by $D_{\poly}^{\bullet}(L[1]\oplus B)$.
The Hochschild cohomology of the Fedosov dg manifold $(L[1]\oplus B,Q)$
is the cohomology of the cochain complex
$\big(D_{\poly}^{\bullet}(L[1]\oplus B),\gerstenhaber{Q+m}{\argument}\big)$.
The algebra of functions $C^\infty(L[1]\oplus B)$ is a module over its subalgebra
$\sections{\Lambda^\bullet L\dual}\isomorphism
\sections{\Lambda^\bullet L\dual\otimes S^0(B\dual)}$.
The subspace of
$D_{\poly}^{\bullet}(L[1]\oplus B)$ comprised of all
$\sections{\Lambda^\bullet L\dual}$-multilinear polydifferential operators
is easily identified to
$\tot\big(\sections{\Lambda^\bullet L\dual}\otimes_R\verticalDpoly{\bullet}\big)$,
the space of polydifferential operators on the Fedosov dg Lie algebroid $\cF$.
Since $\cF$ is a dg Lie subalgebroid of the tangent bundle $T_{\cM}\to\cM$
of the Fedosov dg manifold $(L[1]\oplus B,Q)$, it follows that
the subspace $\tot\big(\sections{\Lambda^\bullet L\dual}\otimes_R\verticalDpoly{\bullet}\big)$ of
$D_{\poly}^{\bullet}(L[1]\oplus B)$ is stable under the Hochschild coboundary
operator ${\gerstenhaber{Q+m}{\argument}}$ of the Fedosov dg manifold
$(L[1]\oplus B,Q)$.

We also have the following
\begin{lemma}
The subspace $\tot\big(\sections{\Lambda^\bullet L\dual}\otimes_R\verticalDpoly{\bullet}\big)$
of $D_{\poly}^{\bullet}(L[1]\oplus B)$ is stable under
$\gerstenhaber{\delta}{\argument}$.
\end{lemma}

\begin{lemma}
The diagram
\[ \begin{tikzcd}
\sections{\Lambda^p L\dual}\otimes_R\verticalDpoly{v}
\arrow{r}{\gerstenhaber{\delta}{\argument}} &
\sections{\Lambda^{p+1} L\dual}\otimes_R\verticalDpoly{v} \\
\sections{\Lambda^p L\dual}\otimes_R\verticalDpoly{v-1}
\arrow{u}{(-1)^p \gerstenhaber{m}{\argument}}
\arrow{r}{\gerstenhaber{\delta}{\argument}} &
\sections{\Lambda^{p+1} L\dual}\otimes_R\verticalDpoly{v-1}
\arrow{u}{(-1)^{p+1} \gerstenhaber{m}{\argument}}
\end{tikzcd} \]
commutes.
\end{lemma}

\begin{proof}
It suffices to verify that the diagrams
\[ \begin{tikzcd}[column sep=huge]
\sections{\Lambda^p L\dual}\otimes_R\verticalDpoly{v}
\arrow{r}{\gerstenhaber{\delta}{\argument}} &
\sections{\Lambda^{p+1} L\dual}\otimes_R\verticalDpoly{v} \\
\sections{\Lambda^p L\dual\otimes\hat{S}(B\dual)}\otimes_R\sections{(SB)^{\otimes v+1}}
\arrow{r}{\delta\otimes\id}
\arrow{u}{\id\otimes\varphi}[swap]{\isomorphism} &
\sections{\Lambda^{p+1} L\dual\otimes\hat{S}(B\dual)}\otimes_R\sections{(SB)^{\otimes v+1}}
\arrow{u}{\id\otimes\varphi}[swap]{\isomorphism}
\end{tikzcd} \]
and
\[ \begin{tikzcd}[column sep=huge]
\sections{\Lambda^p L\dual}\otimes_R\verticalDpoly{v-1}
\arrow{r}{\gerstenhaber{m}{\argument}} &
\sections{\Lambda^{p} L\dual}\otimes_R\verticalDpoly{v} \\
\sections{\Lambda^p L\dual\otimes\hat{S}(B\dual)}\otimes_R\sections{(SB)^{\otimes v}}
\arrow{r}{\id\otimes(-1)^{v-1}\hochschild}
\arrow{u}{\id\otimes\varphi}[swap]{\isomorphism} &
\sections{\Lambda^p L\dual\otimes\hat{S}(B\dual)}\otimes_R\sections{(SB)^{\otimes v+1}}
\arrow{u}{\id\otimes\varphi}[swap]{\isomorphism}
\end{tikzcd} \]
commute.
\end{proof}

\begin{proposition}\label{Liege}
The diagram
\[ \begin{tikzcd}[column sep=large,row sep=small]
\vdots & \vdots & \vdots & \\
\sections{\Lambda^0 L\dual}\otimes_R\verticalDpoly{1}
\arrow{u}{\gerstenhaber{m}{\argument}} \arrow{r}{\gerstenhaber{-\delta}{\argument}} &
\sections{\Lambda^1 L\dual}\otimes_R\verticalDpoly{1}
\arrow{u}{-\gerstenhaber{m}{\argument}} \arrow{r}{\gerstenhaber{-\delta}{\argument}} &
\sections{\Lambda^2 L\dual}\otimes_R\verticalDpoly{1}
\arrow{u}{\gerstenhaber{m}{\argument}} \arrow{r}{\gerstenhaber{-\delta}{\argument}} &
\cdots \\
\sections{\Lambda^0 L\dual}\otimes_R\verticalDpoly{0}
\arrow{u}{\gerstenhaber{m}{\argument}} \arrow{r}{\gerstenhaber{-\delta}{\argument}} &
\sections{\Lambda^1 L\dual}\otimes_R\verticalDpoly{0}
\arrow{u}{-\gerstenhaber{m}{\argument}} \arrow{r}{\gerstenhaber{-\delta}{\argument}} &
\sections{\Lambda^2 L\dual}\otimes_R\verticalDpoly{0}
\arrow{u}{\gerstenhaber{m}{\argument}} \arrow{r}{\gerstenhaber{-\delta}{\argument}} &
\cdots \\
\sections{\Lambda^0 L\dual}\otimes_R\verticalDpoly{-1}
\arrow{u}{\gerstenhaber{m}{\argument}} \arrow{r}{\gerstenhaber{-\delta}{\argument}} &
\sections{\Lambda^1 L\dual}\otimes_R\verticalDpoly{-1}
\arrow{u}{-\gerstenhaber{m}{\argument}} \arrow{r}{\gerstenhaber{-\delta}{\argument}} &
\sections{\Lambda^2 L\dual}\otimes_R\verticalDpoly{-1}
\arrow{u}{\gerstenhaber{m}{\argument}} \arrow{r}{\gerstenhaber{-\delta}{\argument}} &
\cdots
\end{tikzcd} \]
is a double complex.
\end{proposition}

Its total complex
\[ \cdots\to\tot^{n}\Big(\sections{\Lambda^{\bullet} L\dual}
\otimes_R\verticalDpoly{\bullet}\Big) \xto{\gerstenhaber{-\delta+m}{\argument}}
\tot^{n+1}\Big(\sections{\Lambda^{\bullet} L\dual}
\otimes_R\verticalDpoly{\bullet}\Big)\to\cdots \]
admits the descending filtration
$\mathscr{F}_{0}\supset\mathscr{F}_{1}\supset
\mathscr{F}_{2}\supset\mathscr{F}_{3}\supset\cdots$
defined by
\[ \mathscr{F}_m=\bigoplus_{k=0}^{\rk(L)}
\sections{\Lambda^k(L\dual)}\otimes_R\varphi
\Big(\bigoplus_{q=-1}^{\infty}
\sections{\hat{S}^{\geqslant m-k}(B\dual)\otimes
\underset{q+1\text{ factors}}{\underbrace{S(B)\otimes\cdots\otimes S(B)}}}\Big) .\]
Here $\varphi$ is as in~\eqref{eq:varphi}.

We shall denote by $\etendu{\tau}$, $\etendu{\sigma}$ and $\etendu{h}$
the maps defined by the following commutative diagrams
(where $\tau$, $\sigma$, and $h$ are the maps introduced
in Section~\ref{Fedosov_dg_abd}
and $\varphi$ is the identification \eqref{eq:varphi}):

\begin{equation}\label{AAA} \begin{tikzcd}[row sep=tiny]
\sections{\Lambda^{u}L\dual}\otimes_R\verticalDpoly{v}
\arrow{dr}{\etendu{\sigma}} & \\
& \sections{\Lambda^{u}A\dual}\otimes_R\Dpoly{v} \\
\sections{\Lambda^{u}L\dual\otimes\hat{S}(B\dual)}
\otimes_R\sections{(SB)^{\otimes v+1}}
\arrow{uu}{\id\otimes\varphi}[swap]{\isomorphism}
\arrow{ur}[swap]{\sigma\otimes\pbw^{\otimes v+1}} &
\end{tikzcd} \end{equation}

\begin{equation}\label{BBB} \begin{tikzcd}[row sep=tiny]
& \sections{\Lambda^{u}L\dual}\otimes_R\verticalDpoly{v} \\
\sections{\Lambda^{u}A\dual}\otimes_R\Dpoly{v}
\arrow{ur}{\etendu{\tau}}
\arrow{dr}[swap]{\tau\otimes(\pbw\inv)^{\otimes v+1}} & \\
& \sections{\Lambda^{u}L\dual\otimes\hat{S}(B\dual)}
\otimes_R\sections{(SB)^{\otimes v+1}}
\arrow{uu}{\id\otimes\varphi}[swap]{\isomorphism}
\end{tikzcd} \end{equation}

\begin{equation}\label{CCC} \begin{tikzcd}
\sections{\Lambda^{u}L\dual}\otimes_R\verticalDpoly{v}
\arrow{r}{\etendu{h}}
&
\sections{\Lambda^{u-1}L\dual}\otimes_R\verticalDpoly{v}
\\
\sections{\Lambda^{u}L\dual\otimes\hat{S}(B\dual)}
\otimes_R \sections{(SB)^{\otimes v+1}}
\arrow{r}[swap]{h\otimes\id}
\arrow{u}{\id\otimes\varphi}[swap]{\isomorphism}
&
\sections{\Lambda^{u-1}L\dual\otimes\hat{S}(B\dual)}
\otimes_R \sections{(SB)^{\otimes v+1}}
\arrow{u}{\id\otimes\varphi}[swap]{\isomorphism}
\end{tikzcd} \end{equation}

The following proposition can be easily verified.

\begin{proposition}\label{Damas}
The cochain complex
$\big(\tot\big( \sections{\Lambda^{\bullet}L\dual}
\otimes_R\verticalDpoly{\bullet}\big),
\gerstenhaber{-\delta+m}{\argument}\big)$
contracts onto
$\big(\tot\big(\sections{\Lambda^{\bullet}A\dual}
\otimes_R\Dpoly{\bullet}\big),\id\otimes\hochschild\big)$.
More precisely, we have a filtered contraction
\begin{dontbotheriftheequationisoverlong} \begin{tikzcd}[column sep=huge]
\cdots \arrow{r} &
\tot^{n}\Big(\sections{\Lambda^{\bullet} L\dual }\otimes_R\verticalDpoly{\bullet}\Big)
\arrow{d}{\etendu{\sigma}}
\arrow{r}{\gerstenhaber{-\delta+m}{\argument}} &
\tot^{n+1}\Big(\sections{\Lambda^{\bullet} L\dual }\otimes_R\verticalDpoly{\bullet}\Big)
\arrow{d}{\etendu{\sigma}} \arrow{r}
\arrow[dashed]{ddl}[near end]{\etendu{h}} & \cdots \\
\cdots \arrow{r} &
\tot^{n}\Big(\sections{\Lambda^{\bullet} A\dual }\otimes_R\Dpoly{\bullet}\Big)
\arrow{d}[swap]{\etendu{\tau}} \arrow{r}[near start]{\id\otimes\hochschild} &
\tot^{n+1}\Big(\sections{\Lambda^{\bullet} A\dual }\otimes_R\Dpoly{\bullet}\Big)
\arrow{d}[swap]{\etendu{\tau}} \arrow{r} & \cdots \\
\cdots \arrow{r} &
\tot^{n}\Big(\sections{\Lambda^{\bullet} L\dual }\otimes_R\verticalDpoly{\bullet}\Big)
\arrow{r}{\gerstenhaber{-\delta+m}{\argument}} &
\tot^{n+1}\Big(\sections{\Lambda^{\bullet} L\dual }\otimes_R\verticalDpoly{\bullet}\Big)
\arrow{r} & \cdots
\end{tikzcd} \end{dontbotheriftheequationisoverlong}
where
$\etendu{\sigma}$, $\etendu{\tau}$ and $\etendu{h}$ are
the maps defined by the commutative diagrams
\eqref{AAA}, \eqref{BBB}, and \eqref{CCC}.
\end{proposition}
\begin{lemma}\label{lem:polydsemifull}
The contraction $(\tau_\natural,\sigma_\natural,h_\natural)$
in Proposition~\ref{Damas}
is a semifull algebra contraction
(where the associative product on both sides is the cup product).
Moreover, the maps $\etendu{\tau}$ and $\etendu{\sigma}$
respect the cup products.
\end{lemma}

\begin{proof}
This follows easily from the definitions
and the corresponding statements for $(\tau,\sigma,h)$
--- see Lemma~\ref{lem:semifull}.
\end{proof}

\begin{remark}
For future reference, we point out that the same maps
$(\etendu{\tau},\etendu{\sigma},\etendu{h})$
also define a filtered contraction of
$\big(\tot\big(\sections{\Lambda^{\bullet}L\dual}
\otimes_R\verticalDpoly{\bullet}\big),
\gerstenhaber{-\delta}{\argument}\big)$
onto
$\big(\tot\big(\sections{\Lambda^{\bullet}A\dual}
\otimes_R\Dpoly{\bullet}\big),0)$.
As for Proposition~\ref{Damas}, we leave the verification
of this claim as an easy exercise for the reader.
\end{remark}

\begin{lemma}\label{Namur}
The diagram
\[ \begin{tikzcd}
\sections{\Lambda^p L\dual}\otimes_R\verticalDpoly{v+1}
\arrow{r}{\gerstenhaber{\perturbation}{\argument}} &
\sections{\Lambda^{p+1} L\dual}\otimes_R\verticalDpoly{v+1} \\
\sections{\Lambda^p L\dual}\otimes_R\verticalDpoly{v}
\arrow{u}{(-1)^{p}\gerstenhaber{m}{\argument}}
\arrow{r}{\gerstenhaber{\perturbation}{\argument}} &
\sections{\Lambda^{p+1} L\dual}\otimes_R\verticalDpoly{v}
\arrow{u}{(-1)^{p+1}\gerstenhaber{m}{\argument}}
\end{tikzcd} \]
commutes.
\end{lemma}

\begin{proof}[Sketch of proof]
We have $\gerstenhaber{\perturbation}{m}=0$ because,
for every $l\in\sections{L}$, the operator $\interior{l}\perturbation$
is a derivation for the multiplication $m$ on $C^\infty(L[1]\oplus B)$.
\end{proof}

It follows from Proposition~\ref{Liege} and Lemma~\ref{Namur} that
\[ \begin{tikzcd}[column sep=large,row sep=small]
\vdots & \vdots & \vdots & \\
\sections{\Lambda^0 L\dual}\otimes_R\verticalDpoly{1}
\arrow{u}{\gerstenhaber{m}{\argument}} \arrow{r}{\gerstenhaber{-\delta+\perturbation}{\argument}} &
\sections{\Lambda^1 L\dual}\otimes_R\verticalDpoly{1} \arrow{u}{-\gerstenhaber{m}{\argument}}
\arrow{r}{\gerstenhaber{-\delta+\perturbation}{\argument}} &
\sections{\Lambda^2 L\dual}\otimes_R\verticalDpoly{1} \arrow{u}{\gerstenhaber{m}{\argument}}
\arrow{r}{\gerstenhaber{-\delta+\perturbation}{\argument}} & \cdots \\
\sections{\Lambda^0 L\dual}\otimes_R\verticalDpoly{0} \arrow{u}{\gerstenhaber{m}{\argument}}
\arrow{r}{\gerstenhaber{-\delta+\perturbation}{\argument}} &
\sections{\Lambda^1 L\dual}\otimes_R\verticalDpoly{0} \arrow{u}{-\gerstenhaber{m}{\argument}}
\arrow{r}{\gerstenhaber{-\delta+\perturbation}{\argument}} &
\sections{\Lambda^2 L\dual}\otimes_R\verticalDpoly{0} \arrow{u}{\gerstenhaber{m}{\argument}}
\arrow{r}{\gerstenhaber{-\delta+\perturbation}{\argument}} & \cdots \\
\sections{\Lambda^0 L\dual}\otimes_R\verticalDpoly{-1} \arrow{u}{\gerstenhaber{m}{\argument}}
\arrow{r}{\gerstenhaber{-\delta+\perturbation}{\argument}} &
\sections{\Lambda^1 L\dual}\otimes_R\verticalDpoly{-1} \arrow{u}{-\gerstenhaber{m}{\argument}}
\arrow{r}{\gerstenhaber{-\delta+\perturbation}{\argument}} &
\sections{\Lambda^2 L\dual}\otimes_R\verticalDpoly{-1} \arrow{u}{\gerstenhaber{m}{\argument}}
\arrow{r}{\gerstenhaber{-\delta+\perturbation}{\argument}} & \cdots
\end{tikzcd} \]
is a double complex.

Indeed, the operator $\gerstenhaber{\perturbation}{\argument}$
is a perturbation of the filtered complex
\[ \cdots\to\tot^{n}\Big(\sections{\Lambda^{\bullet} L\dual}\otimes_R\verticalDpoly{\bullet}\Big)
\xto{\gerstenhaber{-\delta+m}{\argument}} \tot^{n+1}\Big(\sections{\Lambda^{\bullet} L\dual}
\otimes_R\verticalDpoly{\bullet}\Big)\to\cdots \]

\begin{proposition}\label{Dhaka}
There exists a contraction
\begin{equation}\label{eq:Dhaka}
\begin{tikzcd}[cramped]
\Big(\tot\big(\wa\otimes_R\Dpoly{\bullet}\big),\dhoch \Big)
\arrow[r, shift left, "\etendu{\perturbed{\tau}}"] &
\Big(\tot\big(\sections{\Lambda^\bullet L\dual}\otimes_R
\verticalDpoly{\bullet}\big),\gerstenhaber{Q+ m}{\argument}\Big)
\arrow[l, shift left, "\etendu{\sigma}"]
\arrow[loop, "\etendu{\perturbed{h}}", out=5,in=-5,looseness = 3]
\end{tikzcd}
\end{equation}
--- recall that $\mathfrak{d}_{\mathscr{H}}:=\id\otimes\hochschild$.
More precisely, we have the (filtered) contraction
\begin{dontbotheriftheequationisoverlong} \begin{tikzcd}[column sep=large]
\cdots \arrow{r} &
\tot^{n}\Big(\sections{\Lambda^{\bullet}(L\dual)}\otimes_R\verticalDpoly{\bullet}\Big)
\arrow{d}{\etendu{\sigma}}
\arrow{r}{\gerstenhaber{Q+m}{\argument}} &
\tot^{n+1}\Big(\sections{\Lambda^{\bullet}(L\dual)}\otimes_R\verticalDpoly{\bullet}\Big)
\arrow{d}{\etendu{\sigma}} \arrow{r}
\arrow[dashed]{ddl}[near end]{\etendu{\perturbed{h}}} & \cdots \\
\cdots \arrow{r} &
\tot^{n}\Big(\sections{\Lambda^{\bullet}(A\dual)}\otimes_R\Dpoly{\bullet}\Big)
\arrow{d}[swap]{\etendu{\perturbed{\tau}}} \arrow{r}[near start]{\dhoch} &
\tot^{n+1}\Big(\sections{\Lambda^{\bullet}(A\dual)}\otimes_R\Dpoly{\bullet}\Big)
\arrow{d}[swap]{\etendu{\perturbed{\tau}}} \arrow{r} & \cdots \\
\cdots \arrow{r} &
\tot^{n}\Big(\sections{\Lambda^{\bullet}(L\dual)}\otimes_R\verticalDpoly{\bullet}\Big)
\arrow{r}{\gerstenhaber{Q+m}{\argument}} &
\tot^{n+1}\Big(\sections{\Lambda^{\bullet}(L\dual)}\otimes_R\verticalDpoly{\bullet}\Big)
\arrow{r} & \cdots
\end{tikzcd} \end{dontbotheriftheequationisoverlong}
where $\etendu{\perturbed{\tau}}=\sum_{l=0}^\infty
(\etendu{h}\circ\gerstenhaber{\perturbation}{\argument})^l\circ\etendu{\tau}$
and $\etendu{\perturbed{h}}=\sum_{l=0}^\infty
(\etendu{h}\circ\gerstenhaber{\perturbation}{\argument})^l\circ\etendu{h}$.

Moreover, the cochain maps $\etendu{\perturbed{\tau}}$ and $\etendu{\sigma}$ respect the cup products
on both sides.
\end{proposition}

As an immediate consequence of Proposition~\ref{Dhaka},
by considering the bigradings on both sides of~\eqref{eq:Dhaka},
we obtain the following

\begin{corollary}\label{cor:Dhaka}
For every $k\geqslant -1$, we have a contraction
\begin{equation}\label{eq:Dhaka1}
\begin{tikzcd}[cramped]
\Big(\wa\otimes_R\Dpoly{k},\dau\Big)
\arrow[r, shift left, "\etendu{\perturbed{\tau}}"] &
\Big(\sections{\Lambda^\bullet L\dual}\otimes_R\verticalDpoly{k},
\gerstenhaber{Q}{\argument}\Big)
\arrow[l, shift left, "\etendu{\sigma}"]
\arrow[loop, "\etendu{\perturbed{h}}", out=5,in=-5,looseness = 3]
\end{tikzcd}
\end{equation}
\end{corollary}

The case $k=-1$ was established in~\cite[Proposition~5.4]{MR4150934}.

The proof of Proposition~\ref{Dhaka} requires the following technical results.

\begin{lemma}\label{Vincennes}
Let $\po$ denote the canonical projection
$\hat{S}(B\dual)\otimes S(B)\onto
S^0(B\dual)\otimes S(B)$.
For all $a\in\sections{A}$ and $J\in\NO^r$, we have
\[ \po\big(\gerstenhaber{\cn_a}{\partial^J}\big)
=\cn_a(\partial^J) .\]
\end{lemma}
\begin{proof}
We have seen that, for all $a\in\sections{A}$,
the operator $\cn_a$ is a derivation of
$\sections{\hat{S}(B\dual)}$, which stabilizes
the filtration $\sections{\hat{S}^{\geqslant n}(B\dual)}$.
Therefore, there exist local sections
$\theta_k^M$ of $L\dual$ such that
\[ \cn_a \chi_k=
\sum_{\substack{M\in\NO^r \\ \length{M}\geqslant 1}}
\interior{a}\theta_k^M\cdot \chi^M .\]
It follows that $\cn_a$ may be regarded
as a section of $\hat{S}^{\geqslant 1}(B\dual)\otimes B$:
\[ \cn_a =
\sum_{k=1}^r \bigg(\sum_{\substack{M\in\NO^r \\ \length{M}\geqslant 1}}
\interior{a}\theta_k^M\cdot \chi^M\bigg) \partial_k .\]
On one hand, it follows from
\[ \gerstenhaber{\cn_a}{\partial^J}
= \cn_a\star\partial^J - \partial^J\star
\cn_a =
\sum_{k=1}^r \sum_{\length{M}\geqslant 1}
\interior{a}\theta_k^M\cdot\chi^M\partial^{J+e_k}
-\sum_{k=1}^r \sum_{\length{M}\geqslant 1}
\interior{a}\theta_k^M\cdot\big(\partial^J\star\chi^M\partial_k\big) ,\]
that
\[ \begin{split}
\po\big(\gerstenhaber{\cn_a}{\partial^J}\big)
=&\ -\sum_{k=1}^r \sum_{\length{M}\geqslant 1}
\interior{a}\theta_k^M \frac{J!}{M!(J-M)!}\partial^M(\chi^M)
\cdot\partial^{J-M+e_k} \\
=&\ -\sum_{k=1}^r \sum_{\length{M}\geqslant 1}
\interior{a}\theta_k^M \frac{J!}{(J-M)!}\cdot\partial^{J-M+e_k}
.\end{split} \]
On the other hand, it follows from
\[ 0 = \anchor_{a}\underset{K!\cdot\delta_{K,J}}{\underbrace{\duality{\chi^K}{\partial^J}}}
= \duality{\cn_a\chi^K}{\partial^J}
+ \duality{\chi^K}{\cn_a\partial^J} \]
that
\[ \begin{split}
\cn_a\big(\partial^J\big)
=&\ \sum_K \frac{1}{K!} \duality{\chi^K}{
\cn_a \partial^J}\partial^K \\
=&\ -\sum_K \frac{1}{K!}
\duality{\cn_a \chi^K}{\partial^J} \partial^K \\
=&\ -\sum_K \frac{1}{K!}
\duality{\sum_k K_k \chi^{K-e_k} \cn_a
\chi_k}{\partial^J} \partial^K \\
=&\ -\sum_K \frac{1}{K!} \sum_k K_k
\sum_{\length{M}\geqslant 1}\interior{a}\theta_k^M
\underset{J!\cdot \delta_{K-e_k+M,J}}{\underbrace{\duality{\chi^{K-e_k+M}}{\partial^J}}}
\partial^K \\
=&\ -\sum_k \sum_{\length{M}\geqslant 1}
\frac{J!}{(J-M+e_k)!} (J_k-M_k+1)
\interior{a}\theta_k^M \partial^{J-M+e_k}
\\ =&\ -\sum_k \sum_{\length{M}\geqslant 1} \frac{J!}{(J-M)!}
\interior{a}\theta_k^M \partial^{J-M+e_k}
.\end{split} \]
The proof is complete.
\end{proof}

\begin{lemma}\label{Pantin}
$\etendu{\sigma}\circ\gerstenhaber{\perturbation}{\argument}\circ\etendu{\tau}=\dau$
\end{lemma}

\begin{proof}
Let $(l_k)_{k\in\{1,\dots,\rk{L}\}}$ denote any local frame for $L$
and let $(\lambda_k)_{k\in\{1,\dots,\rk{L}\}}$ denote the dual local frame for $L\dual$.
Likewise let $(a_k)_{k\in\{1,\dots,\rk{A}\}}$ denote any local frame for $A$
and let $(\alpha_k)_{k\in\{1,\dots,\rk{A}\}}$ denote the dual local frame for $A\dual$.
For all $\omega\in\wa$, $n\in\NN$,
and $J_0,\dots,J_n\in\NO^{r}$, we have
\begin{align*}
&\ \etendu{\sigma}\Big(
\gerstenhaber{\perturbation}
{\etendu{\tau}\big(\omega\otimes
\pbw(\partial^{J_0})\otimes\cdots\otimes\pbw(\partial^{J_n})
\big)}\Big) \\
=&\ \etendu{\sigma}\Big( \gerstenhaber{\perturbation}
{p^\top\omega\otimes\varphi(
1\otimes\partial^{J_0}\otimes\cdots\otimes\partial^{J_n}
)} \Big) \\
=&\ \etendu{\sigma} \Big(
d_L(p^\top\omega)\otimes\varphi(
1\otimes\partial^{J_0}\otimes\cdots\otimes\partial^{J_n}
) \\
&\qquad +\sum_k\lambda_k\wedge p^\top\omega\otimes
\gerstenhaber{\nabla_{l_k}-\interior{l_k}\Xi}{\varphi(
1\otimes\partial^{J_0}\otimes\cdots\otimes\partial^{J_n})}
\Big) \\
=&\ \sigma\big(d_L(p^\top\omega)\otimes 1\big)
\otimes \pbw(\partial^{J_0})\otimes\cdots
\otimes\pbw(\partial^{J_n}) \\
&\qquad +\sum_k \etendu{\sigma}
\Big( p^\top\alpha_k\wedge p^\top\omega
\otimes \gerstenhaber{\cn_{a_k}}
{\varphi(1\otimes\partial^{J_0}\otimes\cdots\otimes
\partial^{J_n})} \Big) \\
=&\ d_A\omega\otimes \pbw(\partial^{J_0})\otimes\cdots
\otimes\pbw(\partial^{J_n}) \\
&\qquad +\sum_k \etendu{\sigma}
\bigg(p^\top(\alpha_k\wedge\omega)\otimes\varphi
\bigg\{ \sum_{t=0}^n
1\otimes\partial^{J_0}\otimes\cdots\otimes
\gerstenhaber{\cn_{a_k}}{\partial^{J_t}}\otimes\cdots\otimes\partial^{J_n}
\bigg\} \bigg) \\
=&\ d_A\omega\otimes \pbw(\partial^{J_0})\otimes\cdots
\otimes\pbw(\partial^{J_n}) \\
&\qquad + \sum_k \sum_{t=0}^n
\alpha_k\wedge\omega\otimes
\pbw(\partial^{J_0})\otimes\cdots\otimes
\pbw\big(\po\gerstenhaber{\cn_{a_k}}
{\partial^{J_t}}\big)\otimes\cdots\otimes\pbw(\partial^{J_n})
.\end{align*}
It follows from Lemma~\ref{Vincennes} that
\[ \pbw\big(\po\gerstenhaber{\cn_{a_k}}{\partial^{J_t}}\big)
= \pbw\big(\cn_{a_k}(\partial^{J_t})\big) = a_k\cdot\pbw(\partial^{J_t}) .\]
Hence, we conclude that
$\etendu{\sigma}\circ\gerstenhaber{\perturbation}{\argument}\circ\etendu{\tau}=\dau$.
\end{proof}

\begin{proof}[Proof of Proposition~\ref{Dhaka}]
We proceed by homological perturbation
--- see Lemma~\ref{HPL} and also~\cite[Lemma~A.1]{MR4150934}.
Starting from the filtered contraction of Proposition~\ref{Damas},
it suffices to perturb the coboundary operator $\gerstenhaber{-\delta+m}{\argument}$
by the operator $\gerstenhaber{\perturbation}{\argument}$.
One checks that
$\etendu{\sigma}\circ\gerstenhaber{\perturbation}{\argument}\circ\etendu{h}=0$.
Therefore, we obtain
\[ \etendu{\perturbed{\sigma}}:=\sum_{l=0}^\infty
\etendu{\sigma}\circ(\gerstenhaber{\perturbation}{\argument}\circ\etendu{h})^l
=\etendu{\sigma} \]
and, making use of Lemma~\ref{Pantin},
\[ \vartheta:=\sum_{l=0}^\infty
\etendu{\sigma}\circ(\gerstenhaber{\perturbation}{\argument}\circ\etendu{h})^l\circ
\gerstenhaber{\perturbation}{\argument}\circ\etendu{\tau}=\etendu{\sigma}\circ
\gerstenhaber{\perturbation}{\argument}\circ\etendu{\tau}=\dau .\]
The result follows immediately since
$-\delta+\perturbation=Q$.

As in the proof of Proposition~\ref{Bagnolet},
applying Lemmas~\ref{lem:polydsemifull} and~\ref{lem:semifullstable},
we conclude that $(\etendu{\perturbed{\tau}},\etendu{\sigma},\etendu{\perturbed{h}})$
is a semifull algebra contraction. Since the differential
$\gerstenhaber{Q+m}{\argument}=\gerstenhaber{Q}{\argument}+\hochschild$
is a derivation with respect to the cup product\footnote{Although
the Gerstenhaber bracket $\gerstenhaber{\argument}{\argument}$
is not a biderivation with respect to the cup product,
but only a biderivation up to homotopy,
it is known that $\gerstenhaber{m}{\argument}=\hochschild$
is indeed a derivation with respect to the cup product
--- see~\cite[Equation~(20)]{MR0161898}.},
according to Lemma~\ref{lem:semifulltau},
we conclude that $\etendu{\perturbed{\tau}}$ is an algebra morphism.
The fact that $\sigma_\natural$ is an algebra morphism
is already contained in Lemma~\ref{lem:polydsemifull}.
\end{proof}

The next proposition gives an alternative characterization of the
map $\etendu{\perturbed{\tau}}$ as the solution of an initial value problem.

\begin{proposition}\label{Mandalay2}
Given $x\in\wa\otimes_R\Dpoly{\diamond}$
and $y\in\sections{\Lambda^\bullet L\dual}
\otimes_R\verticalDpoly{\diamond}$,
we have
\[ \etendu{\perturbed{\tau}}(x)=y
\qquad\text{if and only if}\qquad
\begin{cases}
\etendu{h}(y) = 0 \\
\etendu{h}\big(\gerstenhaber{Q}{y}\big) = 0 \\
\etendu{\sigma}(y) = x
\end{cases}
\quad\text{if and only if}\qquad
\begin{cases}
\gerstenhaber{\eta}{y} = 0 \\
\gerstenhaber{\bracket{\eta}{Q}}{y} = 0 \\
\etendu{\sigma}(y) = x
\end{cases} \]
\end{proposition}
The derivation $\eta$ was defined in Equation~\eqref{derivationeta}.
The proof of Proposition~\ref{Mandalay2} is similar
to the proof of Proposition~\ref{Mandalay} and is therefore omitted.

\begin{proposition}\label{taucoalgebra}
The restriction of the map $\etendu{\perturbed{\tau}}$ of Corollary\ref{cor:Dhaka}
to differential (rather than polydifferential) operators
is a morphism of coalgebras
\[ \etendu{\perturbed{\tau}} : \wa\otimes_R\Dpoly{0} \to
\sections{\Lambda^\bullet L\dual}\otimes_R\verticalDpoly{0} .\]
\end{proposition}

\begin{proof}
Since, according to Proposition~\eqref{Dhaka}, $\etendu{\perturbed{\tau}}$
respects the cup products, we have the commutative diagram
\[ \begin{tikzcd}[column sep=large]
\bigotimes_{\wa}^2 \big(\wa\otimes_R\Dpoly{0}\big)
\arrow[r, "\etendu{\perturbed{\tau}}\otimes\etendu{\perturbed{\tau}}"]
\arrow[d, "\smile"', "\cong"] &
\bigotimes_{\sections{\Lambda^\bullet L\dual}}^2
\big(\sections{\Lambda^\bullet L\dual}\otimes_R\Dpoly{0}\big)
\arrow[d, "\smile", "\cong"'] \\
\wa\otimes_R\Dpoly{1} \arrow[r, "\etendu{\perturbed{\tau}}"'] &
\sections{\Lambda^\bullet L\dual}\otimes_R\Dpoly{1}
\end{tikzcd} \]
in which, owing to the very definition of the cup products,
the two vertical arrows are isomorphisms.

Denoting by $\breve{\Delta}$ the composition of the comultiplication
and the cup product, we are thus led to show that the diagram
\[ \begin{tikzcd}
\wa\otimes_R\Dpoly{0} \arrow[r, "\etendu{\perturbed{\tau}}"]
\arrow[d, "\breve{\Delta}"'] &
\sections{\Lambda^\bullet L\dual}\otimes_R\Dpoly{0}
\arrow[d, "\breve{\Delta}"] \\
\wa\otimes_R\Dpoly{1} \arrow[r, "\etendu{\perturbed{\tau}}"'] &
\sections{\Lambda^\bullet L\dual}\otimes_R\Dpoly{1}
\end{tikzcd} \]
commutes.

By virtue of Proposition~\ref{Mandalay2}, it suffices to show that the three identities
\[ \gerstenhaber{\eta}{\breve{\Delta}\circ\etendu{\perturbed{\tau}}(x)}=0 ,\qquad
\gerstenhaber{\bracket{\eta}{Q}}{\breve{\Delta}\circ\etendu{\perturbed{\tau}}(x)}=0,
\qquad\text{and}\qquad
\etendu{\sigma}\big(\breve{\Delta}\circ\etendu{\perturbed{\tau}}(x)\big)=\breve{\Delta}(x) \]
hold for every $x\in\wa\otimes_R\Dpoly{0}$.

Consider the dg Hopf algebroid $\enveloping{\cF}$
arising from the Fedosov dg Lie algebroid $\cF\to\cM$.
Given any $b\in\sections{\cF}=\sections{\Lambda^\bullet L\dual\otimes\hat{S}(B\dual)\otimes B}$
and $u\in\enveloping{\cF}=\sections{\Lambda^\bullet L\dual\otimes\hat{S}(B\dual)\otimes SB}$,
we have
\begin{multline*}
\gerstenhaber{b}{\breve{\Delta}(u)}
= b\star\Big(\sum_{(u)}u_{(1)}\smile u_{(2)}\Big)
-\Big(\sum_{(u)}u_{(1)}\smile u_{(2)}\Big)\star b \\
= \sum_{(u)}\big((b\circ u_{(1)})\smile u_{(2)}+u_{(1)}\smile(b\circ u_{(2)})\big) \\
\qquad -\sum_{(u)}\big((u_{(1)}\circ b)\smile u_{(2)}+u_{(1)}\smile(u_{(2)}\circ b)\big) \\
= \breve{\Delta}(b\circ u)-\breve{\Delta}(u\circ b)
= \breve{\Delta}(\gerstenhaber{b}{u})
.\end{multline*}
This fact together with Proposition~\ref{Mandalay2} immediately implies that
\begin{gather*}
\gerstenhaber{\eta}{\breve{\Delta}\big(\etendu{\perturbed{\tau}}(x)\big)}
=\breve{\Delta}\big(\gerstenhaber{\eta}{\etendu{\perturbed{\tau}}(x)}\big)=0 \\ \intertext{and}
\gerstenhaber{\bracket{\eta}{Q}}{\breve{\Delta}\big(\etendu{\perturbed{\tau}}(x)\big)}
=\breve{\Delta}\big(\gerstenhaber{\bracket{\eta}{Q}}{\etendu{\perturbed{\tau}}(x)}\big)=0
.\end{gather*}
Furthermore, since $\etendu{\sigma}$ is a morphism of coalgebras and
$\etendu{\sigma}\circ\etendu{\perturbed{\tau}}=\id$, we have
\[ \etendu{\sigma}\circ\breve{\Delta}\circ\etendu{\perturbed{\tau}}(x)
= \breve{\Delta}\circ\etendu{\sigma}\circ\etendu{\perturbed{\tau}}(x)
= \breve{\Delta}(x) .\]
The proof is complete.
\end{proof}

Finally, we have the following
\begin{proposition}\label{thm:Naples}
Given a Lie pair $(L,A)$, each choice of a splitting $j:B\to L$
of the short exact sequence of vector bundles $0\to A \to L \to B \to 0$
and of a torsion-free $L$-connection $\nabla$ on $B$ determines
\begin{enumerate}
\item an $L_\infty$ algebra structure on
$\tot\big(\wa\otimes_R\Dpoly{\bullet}\big)$
with the operator $\dhoch$ as unary bracket;
\item and a Gerstenhaber algebra structure on $\hypercohomology^\bullet_{\CE}(A,\Dpoly{\bullet})$,
the cohomology of the total complex
\[ \Big(\tot\big(\wa\otimes_R\Dpoly{\bullet}\big),\dhoch \Big) ,\]
where the Lie bracket is induced by the binary bracket
of the $L_\infty$ algebra structure on
$\tot\big(\wa\otimes_R \Dpoly{\bullet}\big)$
and the multiplication by the cup product \eqref{eq:cup01}.
\end{enumerate}
\end{proposition}

\begin{proof}
Applying the homotopy transfer theorem for $L_\infty$ algebras
\cite{MR3276839,MR1932522,MR1950958}
to the $L_\infty$ algebra obtained in Proposition~\ref{Bujumbura}
and the contraction obtained in Proposition~\ref{Dhaka},
we get an induced $L_\infty$ algebra structure on
$\tot\big(\wa\otimes_R\Dpoly{\bullet}\big)$,
where the unary bracket is the differential $\dhoch$.
This proves~(1).

For~(2), we notice that at the level of cohomology $\etendu{\perturbed{\tau}}$
and $\sigma_\natural$ induce isomorphisms of graded spaces
which are compatible with both the induced graded Lie algebra structures
(by construction, since these are related via homotopy transfer
along $(\etendu{\perturbed{\tau}},\etendu{\sigma},\etendu{\perturbed{h}})$)
and the induced graded associative algebra structures (by Proposition~\ref{Dhaka}).
It follows at once that the induced graded Lie algebra
and graded associative algebra structures make
$\hypercohomology^\bullet_{\CE}(A,\Dpoly{\bullet})$
into a Gerstenhaber algebra, since the same is true for
$\hypercohomology^\bullet\big(\tot\big(\sections{\Lambda^\bullet L\dual}
\otimes_R\verticalDpoly{\bullet}\big),\gerstenhaber{Q+m}{\argument}\big)$,
according to Proposition~\ref{Bujumbura}. In particular,
this shows that the cup product on $\hypercohomology^\bullet_{\CE}(A,\Dpoly{\bullet})$
is graded commutative,
cf.\ the discussion in Remark~\ref{graded_commutativity}.
\end{proof}

\subsection{Uniqueness of the \texorpdfstring{$L_\infty$}{L∞} structure}\label{sec:uniqueness}

A priori, the Gerstenhaber algebra structures on
$\hypercohomology^\bullet_{\CE}(A,\Tpoly{\bullet})$
and $\hypercohomology^\bullet_{\CE}(A,\Dpoly{\bullet})$
in Propositions~\ref{Bari} and~\ref{thm:Naples} are not canonical,
as their constructions depend on a choice of a splitting $j:B\to L$
of the short exact sequence $0\to A\to L\to B\to 0$
and a torsion-free $L$-connection $\nabla$ on $B$.
The aim of this section is to complete the proof of Theorem~\ref{Macau}
from the introduction and show that both Gerstenhaber algebras are indeed canonical.

As observed at the end of Section~\ref{Fedosov_dg_abd},
the Fedosov dg manifolds arising from different choices of a splitting
and a connection are isomorphic with each other
(and we made the isomorphism explicit in terms of the associated PBW maps).
There are induced isomorphisms between the Fedosov dg Lie algebroids,
hence between the corresponding algebras of polyvector fields
and polydifferential operators from Propositions~\ref{pro:Lyon} and~\ref{Bujumbura}.
We can make these isomorphisms explicit, once again in terms of the associated PBW maps.
Throughout the present section we shall concentrate on the (harder) case
of polydifferential operators --- the proof for the case of polyvector fields
is similar (see also \cite{MR4091493}).

We consider two different choices $j_1,\nabla_1$ and $j_2,\nabla_2$
of a splitting of the short sequence $0\to A\to L\to B\to 0$
and a torsion-free $L$-connection on $B$,
together with the induced homological vector fields $Q_1$ and $Q_2$
on $\cM=L[1]\oplus B$, as in Theorem~\ref{strawberry},
and the induced PBW isomorphisms
$\operatorname{pbw}_1$ and $\operatorname{pbw}_2:\sections{SB}
\to\frac{\enveloping{L}}{\enveloping{L}\sections{A}}$,
as in the discussion preceding Theorem~\ref{thm:EW}.

Recall the isomorphism of $R$-coalgebras
\[ \psi:=\pbw_1\inv\circ\pbw_2:\sections{SB}\to\sections{SB}\]
and the dual isomorphism of $R$-algebras
$\psi\dual:\sections{\hat{S}(B\dual)}\to\sections{\hat{S}(B\dual)}$
introduced at the end of Section~\ref{Fedosov_dg_abd}.
There is an induced isomorphism
$\psi_*:\verticalDpoly{\bullet}\to\verticalDpoly{\bullet}$
between the spaces of polydifferential operators,
sending a polydifferential operator $D\in\verticalDpoly{k}$
to the one $\psi_*(D)\in\verticalDpoly{k}$ defined by
\[ \psi_*(D)\left(\chi^{I_0},\ldots,\chi^{I_k}\right)
= \psi\dual\Big( D\Big((\psi\dual)\inv\left(\chi^{I_0}\right),
\ldots,(\psi\dual)\inv\left(\chi^{I_k}\right)\Big) \Big) \]
for all $\chi^{I_0},\ldots\chi^{I_k}\in\sections{\hat{S}(B\dual)}$.
By construction, $\psi_*$ is compatible with the Gerstenhaber bracket
of polydifferential operators, and in fact
\begin{multline}\label{eq:psi*}
\id\otimes\psi_*:\big( \tot\big(\sections{\Lambda^\bullet L\dual}
\otimes_R\verticalDpoly{\bullet}\big), \gerstenhaber{Q_1+ m}{\argument},
\gerstenhaber{\argument}{\argument} \big) \\
\to \big( \tot\big(\sections{\Lambda^\bullet L\dual}
\otimes_R\verticalDpoly{\bullet}\big), \gerstenhaber{Q_2+ m}{\argument},
\gerstenhaber{\argument}{\argument} \big)
\end{multline}
is an isomorphism of dglas.

We shall need the following lemma.
\begin{lemma}\label{lemma:Rom}
Under the identification $\varphi$ from~\eqref{eq:varphi}, the isomorphism $\psi_*$ satisfies
\[ \varphi\inv\circ\psi_*\circ\varphi(\chi^I\otimes\partial_{J_0}\otimes\cdots\otimes\partial_{J_k})
=\chi^I\otimes\psi\inv(\partial_{J_0})\otimes\cdots\otimes\psi\inv(\partial_{J_k})+\xi, \]
where $\xi\in\hat{S}^{>\length{I}}(B\dual)\otimes S(B)^{\otimes k+1}$.
\end{lemma}

\begin{proof}
In terms of any pair of dual local frames
$\{\chi^I\}_{I\in\mathbb{N}^r}$ and
$\{\partial_{J}\}_{J\in\mathbb{N}^r}$
for $\hat{S}(B\dual)$ and $S(B)$ respectively,
the isomorphisms $\psi\inv:\sections{SB}\to\sections{SB}$ and
$(\psi\dual)\inv:\sections{\hat{S}(B\dual)}\to\sections{\hat{S}(B\dual)}$
are given by
\[ \psi\inv(\partial_J) = \sum_{K\in\mathbb{N}^r} \frac{1}{K!}\psi^K_J \partial_K
\qquad\text{and}\qquad
(\psi\dual)\inv(\chi^I) = \sum_{K\in\mathbb{N}^r} \frac{1}{K!} \psi^I_K \chi^K, \]
where the $\psi^I_J$ are smooth functions on the base manifold
(more precisely, on the open subset $U\subset M$ on which the local frames are defined).

Let $D$ and $D'\in\verticalDpoly{k}$ be the polydifferential
operators defined by
\[ D:=\psi_*\circ\varphi(\chi^I\otimes\partial_{J_0}
\otimes\cdots\otimes\partial_{J_k})
\qquad\text{and}\qquad
D':=\varphi(\chi^I\otimes\psi\inv(\partial_{J_0})
\otimes\cdots\otimes\psi\inv(\partial_{J_k})) .\]
We have to show that their difference $D-D'$
sends $\sections{\hat{S}(B\dual)}^{\otimes k+1}$
into $\sections{\hat{S}^{>\length{I}}(B\dual)}$.

For all $I,J,K\in\NN^r$, we have
\[ \psi\dual(\chi^I) = \chi^{I}
+\text{terms in }\sections{S^{>\length{I}}(B\dual)} \]
and
\[ \partial_J(\chi^K) = \begin{cases} (K-J)!\cdot\chi^{K-J}
& \text{if}\ J\prec K, \\ 0 & \text{otherwise.} \end{cases} \]

The partial order $\prec$ on $\NN^r$ is defined as follows:
$(j_1,j_2,\cdots,j_r)\prec(k_1,k_2,\cdots,k_r)$ if and only if
$j_p\leqslant k_p$ for each $p\in\{1,2,\cdots,r\}$.

It follows that, for all
$\chi^{I_0},\ldots,\chi^{I_k}\in\sections{\hat{S}(B\dual)}$,
\[ \begin{split}
D(\chi^{I_0},\ldots,\chi^{I_K}) & = \psi\dual\left( \chi^I\cdot \partial_{J_0}((\psi\dual)
\inv(\chi^{I_0}))\cdots\partial_{J_K}((\psi\dual)\inv(\chi^{I_K})) \right) \\
& = \psi\dual\left(\sum_{K_0,\ldots,K_k\in\mathbb{N}^r}
\frac{1}{K_0!\cdots K_k!} \psi^{I_0}_{K_0}\cdots\psi^{I_k}_{K_k}
\chi^I\cdot\partial_{J_0}(\chi^{K_0})\cdots\partial_{J_k}(\chi^{K_k})\right) \\
& = \psi^{I_0}_{J_0}\cdots\psi^{I_k}_{J_k}\chi^{I}
+\text{terms in }\sections{S^{>\length{I}}(B\dual)}
,\end{split} \]
while
\[ \begin{split}
D'(\chi^{I_0},\ldots,\chi^{I_K})
& = \chi^I\cdot\psi\inv(\partial_{J_0})
(\chi^{I_0})\cdots\psi\inv(\partial_{J_K})(\chi^{I_K}) \\
& = \sum_{K_0,\ldots,K_k\in\mathbb{N}^r} \frac{1}{K_0!\cdots K_k!}
\psi^{K_0}_{J_0}\cdots\psi^{K_k}_{J_k}
\chi^I\cdot\partial_{K_0}(\chi^{I_0})\cdots\partial_{K_k}(\chi^{I_k}) \\
& = \psi^{I_0}_{J_0}\cdots\psi^{I_k}_{J_k}\chi^{I}\,
+\,\text{terms in }\sections{S^{>\length{I}}(B\dual)}
.\end{split} \]
This concludes the proof of the lemma.
\end{proof}

With these preparations, we are finally ready to complete
the proof of Theorem~\ref{Macau} from the introduction.
In light of Propositions~\ref{Bari} and~\ref{thm:Naples},
the only thing which remains to be shown is the following

\begin{proposition}\label{prop:uniqueness}
The $L_\infty$ algebra structures on
$\tot\big(\wa\otimes_R\Tpoly{\bullet}\big)$ and
$\tot\big(\wa\otimes_R \Dpoly{\bullet}\big)$
from Propositions~\ref{Bari} and~\ref{thm:Naples}, respectively, are independent
of the involved choices up to an $L_\infty$ isomorphism
with linear part the identity map.
In particular, the induced Gerstenhaber algebra structures
on $\hypercohomology^\bullet_{\CE}(A,\Tpoly{\bullet})$
and $\hypercohomology^\bullet_{\CE}(A,\Dpoly{\bullet})$
are independent of the involved choices.
\end{proposition}

\begin{proof}
We shall prove the proposition in detail for the $L_\infty$ algebra structure
on $\tot\big(\wa\otimes_R\Dpoly{\bullet}\big)$.
The claim for $\tot\big(\wa\otimes_R\Tpoly{\bullet}\big)$
can be proved by a similar reasoning, or by comparison with the results from~\cite{MR4091493},
where in fact a stronger result is proven: the $L_\infty$ algebra structure
on $\tot\big(\wa\otimes_R\Tpoly{\bullet}\big)$
is independent of the choice of $\nabla$ altogether,
and is independent of the choice of $j$ up to an $L_\infty$
isomorphism with linear part the identity map
(cf.~\cite[Propositions~4.9 and~4.17]{MR4091493}).

Let $j_1,\nabla_1$ and $j_2,\nabla_2$ be two choices of a splitting and a connection.
Each choice $j_i,\nabla_i$ (with $i\in\{1,2\}$) determines a homological vector field
$Q_i$ on $L[1]\oplus B$, a Poincar\'e--Birkhoff--Witt isomorphism
$\pbw_i:\sections{SB}\to\frac{\enveloping{L}}{\enveloping{L}\sections{A}}$,
and a Dolgushev--Fedosov contraction
\begin{equation}\label{Kigali} \begin{tikzcd}[cramped]
\Big(\tot\big(\wa\otimes_R\Dpoly{\bullet}\big),\dhoch \Big)
\arrow[r, shift left, "\perturbed{\tau}_{\natural,i}"] &
\Big(\tot\big(\sections{\Lambda^\bullet L\dual}\otimes_R\verticalDpoly{\bullet}\big),
\gerstenhaber{Q_i+ m}{\argument}\Big)
\arrow[l, shift left, "\sigma_{\natural,i}"]
\arrow[loop, "\perturbed{h}_{\natural,i}", out=5,in=-5,looseness = 3]
\end{tikzcd} \end{equation}
as in Proposition~\ref{Dhaka}.

Together with an $L_\infty$ algebra structure $\Upsilon_i$ on
$\Big(\tot\big(\wa\otimes_R\Dpoly{\bullet}\big),\dhoch \Big)$,
homotopy transfer along the Dolgushev--Fedosov contraction \eqref{Kigali}
induces a pair of $L_\infty$ quasi-isomorphisms
\begin{gather*} \tot\big(\wa\otimes_R\Dpoly{\bullet}\big)\rightsquigarrow
\tot\big(\sections{\Lambda^\bullet L\dual}\otimes_R\verticalDpoly{\bullet}\big) \\
\tot\big(\sections{\Lambda^\bullet L\dual}\otimes_R\verticalDpoly{\bullet}\big)
\rightsquigarrow \tot\big(\wa\otimes_R\Dpoly{\bullet}\big) \end{gather*}
with linear parts $\perturbed{\tau}_{\natural,i}$ and $\sigma_{\natural,i}$ respectively.
Recall the isomorphism of dglas
\begin{multline*}
\big(\tot\big(\sections{\Lambda^\bullet L\dual}\otimes_R\verticalDpoly{\bullet}\big),
\gerstenhaber{Q_1+m}{\argument},
\gerstenhaber{\argument}{\argument}\big)
\xto{\id\otimes\psi_*} \\
\big(\tot\big(\sections{\Lambda^\bullet L\dual}
\otimes_R\verticalDpoly{\bullet}\big),
\gerstenhaber{Q_2+m}{\argument},
\gerstenhaber{\argument}{\argument}\big)
\end{multline*}
defined in~\eqref{eq:psi*}.
The two induced $L_\infty$ algebra structures
on $\tot\big(\wa\otimes_R\Dpoly{\bullet}\big)$
are related by an $L_\infty$ morphism
\[ F\:\colon\:\Big( \tot\big(\wa\otimes_R\Dpoly{\bullet}\big),
\Upsilon_1\Big)\rightsquigarrow \Big( \tot\big(\wa
\otimes_R\Dpoly{\bullet}\big), \Upsilon_2 \Big) \]
with linear part $f_1=\sigma_{\natural,2}\circ(\id\otimes\psi_*)\circ\perturbed{\tau}_{\natural,1}$.

In order to conclude the proof, we only need to show that $f_1$ is the identity map.
In fact, since an $L_\infty$ morphism $F$ is an isomorphism of $L_\infty$ algebras
if and only if its linear part $f_1$ is an isomorphism between the underlying tangent complexes,
this will show that $F$ is an $L_\infty$ isomorphism.
Moreover, once we have proven that $f_1$ is the identity map,
the second claim will also follow.
In fact, denoting by $\gerstenhaber{\argument}{\argument}_i$ the Lie bracket
on $\hypercohomology^\bullet_{\CE}(A,\Dpoly{\bullet})$ induced
by the quadratic bracket of $\Upsilon_i$, $i=1,2$, in order to conclude the proof
we only need to show
$\gerstenhaber{\argument}{\argument}_1=\gerstenhaber{\argument}{\argument}_2$,
since the associative product (the cup product) is already independent
of the choices at the cochain level: but, $F$ being an $L_\infty$ morphism,
its linear part $f_1$ commutes with the quadratic brackets of $\Upsilon_1$, $\Upsilon_2$,
up to the homotopy $f_2$, thus the induced map on cohomology intertwines the two brackets
$\gerstenhaber{\argument}{\argument}_1$ and $\gerstenhaber{\argument}{\argument}_2$.

Recall the identification $\varphi$ from~\eqref{eq:varphi}, the map
$\sigma:\Lambda^\bullet L\dual\otimes\hat{S}(B\dual)
\to\Lambda^\bullet A\dual$
from~\eqref{Eq:sigma} and the commutative diagram
\[ \begin{tikzcd}[row sep=tiny]
\sections{\Lambda^{\bullet}L\dual}\otimes_R\verticalDpoly{k}
\arrow{dd}{\isomorphism}[swap]{\id\otimes\varphi\inv}
\arrow{dr}{\sigma_{\natural,i}} & \\
& \sections{\Lambda^{\bullet}A\dual}\otimes_R\Dpoly{k} \\
\sections{\Lambda^{\bullet}L\dual\otimes \hat{S}(B\dual)}
\otimes_R\sections{(SB)^{\otimes k+1}}
\arrow{ur}[swap]{\sigma\otimes\pbw_i^{\otimes k+1}} &
\end{tikzcd} \]
defining $\sigma_{\natural,i}$.
According to Lemma~\ref{lemma:Rom}, the maps
$\varphi\inv\circ\psi_*$ and
$\left(\id\otimes(\psi\inv)^{\otimes k+1}\right)\circ\varphi\inv$
from $\verticalDpoly{k}$
to $\sections{\hat{S}(B\dual)\otimes S(B)^{\otimes k+1}}$
coincide up to terms in
$\sections{\hat{S}^{\geqslant 1}(B\dual)\otimes S(B)^{\otimes k+1}}$.
Therefore, we obtain
\[ \sigma_{\natural,2}\circ(\id\otimes\psi_*)
= (\sigma\otimes\pbw_2^{\otimes k+1})\circ(\id\otimes(\varphi\inv\circ\psi_*))
= (\sigma\otimes (\pbw_2\circ\psi\inv)^{\otimes k+1})\circ(\id\otimes\varphi\inv)
= \sigma_{\natural,1} \]
since $\sections{\Lambda^{\bullet}L\dual\otimes\hat{S}^{\geqslant 1}(B\dual)}\subset\ker(\sigma)$ and $\psi=\pbw_1\inv\circ\pbw_2$.
It follows that
\[ f_1=\sigma_{\natural,2}\circ(\id\otimes\psi_*)
\circ\perturbed{\tau}_{\natural,1}
=\sigma_{\natural,1}\circ\perturbed{\tau}_{\natural,1}=\id
.\qedhere \]
\end{proof}

\section{Matched pair case}

This section is devoted to the proof of Theorem~\ref{thm:A},
which was stated in the introduction.
See Theorems~\ref{thm:Bonn} and~\ref{thm:Berlin} below.

Let $(L,A)$ be a Lie pair with quotient $B:=L/A$.
Recall that, if a splitting $j:B\to L$ of
the short exact sequence $0\to A\xto{i} L\to B\to 0$ is given,
whose image $j(B)$ happens to be a Lie subalgebroid of $L$,
then $A$ and $B$ are said to form a
\emph{matched pair of Lie algebroids}
--- see~\cite{MR1460632, MR1716681} for more details.
In such a situation, we write $L=A\bowtie B$ to stress
that $A$ and $B$ --- more precisely $i(A)$ and $j(B)$ ---
play symmetric roles as a pair of complementary
Lie subalgebroids of the Lie algebroid $L$.
In the case of matched pairs, the algebraic structures
on the space of polyvector fields and the space of
polydifferential operators described in Section~\ref{Chatanooga}
reduce to the natural ones described in Theorem~\ref{thm:A}.

\subsection{Dg Lie algebroid arising from a matched pair}

Let $L=A\bowtie B$ be a matched pair of Lie algebroids over a manifold $M$.
Consider the double vector bundle
\[ \begin{tikzcd}[column sep=small,row sep=small]
A\oplus B \arrow[r] \arrow[d] & B \arrow[d, "\varpi"] \\
A \arrow[r, "\pi"'] & M
\end{tikzcd} \]
where the vector bundle $A\oplus B\to A$ is the pullback of the vector bundle
$B\xto{\varpi}M$ via the map $\pi:A\to M$,
while the vector bundle $A\oplus B\to B$ is the pullback of the vector bundle
$A\xto{\pi}M$ via the map $\varpi:B\to M$.

Each section $b\in\sections{B}$ determines a derivation $\vec{b}$
of the algebra of smooth functions $C^\infty(A)$ through the relations
\[ \vec{b}(\pi^* f)=\pi^*\big(\anchor_{b}f\big), \quad\forall f\in C^\infty(M) \]
and \[ \vec{b}(l_\xi)=l_{\nabla^{\Bott}_b \xi}, \quad\forall \xi\in\sections{A\dual} ,\]
where $l_\xi$ denotes the fiberwise linear function
$A\ni a\mapsto\duality{\xi}{a}\in\RR$ on $A$.

The vector bundle $A\oplus B\to A$, whose space of sections is naturally
identified to $C^\infty(A)\otimes_{C^\infty(M)}\sections{B}$,
admits a natural Lie algebroid structure with anchor map
\[ C^\infty(A)\otimes_{C^\infty(M)}\sections{B} \ni g\otimes b
\mapsto g\cdot\vec{b}\in\XX(A), \quad\forall g\in C^\infty(A), b\in\sections{B} \]
and Lie bracket
\[ \bracket{g_1\otimes b_1}{g_2\otimes b_2} = g_1 g_2\otimes\bracket{b_1}{b_2}
+g_1\cdot\vec{b_1}(g_2)\otimes b_2 -g_2\cdot\vec{b_2}(g_1)\otimes b_1,
\quad \forall g_1, g_2\in C^\infty(A), b_1, b_2\in \sections{B} .\]
Similarly, the vector bundle $A\oplus B\to B$ admits a natural Lie algebroid structure.
These two Lie algebroid structures on $A\oplus B$ are known to be
compatible in the following sense:
\begin{lemma}[Mackenzie~\cite{MR2831518}]
If $A\bowtie B$ is a matched pair of Lie algebroids, then
\[ \begin{tikzcd}[column sep=small,row sep=small]
A\oplus B \arrow{r} \arrow{d} & B \arrow{d} \\
A \arrow{r} & M
\end{tikzcd} \]
is a double Lie algebroid.
\end{lemma}

According to Voronov~\cite{MR2971727}, any double Lie algebroid induces
a pair of dg Lie algebroids. As an immediate consequence, we have
the following

\begin{corollary}
If $A\bowtie B$ is a matched pair of Lie algebroids,
then $(A[1]\oplus B, d^{\Bott}_A )$ is a dg Lie algebroid over $(A[1], \dace)$.
\end{corollary}

Here the dg manifold structures on $(A[1]\oplus B,\dabott)$
and $(A[1],\dace)$ are induced, respectively,
from the Lie algebroid structures
on $A\oplus B\to B$ and $A\to M$ according to Va\u{\i}ntrob's
theorem \cite{MR1480150} --- see Example~\ref{example:Vaintrob_dg_vs_Lie_abd}.
In what follows, we write $\cB$ to denote
the dg manifold $(A[1]\oplus B,\dabott)$ and
$A[1]$ to denote the dg manifold $(A[1],\dace)$.

The space of sections of the dg Lie algebroid $\cB\to A[1]$
can be identified naturally with
$\sections{\Lambda^\bullet A\dual\otimes B}$.
Then, the Lie bracket on $\sections{\Lambda^\bullet A\dual\otimes B}$ is
\begin{equation}\label{eq:bracket}
\bracket{\xi_1\otimes b_1}{\xi_2\otimes b_2}
=\xi_1\wedge\xi_2\otimes\bracket{b_1}{b_2}
+\xi_1\wedge(\nabla^{\Bott}_{b_1}\xi_2)\otimes b_2
-(\nabla^{\Bott}_{b_2}\xi_1)\wedge\xi_2\otimes b_1
\end{equation}
for all $\xi_1,\xi_2\in\wa$ and $b_1,b_2\in\sections{B}$,
while the anchor map
\[ \sections{\Lambda^\bullet A\dual\otimes B}
\xto{\baranchor}\Der\big(\wa\big) \]
is characterized by the relation
\begin{equation}\label{eq:anchor}
\baranchor_{\xi\otimes b}(\eta)=\xi\wedge\nabla^{\Bott}_b\eta
,\end{equation}
for all $\xi,\eta\in\wa$
and $b\in\sections{B}$.
Finally, the differential on the space of sections
of $\cB\to A[1]$ induced by
the homological vector fields on $\cB$ and $A[1]$ is simply the
Chevalley--Eilenberg differential
\[ \dabott: \sections{\Lambda^\bullet A\dual\otimes B}
\to \sections{\Lambda^{\bullet+1}A\dual\otimes B} \]
corresponding to the Bott representation of $A$ on $B$.

\subsection{Fedosov dg manifolds associated with matched pairs}

The identification $L=A\oplus B$ induces a decomposition
\begin{equation}\label{eq:DL}
\sections{\Lambda^n L\dual}=\bigoplus_{\substack{p+q=n\\ p,q\geqslant 0}}
\sections{\Lambda^p A\dual\otimes\Lambda^q B\dual},
\quad n\geqslant 0.
\end{equation}

Denote by
\[ d_L:\sections{\Lambda^\bullet L\dual}
\to\sections{\Lambda^\bullet L\dual} \]
the Chevalley--Eilenberg differential of the Lie algebroid $L$
for cochains with trivial coefficients.

Since $A$ and $B$ play symmetric roles as
a pair of complementary Lie subalgebroids of the Lie algebroid $L$,
we have a pair of Bott connections: the Bott $A$-connection on
$B$ and the Bott $B$-connection on $A$, both denoted by $\nabla^{\Bott}$
by abuse of notations.
Denote by
\[ d_A^{\Bott}:
\sections{\Lambda^\bullet A\dual\otimes\Lambda^\diamond B\dual}
\to\sections{\Lambda^{\bullet+1}A\dual\otimes\Lambda^\diamond B\dual} \]
the Chevalley--Eilenberg differential of the Lie algebroid $A$
for cochains with coefficients in the $A$-module $\Lambda B\dual$
--- the implicit flat $A$-connection $\nabla^{\Bott}$
on $\Lambda B\dual$ is induced from the Bott $A$-connection on $B$.
Similarly, denote by
\[ d_B^{\Bott}:
\sections{\Lambda^\bullet A\dual\otimes\Lambda^\diamond B\dual}\to
\sections{\Lambda^\bullet A\dual\otimes\Lambda^{\diamond+1}B\dual} \]
the Chevalley--Eilenberg differential of the Lie algebroid $B$
for cochains with coefficients in the $B$-module $\Lambda A\dual$
--- the implicit flat $B$-connection $\nabla^{\Bott}$
on $\Lambda A\dual$ is induced from the Bott $B$-connection on $A$.

In order to describe the Fedosov dg manifold arising from the Lie pair $(L, A)$,
we need to choose a torsion-free $L$-connection $\nabla$ on $B$.
Such an $L$-connection on $B$ is completely determined by,
and in fact equivalent to, a torsion-free $B$-connection $\nabla^{1,0}$ on $B$.

The following lemma can be verified by a direct computation.
\begin{lemma}\label{lem:DL}
For a matched pair $L=A\bowtie B$, having identified
$\Lambda L\dual$ with $\Lambda A\dual\otimes\Lambda B\dual$
as in \eqref{eq:DL}, we have
\[ d_L=d_A^{\Bott}+d_B^{\Bott} .\]
Furthermore, the covariant differential $d_L^\nabla$
appearing in Theorem~\ref{strawberry} decomposes as the sum
\[ d^\nabla_L=d^{\Bott}_A+d^{\nabla^{1,0}}_B \]
of
\[ d^{\Bott}_A: \sections{\Lambda^\bullet A\dual
\otimes\Lambda^\diamond B\dual\otimes\hat{S}(B\dual)}\to
\sections{\Lambda^{\bullet+1} A\dual
\otimes\Lambda^\diamond B\dual\otimes\hat{S}(B\dual)} \]
and
\[ d^{\nabla^{1,0}}_B: \sections{\Lambda^\bullet A\dual
\otimes\Lambda^\diamond B\dual\otimes\hat{S}(B\dual)}\to
\sections{\Lambda^\bullet A\dual
\otimes\Lambda^{\diamond+1} B\dual\otimes\hat{S}(B\dual)} .\]
\end{lemma}

Similarly, the $1$-form
$X^\nabla\in\sections{L\dual\otimes\hat{S}^{\geqslant 2}(B\dual)
\otimes B}$
valued in formal vertical vector fields on $B$
constructed in Theorem~\ref{strawberry} decomposes as the sum
\[ X^\nabla=X^{1,0}+X^{0,1} \]
of two formal vertical vector fields
\[ X^{0,1}\in\sections{A\dual\otimes\hat{S}^{\geqslant 2}(B\dual)
\otimes B} \qquad\text{and}\qquad X^{1,0}\in
\sections{B\dual\otimes\hat{S}^{\geqslant 2}(B\dual)\otimes B} .\]

The following lemma is quite obvious --- see~\cite[Section 5]{MR3724780}.

\begin{lemma}\label{Qpq}
Given a matched pair $L=A\bowtie B$, consider the Lie pair $(L,A)$
with the splitting identifying $B$ to a Lie subalgebroid of $L$
complementary to $A$ and choose a torsion-free $B$-connection
$\nabla^{1,0}$ on $B$.
Then, the Fedosov homological vector field $Q$ constructed in
Theorem~\ref{strawberry} is the sum $Q=Q^{1,0}+Q^{0,1}$ of the pair of operators
\begin{gather*}
Q^{1,0}:\sections{\Lambda^\bullet A\dual\otimes
\Lambda^\diamond B\dual\otimes\hat{S}(B\dual)}
\to\sections{\Lambda^\bullet A\dual\otimes\Lambda^{\diamond+1}
B\dual\otimes\hat{S}(B\dual)}
\\ \intertext{and}
Q^{0,1}:\sections{\Lambda^\bullet A\dual\otimes\Lambda^\diamond
B\dual\otimes\hat{S}(B\dual)}
\to\sections{\Lambda^{\bullet+1} A\dual\otimes\Lambda^\diamond
B\dual\otimes\hat{S}(B\dual)}
\end{gather*}
defined by the relations
\[ Q^{1,0}=-\delta+d^{\nabla^{1,0}}_B+X^{1,0}
\qquad\text{and}\qquad Q^{0,1}=d^{\Bott}_A+X^{0,1} \]
and satisfying the relations
\[ Q^{1,0}\circ Q^{1,0}=0, \qquad
Q^{0,1}\circ Q^{0,1}=0, \qquad\text{and}\qquad
Q^{0,1}\circ Q^{1,0}+Q^{1,0}\circ Q^{0,1}=0 .\]
\end{lemma}

We now give a more detailed description of the operators
$Q^{1,0}$ and $Q^{0,1}$, which will be needed later on.

Consider (i) the isomorphism of left $R$-modules
$\pbw:\sections{SB}\to\frac{\enveloping{L}}{\enveloping{L}\sections{A}}$
arising from the Lie pair $(L,A)$ and the $L$-connection $\nabla$ on $B$,
(ii) the isomorphism of left $R$-modules
$\pbwb:\sections{SB}\to\enveloping{B}$
arising from the Lie pair $(B,0)$ and the $B$-connection $\nabla^{1,0}$ on $B$
--- see Equation~\eqref{eq:pbw} ---
and (iii) the natural isomorphism of left $R$-modules
\begin{equation}\label{eq:BLA}
\begin{tikzcd}
\frac{\enveloping{L}}{\enveloping{L}\sections{A}}
\arrow[r, "\isomorphism"] & \enveloping{B}
\end{tikzcd}
.\end{equation}
The following lemma can be verified easily by applying
the PBW iteration formula in \cite{MR2989383,arXiv:1408.2903}
--- see also \cite[\S~3.4]{MR4150934}.

\begin{lemma}\label{pbwB}
Given a matched pair $L=A\bowtie B$, the diagram
\[ \begin{tikzcd}[row sep=tiny]
& \frac{\enveloping{L}}{\enveloping{L}\sections{A}}
\arrow[dd, "\isomorphism"] \\
\sections{SB} \arrow[ru, "\pbw"] \arrow[rd, "\pbwb"'] & \\
& \enveloping{B}
\end{tikzcd} \]
commutes.
\end{lemma}

The flat $L$-connection $\cn$ on $SB$ defined
by Equation~\eqref{eq:cn} gives rise
to a flat $A$-connection on $SB$:
\begin{equation}\label{eq:cnta}
\cnt_a(s)=\cn_{i(a)}s=\pbw\inv\big(i(a)\cdot\pbw(s)\big)
=\pbwb\inv\big(a\star\pbw(s)\big)
\end{equation}
and a flat $B$-connection on $SB$:
\begin{equation}\label{eq:cntb}
\cnt_b(s)=\cn_{j(b)}s=\pbw\inv\big(j(b)\cdot\pbw(s)\big)
=\pbwb\inv\big(b\cdot\pbwb(s)\big)
.\end{equation}
Here $a\in\sections{A}$, $b\in\sections{B}$, $s\in\sections{SB}$.
The symbol $\cdot$ appearing in the r.h.s.\ of
Equation~\eqref{eq:cntb} denotes the multiplication
in $\enveloping{B}$, while the symbol $\star$ appearing
in the r.h.s.\ of Equation~\eqref{eq:cnta} denotes the action
of $A$ on $\enveloping{B}$ induced by the multiplication
in $\enveloping{L}$ and the natural identification
of $\frac{\enveloping{L}}{\enveloping{L}\sections{A}}$
with $\enveloping{B}$.

According to Theorem~\ref{thm:EW}
(see also \cite[Theorem~4.7]{MR4150934}),
the homological vector field $Q$
on the Fedosov dg manifold $L[1]\oplus B$
is the Chevalley–Eilenberg differential \eqref{eq:CEL}
corresponding to the flat $L$-connection on $\hat{S}(B\dual)$
dual to the flat $L$-connection $\cn$ on $SB$
defined by Equation~\eqref{eq:cn}.
Therefore, as an immediate consequence of Lemma~\ref{lem:DL},
we obtain
\begin{corollary}\label{cor:Q1001}
Under the assumptions of Lemma~\ref{Qpq},
\begin{enumerate}
\item the operator $Q^{1,0}$ coincides with the Chevalley--Eilenberg
differential of the Lie algebroid $B$ for cochains with coefficients
in the $B$-module $\Lambda A\dual\otimes\hat{S}(B\dual)$
with the $B$-representation
\[ \nabla_b(\alpha\otimes\varsigma)=\nabla^{\Bott}_b\alpha\otimes\varsigma
+\alpha\otimes\cnt_b\varsigma ,\]
for all $\alpha\otimes\varsigma\in\sections{\Lambda A\dual\otimes\hat{S}(B\dual)}$.
\item and the operator $Q^{0,1}$ coincides
with the Chevalley--Eilenberg differential of the Lie algebroid $A$
for cochains with coefficients in the $A$-module
$\Lambda B\dual\otimes\hat{S}(B\dual)$ with the $A$-representation
\[ \nabla_a(\beta\otimes\varsigma)=\nabla^{\Bott}_a\beta\otimes\varsigma
+\beta\otimes\cnt_a\varsigma .\]
for all
$\beta\otimes s\in\sections{\Lambda B\dual\otimes\hat{S}(B\dual)}$.
\end{enumerate}
Here $\cnt_a$ and $\cnt_b$ are the flat connections introduced
in Equations~\eqref{eq:cnta} and~\eqref{eq:cntb}, respectively.
\end{corollary}

Restricting the operator $Q^{0,1}$
to $\sections{\Lambda^\bullet A\dual\otimes\hat{S}(B\dual)}$
determines a derivation $Q^{0,1}$ of
$\sections{\Lambda^\bullet A\dual\otimes\hat{S} B\dual}$
of degree $+1$ such that $Q^{0,1}\circ Q^{0,1}=0$.
In other words, $Q^{0,1}$ is a homological vector field on
the graded manifold $A[1]\oplus B$.
Hence $(A[1]\oplus B,Q^{0,1})$ is a dg manifold.
Corollary~\ref{cor:Q1001} implies that $(A[1]\oplus B,Q^{0,1})$
is indeed an instance of the Kapranov dg manifolds investigated
in~\cite{MR2989383,arXiv:1408.2903}.

\begin{remark}
Given a complex manifold $X$, let $A=T^{0,1}_X$ and $B= T^{1,0}_X$.
Then $T_X^\CC=A\bowtie B$ is a matched pair of Lie algebroids
over $\CC$. The Bott $T^{0,1}_X$-connection on $T^{1,0}_X$
encodes the holomorphic vector bundle structure of $T^{1,0}_X$;
the (local) sections of $T^{1,0}_X$ which are flat
w.r.t.\ the $T^{0,1}_X$-connection are precisely
the (local) holomorphic sections of $T^{1,0}_X$.
In other words, the Chevalley--Eilenberg differential
associated with the Bott $T^{0,1}_X$-connection on $T^{1,0}_X$
is the Dolbeault operator
\[ \bar{\partial}:\Omega^{0,\bullet}(X,T^{1,0}_X)
\to\Omega^{0,\bullet+1}(X,T^{1,0}_X) .\]
Similarly, the Chevalley--Eilenberg differential associated with
the Bott $T^{1,0}_X$-connection on $T^{0,1}_X$ is the complex
conjugate operator
\[ {\partial}:\Omega^{\bullet,0}(X,T^{0,1}_X)
\to\Omega^{\bullet+1,0}(X,T^{0,1}_X) .\]
To construct a Fedosov dg manifold corresponding to the matched
pair $(T^{0,1}_X,T^{1,0}_X)$, we need a torsion-free
$T_X^\CC$-connection $\nabla$ on $T^{1,0}_X$,
which is necessarily the sum $\nabla=\bar{\partial}+\nabla^{1,0}$
of the Dolbeault operator and a torsion-free $T^{1,0}_X$-connection
$\nabla^{1,0}$ on $T^{1,0}_X$ --- more precisely, we have $d^\nabla=\bar{\partial}+d^{\nabla^{1,0}}$.
The graded manifold underlying this Fedosov dg manifold
is $T_X^\CC[1]\oplus T^{1,0}_X$ with the algebra of functions
\[ C^\infty\big(T_X^\CC[1]\oplus T^{1,0}_X\big)\cong
\bigoplus_{p\geqslant 0,q\geqslant 0}\Omega^{p,q}(X,\hat{S}(T_X^{1,0})\dual)
.\]
Its homological vector field decomposes as the sum
\[ Q=Q^{1,0}+Q^{0,1} \] of two operators
\[ Q^{1,0}:\Omega^{p,q}(X,\hat{S}(T_X^{1,0})\dual)
\to\Omega^{p+1,q}(X,\hat{S}(T_X^{1,0})\dual) \]
and
\begin{equation}\label{eq:Q01}
Q^{0,1}:\Omega^{p,q}(X,\hat{S}(T_X^{1,0})\dual)
\to\Omega^{p,q+1}(X,\hat{S}(T_X^{1,0})\dual)
\end{equation}
given by
\[ Q^{1,0}=-\delta+d^{\nabla^{1,0}}+X^{1,0}
\qquad\text{and}\qquad
Q^{0,1}=\bar{\partial}+X^{0,1} \]
with
\[ X^{1,0}\in\Omega^{1,0}(X,\hat{S}^{\geqslant 2}(T_X^{1,0})\dual
\otimes T_X^{1,0})
\qquad\text{and}\qquad
X^{0,1}\in\Omega^{0,1}(X,\hat{S}^{\geqslant 2}(T_X^{1,0})\dual
\otimes T_X^{1,0}) .\]
Here $\delta$ is the usual Koszul operator
and $d^{\nabla^{1,0}}$ is the Chevalley--Eilenberg
differential associated with the $T^{1,0}_X$-connection
$\partial\otimes\id+\id\otimes\nabla^{1,0}$
on $\Lambda^q(T_X^{0,1})\dual\otimes\hat{S}(T_X^{1,0})\dual$.

Restricting to $p=0$ in~\eqref{eq:Q01},
we obtain a derivation $Q^{0,1}$ of degree $+1$ of the algebra
$\Omega^{0,\bullet}(X,\hat{S}(T_X^{1,0})\dual)$ satisfying
$Q^{0,1}\circ Q^{0,1}=0$.
Therefore, $Q^{0,1}$ induces a $L_\infty[1]$ algebra structure
(see \cite{MR1671737}) on $\Omega^{0,\bullet}(X,T_X^{1,0})$,
and $\big(T_X^{0,1}[1]\oplus T_X^{1,0},Q^{0,1}\big)$ is a
Kapranov dg manifold --- see \cite[Section 5.5]{arXiv:1408.2903}.
If the complex manifold $X$ admits a Kähler metric,
there is a canonical torsion-free flat $T^{1,0}_X$-connection
$\nabla^{1,0}$ on $T^{1,0}_X$ induced by the Levi-Civita connection
on $T_X$. In that case, Kapranov obtained an explicit formula
for the $L_\infty[1]$ algebra structure on
$\Omega^{0,\bullet}(X,T_X^{1,0})$
--- see \cite[Theorem 2.6]{MR1671737}.
Such $L_\infty[1]$ algebras played an important role
in Kapranov's investigation \cite{MR1671737} of Atiyah classes
and Rozansky--Witten invariants
--- see also \cite{arXiv:1408.2903,MR2989383,MR3322372}.
\end{remark}

\subsection{Polyvector fields associated with matched pairs}
According to Proposition~\ref{pro:hongkong},
the dg Lie algebroid structure on $\cB\to A[1]$ induces a
differential Gerstenhaber algebra structure on
$\sections{\Lambda^{\bullet+1}\cB}\isomorphism
\sections{\Lambda^\bullet A\dual\otimes\Lambda^{\bullet+1}B}$.
Its differential is the Chevalley--Eilenberg differential
\begin{equation}\label{eq:da}
\dabott: \sections{\Lambda^k A\dual\otimes\Lambda^{p+1}B}
\to\sections{\Lambda^{k+1} A\dual\otimes\Lambda^{p+1}B}
\end{equation}
corresponding to the Bott representation of $A$ on $\Lambda B$;
its associative multiplication is the wedge product
\begin{equation}\label{eq:wedge03}
\wedge:\sections{\Lambda^{k}A\dual\otimes\Tpolya{p}}
\otimes\sections{\Lambda^{l}A\dual\otimes\Tpolya{q}}
\to\sections{\Lambda^{k+l}A\dual\otimes\Tpolya{(p+q+1)}}
;\end{equation} and its Lie bracket
\begin{equation}\label{eq:lie01}
\schouten{\argument}{\argument}:
\sections{\Lambda^{k}A\dual\otimes\Tpolya{p}}
\otimes\sections{\Lambda^{l}A\dual\otimes\Tpolya{q}}
\to\sections{\Lambda^{k+l}A\dual\otimes\Tpolya{p+q}}
\end{equation}
is the Schouten bracket of the dg Lie algebroid $\cB\to A[1]$
extending the Lie bracket \eqref{eq:bracket} by way of the Leibniz
rule and the anchor map \eqref{eq:anchor}.

Applying Proposition~\ref{pro:hongkong} to the dg Lie algebroid
$\cB\to A[1]$, we obtain the following
\begin{proposition}\label{pro:zurich}
Let $A\bowtie B$ be a matched pair of Lie algebroids.
\begin{enumerate}
\item When endowed with the differential $\dabott$ \eqref{eq:da};
the associative multiplication \eqref{eq:wedge03};
and the Lie bracket \eqref{eq:lie01},
$\tot\sections{\Lambda^\bullet A\dual\otimes\Lambda^{\diamond+1}B}$
is a differential Gerstenhaber algebra, whence a dgla.
\item When endowed with the wedge product \eqref{eq:wedge03}
and the Schouten bracket \eqref{eq:lie01}, the cohomology
$\hypercohomology^\bullet_{\CE}(A,\Lambda^{\bullet+1}B)$
is a Gerstenhaber algebra.
\end{enumerate}
\end{proposition}

The following theorem is the first main result of the present section.

\begin{theorem}\label{thm:Bonn}
Let $L=A\bowtie B$ be a matched pair and let $\nabla$ be a torsion-free
$L$-connection on $B$.
Then the $L_\infty$ algebra
$\tot\big(\wa
\otimes_R\Tpoly{\diamond}\big)$
and the Gerstenhaber algebra
$\hypercohomology^\bullet_{\CE}(A,\Tpoly{\diamond})$
of Proposition~\ref{Bari} coincide with the dgla
$\tot\sections{\Lambda^\bullet A\dual\otimes\Lambda^{\diamond+1}B}$
and the Gerstenhaber algebra
$\hypercohomology^\bullet_{\CE}(A,\Lambda^{\diamond+1}B)$
of Proposition~\ref{pro:zurich}, respectively.
\end{theorem}

Theorem~\ref{thm:Bonn}
is a direct consequence of \cite[Proposition~4.9]{MR4091493}
--- see also \cite[Theorem~4.20]{MR4091493}.
However, for the sake of completeness, we proceed to outline a direct proof.

Denote by $\verticalX(B)$ the space of formal vertical vector fields
on the vector bundle $B\to M$. We have the natural identification
\[ \verticalX(B)\cong\sections{\hat{S}(B\dual)\otimes B} .\]

Since $\Tpoly{-1}=C^\infty(M)$; $\Tpoly{0}=\sections{B}$;
$\verticalTpoly{-1}\cong\sections{\hat{S}(B\dual)}$;
and $\verticalTpoly{0}\cong\verticalX(B)$,
specializing the contraction of Corollary~\ref{Smith}
in the cases where $k=-1$ and $k=0$ yields
a pair of contractions:
\begin{equation}\label{eq:tau0}
\begin{tikzcd}[cramped]
\Big(\wa, d_A \Big)
\arrow[r, "{\perturbed{\tau}}", shift left] &
\Big( \sections{\Lambda^\bullet L\dual\otimes\hat{S}(B\dual)}, Q \Big)
\arrow[l, "{\sigma}", shift left]
\arrow[loop, "{\perturbed{h}}",out=5,in=-5,looseness = 3]
\end{tikzcd}
\end{equation}
and
\begin{equation}\label{eq:tau1}
\begin{tikzcd}[cramped]
\Big( \sections{\Lambda^\bullet A\dual\otimes B}, d_A^{\Bott}\Big)
\arrow[r, "\etendu{\perturbed{\tau}}", shift left] &
\Big(\sections{\Lambda^\bullet L\dual}\otimes_R\verticalX(B),\lie{Q}\Big)
\arrow[l, "\etendu{\sigma}", shift left]
\arrow[loop, "\etendu{\perturbed{h}}",out=5,in=-5,looseness = 3]
\end{tikzcd}
\end{equation}

Note that
$\big(\sections{\Lambda^\bullet A\dual\otimes B},\dabott\big)$
is a dg Lie--Rinehart algebra over the dg ring
$\big(\wa,d_A\big)$
while
$\big(\sections{\Lambda^\bullet L\dual}\otimes_R\verticalX(B),\lie{Q}\big)$
is a dg Lie--Rinehart algebra over the dg ring
$\big(\sections{\Lambda^\bullet L\dual\otimes\hat{S}(B\dual)},Q\big)$.

\begin{proposition}\label{pro:Geneva}
The pair of maps $\perturbed{\tau}$ and $\etendu{\perturbed{\tau}}$
in the contractions \eqref{eq:tau0} and \eqref{eq:tau1}
constitutes a morphism of dg Lie--Rinehart algebras
from $\big(\sections{\Lambda^\bullet A\dual\otimes B},\dabott\big)$
to $\big(\sections{\Lambda^\bullet L\dual}\otimes_R\verticalX(B),\lie{Q}\big)$.
\end{proposition}

Recall that $\etendu{\perturbed{\tau}}=\sum_{l=0}^\infty
(\etendu{h}\lie{\perturbation})^l\etendu{\tau}$
where \[ \perturbation=d_L^\nabla+X^\nabla
=(d^{\Bott}_A+d^{\nabla^{1,0}}_B)+(X^{0,1}+X^{1,0}) ,\]
since $L=A\oplus B$ is a matched pair.
It is simple to see that
\begin{equation}\label{Naypyitaw}
\lie{d^{\Bott}_A+X^{0,1}}
\big(\sections{\Lambda^\bullet A\dual\otimes\Lambda^0 B\dual}
\otimes_R\verticalX(B)\big)\subset
\sections{\Lambda^{\bullet+1} A\dual\otimes\Lambda^0 B\dual}
\otimes_R\verticalX(B)\subset
\ker\etendu{h}
\end{equation}
and
\[ \lie{d^{\nabla^{1,0}}_B+X^{1,0}}
\big(\sections{\Lambda^\bullet A\dual\otimes\Lambda^0 B\dual}
\otimes_R\verticalX(B)\big)\subset
\sections{\Lambda^\bullet A\dual\otimes\Lambda^1 B\dual}
\otimes_R\verticalX(B) .\]
Therefore, the operator $\etendu{h}\lie{\perturbation}$
stabilizes
$\sections{\Lambda^\bullet A\dual\otimes\Lambda^0 B\dual}
\otimes_R\verticalX(B)$
and we can conclude that
\begin{equation}\label{eq:tauTpoly}
\etendu{\perturbed{\tau}}\big(\sections{\Lambda^\bullet A\dual
\otimes B}\big)\subset\sections{\Lambda^\bullet A\dual
\otimes\Lambda^0 B\dual}\otimes_R\verticalX(B).
\end{equation}
Since $\etendu{\perturbed{\tau}}$ is a cochain map,
we have $\lie{Q}\circ\etendu{\perturbed{\tau}}=\etendu{\perturbed{\tau}}\circ d^{\Bott}_A$,
and it follows that
\begin{equation}\label{eq:NAP}
\lie{Q^{0,1}}\circ\etendu{\perturbed{\tau}}
=\etendu{\perturbed{\tau}}\circ d^{\Bott}_A
\qquad\text{and}\qquad
\lie{Q^{1,0}}\circ\etendu{\perturbed{\tau}}=0
,\end{equation}
where $Q^{0,1}$ and $Q^{1,0}$ are the vector fields defined
in Lemma~\ref{Qpq}.

\begin{proof}[Proof of Proposition~\ref{pro:Geneva}]
It suffices to verify that the pair of maps
$\perturbed{\tau}$ and $\etendu{\perturbed{\tau}}$
in the contractions \eqref{eq:tau0} and \eqref{eq:tau1}
satisfy the identities
\begin{gather}
\bracket{\etendu{\perturbed{\tau}}(\xi\otimes b)}{\etendu{\perturbed{\tau}}(\eta\otimes c)}
=\etendu{\perturbed{\tau}}\bracket{\xi\otimes b}{\eta\otimes c} \label{eq:wuxi2} \\
\bracket{\etendu{\perturbed{\tau}}(\xi\otimes b)}{\perturbed{\tau}(\eta)}
=\perturbed{\tau}\bracket{\xi\otimes b}{\eta} \label{eq:wuxi3} \\
\perturbed{\tau}(\eta)\cdot\etendu{\perturbed{\tau}}(\xi\otimes b)
=\etendu{\perturbed{\tau}}(\eta\cdot\xi\otimes b) \label{eq:wuxi1}
\end{gather}
for all $\xi,\eta\in\wa$ and $b,c\in\sections{B}$.
In Equations~\eqref{eq:wuxi2} and~\eqref{eq:wuxi3}, the brackets on the r.h.s.\
are Schouten brackets of polyvector fields on the dg Lie algebroid
$\cB$, while the brackets on the l.h.s.\ are Schouten brackets
of polyvector fields on the dg Lie algebroid $\cF$.

Consider \[ \mathscr{Y}=\bracket{\etendu{\perturbed{\tau}}(\xi\otimes b)}
{\etendu{\perturbed{\tau}}(\eta\otimes c)} .\]
It follows from~\eqref{eq:tauTpoly} that
$\mathscr{Y}\in\sections{\Lambda^\bullet A\dual\otimes\Lambda^0 B\dual}\otimes_R\verticalX(B)$
and thence $\etendu{h}(\mathscr{Y})=0$.
Since $Q^{0,1}=\dabott+X^{0,1}$, according to \eqref{Naypyitaw} we
also get $\etendu{h}\lie{Q^{0,1}}(\mathscr{Y})=0$.
Furthermore, from Equation~\eqref{eq:NAP}, we obtain
\begin{multline*}
\lie{Q^{1,0}}\mathscr{Y}=\lie{Q^{1,0}}
\big(\bracket{\etendu{\perturbed{\tau}}(\xi\otimes b)}
{\etendu{\perturbed{\tau}}(\eta\otimes c)}\big) \\
=\bracket{\lie{Q^{1,0}}\big(\etendu{\perturbed{\tau}}(\xi\otimes b)\big)}
{\etendu{\perturbed{\tau}}(\eta\otimes c)}
\pm\bracket{\etendu{\perturbed{\tau}}(\xi\otimes b)}
{\lie{Q^{1,0}}\big(\etendu{\perturbed{\tau}}(\eta\otimes c)\big)}=0
.\end{multline*}
Therefore, we conclude that
\[ \etendu{h}\lie{Q}(\mathscr{Y})
=\etendu{h}\lie{Q^{0,1}}(\mathscr{Y})
+\etendu{h}\lie{Q^{1,0}}(\mathscr{Y})=0 .\]

From the definitions \eqref{eq:etendusigma} and~\eqref{eq:perturbedetendutau}
(see also \cite[Lemma~4.13]{MR4091493}) of $\etendu{\sigma}$
and $\etendu{\perturbed{\tau}}$, we obtain
\[ \etendu{\sigma}(\mathscr{Y})=
\etendu{\sigma}\bracket{\etendu{\perturbed{\tau}}(\xi\otimes b)}
{\etendu{\perturbed{\tau}}(\eta\otimes c)}
=\bracket{\xi\otimes b}{\eta\otimes c} .\]

Since $\etendu{h}(\mathscr{Y})=0$; $\etendu{h}\lie{Q}(\mathscr{Y})=0$;
and $\etendu{\sigma}(\mathscr{Y})=\bracket{\xi\otimes b}{\eta\otimes c}$,
it follows from Proposition~\ref{Mandalay} that
\[ \mathscr{Y}=\etendu{\perturbed{\tau}}
\big(\bracket{\xi\otimes b}{\eta\otimes c}\big) .\]
Identity~\eqref{eq:wuxi2} is thus established.
Identities~\eqref{eq:wuxi3} and~\eqref{eq:wuxi1}
can be verified in a similar fashion.
\end{proof}

We are now ready to prove Theorem~\ref{thm:Bonn}.

\begin{proof}[Proof of Theorem~\ref{thm:Bonn}]
Since $L=A\bowtie B$ is a matched pair,
the cochain complex on the l.h.s.\ of~\eqref{eq:Bagnolet}
in Proposition~\ref{Bagnolet} is
$\tot\sections{\Lambda^\bullet A\dual\otimes\Lambda^{\diamond+1}B}$.
It suffices to prove that the injection
$\etendu{\perturbed{\tau}}$ in~\eqref{eq:Bagnolet}
is a morphism of Lie algebras,
where the Lie bracket on
$\tot\sections{\Lambda^\bullet A\dual\otimes\Lambda^{\diamond+1}B}$
is as in Proposition~\ref{pro:zurich}.
This follows immediately from Proposition~\ref{pro:Geneva}
and the fact that $\etendu{\perturbed{\tau}}$ respects wedge products
by virtue of Proposition~\ref{Bagnolet}.
\end{proof}

\subsection{Polydifferential operators associated with matched pairs}

We now turn to the study of polydifferential operators.

Recall that the universal enveloping algebra $\enveloping{\cB}$
of the dg Lie algebroid $\cB\to A[1]$ is a dg Hopf algebroid
over the dgca $C^\infty(A[1])=\wa$.
There is a natural isomorphism of left
$\wa$-modules
\begin{equation}\label{eq:PSU}
\enveloping{\cB}\isomorphism
\wa\otimes_R\enveloping{B}
.\end{equation}
Consequently,
$\wa\otimes_R\enveloping{B}$
admits a structure of dg Hopf algebroid over the dgca
$\big(\wa,d_A\big)$:

\begin{enumerate}
\item The multiplication
is characterized by the relations
\begin{align*}
(\xi\otimes 1)\cdot(\eta\otimes 1) &= \xi\wedge\eta\otimes 1,
&&\forall \xi,\eta\in\wa; \\
(1\otimes u)\cdot(1\otimes v) &= 1\otimes u\cdot v,
&&\forall u,v\in\enveloping{B}; \\
(\xi\otimes 1)\cdot(1\otimes u) &= \xi\otimes u,
&&\forall \xi\in\wa,\
\forall u\in\enveloping{B}; \\
(1\otimes b)\cdot(\xi \otimes 1)-(\xi \otimes 1)\cdot(1\otimes b)
&= (\nabla^{\Bott}_b\xi )\otimes 1,
&&\forall b\in\sections{B},\ \forall \xi \in\sections{A\dual}
.\end{align*}
Indeed, the multiplication is defined by the relation
\begin{equation}\label{eq:AUB}
(\xi\otimes b_1b_2\cdots b_n)\cdot(\eta\otimes u)
= \sum_{k=0}^n\sum_{\sigma\in\shuffle{k}{n-k}}
(\xi\wedge\nabla^{\Bott}_{b_{\sigma(1)}}\cdots\nabla^{\Bott}_{b_{\sigma(k)}}\eta)
\otimes b_{\sigma(k+1)}\cdots b_{\sigma(n)}\cdot u ,
\end{equation}
for all $\xi,\eta\in\wa$,
$b_1,b_2,\dots,b_n\in\sections{B}$, and $u\in\enveloping{B}$.
Note that the multiplication is well-defined
by Equation~\eqref{eq:AUB} because the Bott $B$-connection
on $A\dual$ is flat.
\item The source and target maps
\[ \begin{tikzcd}
\wa \arrow[r, shift left, "\alpha"]
\arrow[r, shift right, "\beta"'] & \wa\otimes_R\enveloping{B}
\end{tikzcd} \]
are one and the same map: the inclusion $\xi\mapsto\xi\otimes 1$.
\item The differential is the Chevalley--Eilenberg differential
\[ \dau: \wa\otimes_R\enveloping{B}
\to\sections{\Lambda^{\bullet+1} A\dual}\otimes_R\enveloping{B} \]
of the Lie algebroid $A$ for cochains with coefficients
in $\enveloping{B}$.
The $A$-module structure on $\enveloping{B}$ follows
from the canonical identification \eqref{eq:BLA}
--- the Lie algebroid $A$ acts on $\enveloping{L}$ by multiplication from the left.
\item The comultiplication $\Delta$ is defined
by the commutative diagram of left $\wa$-modules
\[ \begin{tikzcd}[row sep=small]
& \big( \wa\otimes_R \enveloping{B}\big)
\otimes_{\wa}
\big( \wa\otimes_R \enveloping{B}\big)
\arrow{dd}{\isomorphism} \\
\wa\otimes_R \enveloping{B}
\arrow{ru}[near start]{\Delta} \arrow{rd}[swap]{\id\otimes \tilde{\Delta} } & \\
& \wa\otimes_R \enveloping{B}\otimes_R \enveloping{B}
. \end{tikzcd} \]
Indeed, it is the $\wa$-linear extension of the comultiplication
\[ \widetilde{\Delta}:\enveloping{B}
\to\enveloping{B}\otimes_R\enveloping{B} \]
of the Hopf algebroid $\enveloping{B}$ --- see~\cite{MR1815717}.
\item The counit map
$\varepsilon:\sections{\Lambda^\bullet A^\vee}\otimes_R\enveloping{B}
\to\sections{\Lambda^\bullet A^\vee}$ is the canonical projection.
\end{enumerate}

From the isomorphism \eqref{eq:PSU},
we obtain an isomorphism
\begin{equation}\label{eq:CDG}
\big(\suspended\enveloping{\cB}\big)^{\otimes k+1}\isomorphism
\wa\otimes_R\enveloping{B}^{\otimes k+1}[-k-1]
,\end{equation}
which identifies (up to a grading shift) the differential
$\cQ:\suspended\enveloping{\cB}^{\otimes k+1}
\to\suspended\enveloping{\cB}^{\otimes k+1}$
to the Chevalley--Eilenberg differential
\begin{equation}\label{eq:dAUB}
\dau: \wa
\otimes_R\enveloping{B}^{\otimes k+1}
\to\sections{\Lambda^{\bullet+1} A\dual}
\otimes_R\enveloping{B}^{\otimes k+1}
.\end{equation}
Here $\enveloping{B}^{\otimes k+1}$ with $k\geqslant -1$ denotes
the tensor product $\enveloping{B}\otimes_R \cdots\otimes_R \enveloping{B}$
of $(k+1)$-copies of the left $R$-module $\enveloping{B}$,
and the $A$-module structure on $\enveloping{B}^{\otimes k+1}$
is the natural extension of the $A$-module structure on $\enveloping{B}$.

The Hochschild coboundary differential \eqref{eq:hochschild},
the Gerstenhaber bracket \eqref{eq:Gbraket},
and the cup product \eqref{eq:cup} on
$\tot_\oplus\suspended\enveloping{\cB}^{\otimes\bullet+1}$
arising from the dg Lie algebroid $\cB$
carry over, through the identification \eqref{eq:CDG},
to a Hochschild coboundary differential
\begin{equation}\label{eq:hochschild03}
\wa\otimes_R
\enveloping{B}^{\otimes k}
\xto{\mathfrak{d}_{\mathscr{H}}}
\sections{\Lambda^{\bullet}A\dual}\otimes_R
\enveloping{B}^{\otimes k+1}
,\end{equation}
a Gerstenhaber bracket
\begin{equation}\label{eq:Gbracket03}
\big(\wa\otimes_R
\enveloping{B}^{\otimes p+1}\big)\otimes
\big(\wa\otimes_R
\enveloping{B}^{\otimes q+1}\big)
\xto{\gerstenhaber{\argument}{\argument}}
\sections{\Lambda^{\bullet} A\dual}\otimes_R
\enveloping{B}^{\otimes p+q+1}
,\end{equation}
and a cup product
\begin{equation}\label{eq:cup03}
\big(\wa\otimes_R
\enveloping{B}^{\otimes p+1}\big)\otimes
\big(\wa\otimes_R
\enveloping{B}^{\otimes q+1}\big) \xto{\smile}
\sections{\Lambda^{\bullet} A\dual}\otimes_R
\enveloping{B}^{\otimes(p+q+1)+1}
\end{equation}
on $\wa\otimes_R\enveloping{B}^{\otimes\diamond+1}$.

Note that both the Hochschild coboundary differential and
the cup product are $\wa$-linear.
That is, we have
\[ \mathfrak{d}_{\mathscr{H}}(\omega\otimes u)
=(-1)^k\omega\otimes\hochschild(u) \]
and
\[ (\omega\otimes u)\smile(\theta\otimes v)
=(-1)^{l(p+1)}(\omega\wedge\theta)\otimes(u\otimes v) \]
for all $\omega\in\sections{\Lambda^{k}A\dual}$,
$\theta\in\sections{\Lambda^{l}A\dual}$,
$u\in\enveloping{B}^{\otimes p+1}$ and $v\in\enveloping{B}^{\otimes q+1}$.
However, the Gerstenhaber bracket \eqref{eq:Gbracket03}
is \emph{not} the obvious extension of the Gerstenhaber bracket
on $\enveloping{B}^{\otimes\bullet+1}$ obtained by tensoring
with the commutative associative algebra $\wa$.
In fact, to write down an explicit formula --- which is quite involved ---
one must use the Bott representation of $B$ on $\Lambda A\dual$.

Applying Proposition~\ref{pro:hongkong1}
to the dg Lie algebroid $\cB\to A[1]$,
we are led to the following
\begin{proposition}\label{pro:zurich1}
Let $A\bowtie B$ be a matched pair of Lie algebroids.
\begin{enumerate}
\item When endowed with the differential $\dhoch$
(see~\eqref{eq:dAUB} and~\eqref{eq:hochschild03})
and the Gerstenhaber bracket \eqref{eq:Gbracket03},
$\tot\big(\wa\otimes_R\enveloping{B}^{\otimes\diamond+1}\big)$ is a dgla.
\item When endowed with the cup product \eqref{eq:cup03}
and the Gerstenhaber bracket \eqref{eq:Gbracket03},
the Hochschild cohomology $\hypercohomology^\bullet_{\CE}
\big(A,{\enveloping{B}}^{\otimes\diamond+1}\big)$,
i.e.\ the cohomology of the complex
\[ \big(\tot(\Lambda^\bullet A\dual\otimes_R
\enveloping{B}^{\otimes\diamond+1}),\dhoch\big) ,\]
is a Gerstenhaber algebra.
\end{enumerate}
\end{proposition}

As pointed out in Remark~\ref{rk:compatible}, on the cochain level,
the Gerstenhaber bracket \eqref{eq:Gbracket03} satisfies
the graded Leibniz rule with respect to the cup product
\eqref{eq:cup03} only up to homotopy. Therefore,
$\tot\big(\wa\otimes_R
\enveloping{B}^{\otimes\diamond+1}\big)$ is \emph{not}
a differential Gerstenhaber algebra.
Likewise, the cup product is graded commutative only up to homotopy.
Again this is reminiscent of ordinary Hochschild cohomology theory
of associative algebras~\cite{MR0171807}.

Theorem~\ref{thm:Berlin} below is the second main result of the present section,
the remainder of which is devoted to its proof.

\begin{theorem}\label{thm:Berlin}
Let $L=A\bowtie B$ be a matched pair and let $\nabla$ be a torsion-free
$L$-connection on $B$.
Then the $L_\infty$ algebra
$\tot\big(\wa\otimes_R\Dpoly{\diamond}\big)$
and the Gerstenhaber algebra
$\hypercohomology^\bullet_{\CE}(A,\Dpoly{\diamond})$
of Proposition~\ref{thm:Naples} coincide with the dgla
$\tot\big(\wa\otimes_R\enveloping{B}^{\diamond+1}\big)$
and the Gerstenhaber algebra
$\hypercohomology^\bullet_{\CE}(A,\enveloping{B}^{\diamond+1})$
of Proposition~\ref{pro:zurich1}, respectively.
\end{theorem}

Denote by $\verticalD(B)$ the algebra of formal vertical
differential operators on the vector bundle $B\to M$.
The canonical isomorphism \eqref{eq:varphi},
specialized to the case $k=0$, gives the identification
\[ \verticalD(B)\cong\Gamma\big(\hat{S}(B\dual)\otimes SB\big) .\]
Consider the contraction \eqref{eq:Dhaka1} in Corollary~\ref{cor:Dhaka}.
As an immediate consequence of isomorphism \eqref{eq:BLA},
we have $\Dpoly{k}\cong\enveloping{B}^{\otimes k+1}$.
We also have isomorphism \eqref{eq:varphi}:
$\verticalDpoly{k}\cong\sections{\hat{S}(B\dual)\otimes(SB)^{\otimes k+1}}$.
Specializing Corollary~\ref{cor:Dhaka} to the case $k=0$,
we obtain the contraction
\begin{equation}\label{eq:Metz}
\begin{tikzcd}
\Big(\wa\otimes_R\enveloping{B},\dau\Big)
\arrow[r, shift left, "\etendu{\perturbed{\tau}}"] &
\Big(\sections{\Lambda^\bullet L\dual}\otimes_R
\verticalD(B),\gerstenhaber{Q}{\argument}\Big)
\arrow[l, shift left, "\etendu{\sigma}"]
\arrow[loop, "\etendu{\perturbed{h}}", out=5,in=-5,looseness = 3]
\end{tikzcd}
\end{equation}

Likewise, specializing Proposition~\ref{Mandalay2} to the case of
differential (rather than polydifferential) operators, we obtain
\begin{proposition}\label{Mandalay3}
Given $x\in\wa\otimes_R\enveloping{B}$
and $y\in\sections{\Lambda^\bullet L\dual}
\otimes_R\verticalD(B)$,
we have
\[ \etendu{\perturbed{\tau}}(x)=y
\qquad\text{if and only if}\qquad
\begin{cases}
\etendu{h}(y) = 0 \\
\etendu{h}\big(\gerstenhaber{Q}{y}\big) = 0 \\
\etendu{\sigma}(y) = x
\end{cases} \]
\end{proposition}

Both sides of~\eqref{eq:Metz} are universal enveloping algebras
of dg Lie algebroids.
Indeed, isomorphism \eqref{eq:PSU} identifies
$\wa\otimes_R\enveloping{B}$
with the universal enveloping algebra $\enveloping{\cB}$
of the dg Lie algebroid $\cB\to A[1]$, while
$\sections{\Lambda^\bullet L\dual}\otimes_R\verticalD(B)$
is naturally identified with the universal enveloping algebra
of the Fedosov dg Lie algebroid $\cF\to\cM$
(appearing in Proposition~\ref{pro:Rome})
since $\cF$ is isomorphic to the pull back bundle $\pr^* T_{\ver}B$.

Following \cite[Definition 2.6]{MR1913813},
we consider the associative algebra \[ \jet{B}:=\Hom_R\big(\enveloping{B},R\big) \]
of $B$-jets on $M$ --- the multiplication on $\jet{B}$ arises as the map
dual to the comultiplication on $\enveloping{B}$.

Dualizing the isomorphism of $R$-coalgebras
\[ \pbwb:\sections{SB}\xto{\cong}\enveloping{B} \]
appearing in Lemma~\ref{pbwB}, we obtain
an isomorphism of associative $R$-algebras
\begin{equation}\label{eq:pbwBT}
\pbwb^\top:\jet{B}\xto{\cong}\sections{\hat{S}(B\dual)}
.\end{equation}

The isomorphism $\pbwb^\top$ identifies
the \emph{Grothendieck $B$-connection $\nablat^G$ on $\jet{B}$}
introduced by Nest--Tsygan \cite[Proposition~2.7]{MR1913813}
and characterized by the relation
\[ \duality{\nablat^G_b\varphi}{u}
=\anchor_{b}\duality{\varphi}{u}-\duality{\varphi}{b\cdot u} ,\]
for all $b\in\sections{B}$; $\varphi\in\jet{B}$; and $u\in\enveloping{B}$,
with the $B$-connection $\cnt$ on $\sections{\hat{S}(B\dual)}$
dual to the flat $B$-connection on $\sections{SB}$
defined by Equation~\eqref{eq:cntb}.
Indeed, the diagram
\begin{equation}\label{eq:Gconnection}
\begin{tikzcd}
\jet{B} \arrow[r, "\nablat^G_b"] \arrow[d, "\isomorphism", "\pbwb^\top"']
& \jet{B} \arrow[d, "\isomorphism"', "\pbwb^\top"]
\\ \sections{\hat{S}(B\dual)} \arrow[r, "\cnt_b"']
& \sections{\hat{S}(B\dual)}
\end{tikzcd} \end{equation}
commutes for all $b\in\sections{B}$ since
\begin{align*}
\duality{\pbwb^\top(\nablat^G_b\varphi)}{s}
&= \duality{\nablat^G_b\varphi}{\pbwb(s)} \\
&= \anchor_{b}\duality{\varphi}{\pbwb(s)}-\duality{\varphi}{b\cdot\pbwb(s)} \\
&= \anchor_{b}\duality{\varphi}{\pbwb(s)}-\duality{\varphi}{\pbwb(\cnt_b s)} \\
&= \anchor_{b}\duality{\pbwb^\top(\varphi)}{s}-\duality{\pbwb^\top(\varphi)}{\cnt_b s} \\ &= \duality{\cnt_b\big(\pbwb^\top(\varphi)\big)}{s}
\end{align*}
for all $\varphi\in\jet{B}$ and $s\in\sections{SB}$.

Given $x\in\enveloping{B}$,
we think of the multiplication $u\mapsto u\cdot x$ by $x$
from the right in $\enveloping{B}$
as an endomorphism $\widetilde{R}_x$ of the $R$-module $\enveloping{B}$,
and we consider the dual endomorphism
\[ \widetilde{R}_x^\top:\jet{B}\to\jet{B} .\]

For all $b\in\sections{B}$ and $x\in\enveloping{B}$, we have
\[ \nablat^G_{b}\circ\widetilde{R}_x^\top=\widetilde{R}_x^\top\circ\nablat^G_{b} ,\]
since
\begin{multline*}
\duality{\nablat^G_{b}\circ\widetilde{R}_x^\top(\varphi)}{u}
=\anchor_b\duality{\widetilde{R}_x^\top(\varphi)}{u}
-\duality{\widetilde{R}_x^\top(\varphi)}{b\cdot u}
=\anchor_b\duality{\varphi}{u\cdot x}
-\duality{\varphi}{b\cdot u\cdot x} \\
=\duality{\nablat^G_{b}\varphi}{u\cdot x}
=\duality{\nablat^G_{b}\varphi}{\widetilde{R}_x(u)}
=\duality{\widetilde{R}_x^\top\circ\nablat^G_{b}(\varphi)}{u}
\end{multline*}
for all $\varphi\in\jet{B}$ and $u\in\enveloping{B}$.

\begin{lemma}\label{chipmunk}
For every $x\in\enveloping{B}$,
the endomorphism $\widetilde{R}_x^\top$ of $\jet{B}$
is an $R$-linear differential operator on the algebra $\jet{B}$.
Furthermore, the map $x\mapsto\widetilde{R}_x^\top$
is a morphism of associative algebras from $\enveloping{B}$
to the algebra of $R$-linear differential operators
acting on the algebra $\jet{B}$.
\end{lemma}

\begin{proof}
Adopting the Sweedler notation
\[ \Delta(u)=\sum_{(u)}u_{(1)}\otimes u_{(2)}=\sum_{(u)}u_{(2)}\otimes u_{(1)} \]
to denote the cocommutative comultiplication on $\enveloping{B}$ defined by
Equation~\eqref{eq:325}, and using the very definition of the multiplication
in $\enveloping{B}$ --- see relations \eqref{eq:four} or Equation~\eqref{eq:AUB} ---
we easily obtain that
\[ \widetilde{R}_f(u)=u\cdot f=\sum_{(u)}\anchor_{u_{(1)}}(f)\cdot u_{(2)}
=\sum_{(u)}\anchor_{u_{(2)}}(f)\cdot u_{(1)} \]
for all $f\in R\subset\enveloping{B}$ and $u\in\enveloping{B}$.
Here the anchor map $\anchor:\sections{B}\to\XX(M)$ of the Lie algebroid $B\to M$ was implicitly
extended to a morphism of associative algebras
$\anchor:\enveloping{B}\to\DD(M)$.

It follows that, for all $f\in C^\infty(M)\subset\enveloping{B}$,
$\xi\in\jet{B}$, and $u\in\enveloping{B}$, we have
\begin{multline*}
\duality{\widetilde{R}_f^\top(\xi)}{u}
=\duality{\xi}{\widetilde{R}_f(u)}
=\big\langle\xi\big|\sum_{(u)}\anchor_{u_{(2)}}(f)\cdot u_{(1)}\big\rangle
=\sum_{(u)}\duality{\xi}{u_{(1)}}\cdot\anchor_{u_{(2)}}(f) \\
=\sum_{(u)}\duality{\xi}{u_{(1)}}\cdot\anchor_{u_{(2)}\cdot f}(1)
=\sum_{(u)}\duality{\xi}{u_{(1)}}\cdot\duality{\boldsymbol{1}}{u_{(2)}\cdot f}
=\sum_{(u)}\duality{\xi}{u_{(1)}}\cdot\duality{\widetilde{R}_f^\top(\boldsymbol{1})}{u_{(2)}} \\
=\big\langle\xi\otimes\widetilde{R}_f^\top(\boldsymbol{1})\big|\sum_{(u)}u_{(1)}\otimes u_{(2)}\big\rangle
=\duality{\xi\otimes\widetilde{R}_f^\top(\boldsymbol{1})}{\Delta(u)}
=\duality{\Delta^\top\big(\xi\otimes\widetilde{R}_f^\top(\boldsymbol{1})\big)}{u} \\
=\duality{\xi\cdot\widetilde{R}_f^\top(\boldsymbol{1})}{u}
,\end{multline*}
where the $B$-jet $\boldsymbol{1}$ is the morphism
of left $R$-modules $\enveloping{B}\ni u\mapsto\anchor_{u}(1)\in R$
associated with the constant function $1\in R=C^\infty(M)$.
Hence, we have
\[ \widetilde{R}_f^\top(\xi)=\xi\cdot\widetilde{R}_f^\top(\boldsymbol{1}),
\qquad\forall f\in C^\infty(M),\;\forall\xi\in\jet{B} ,\]
which shows that $\widetilde{R}_f^\top$ is indeed a differential operator
of order zero on the algebra $\jet{B}$.

For every $b\in\sections{B}\subset\enveloping{B}$,
the operator $\widetilde{R}_b$ is a coderivation of the left $R$-coalgebra $\enveloping{B}$
and, consequently, $\widetilde{R}_b^\top$ is a derivation of the left $R$-algebra $\jet{B}$.

For any two elements $u,v\in\enveloping{B}$, we have
$\widetilde{R}_{uv}^\top=\widetilde{R}_{u}^\top\circ\widetilde{R}_{v}^\top$
since $\widetilde{R}_{uv}=\widetilde{R}_{v}\circ\widetilde{R}_{u}$.

Since the universal enveloping algebra $\enveloping{B}$
of the Lie algebroid $B\to M$ is generated multiplicatively
by the elements of its subspace $C^\infty(M)\oplus\sections{B}$,
it follows immediately that $\widetilde{R}_x^\top$ acts on the algebra of jets $\jet{B}$
in the manner of a differential operator.
The $R$-linearity of $\widetilde{R}_x^\top:\jet{B}\to\jet{B}$ is obvious.
\end{proof}

Similarly, we can consider the Lie algebroid $\cB\to A[1]$,
the graded associative algebra of $\cB$-jets on $A[1]$
\[ \jet{\cB}:=\Hom_{\wa}\big(\enveloping{\cB},\wa\big) ,\]
and the \emph{Grothendieck $\cB$-connection $\nabla^G$ on $\jet{\cB}$}
characterized by the relation
\[ \duality{\nabla^G_b\varphi}{u}=\baranchor_{b}\duality{\varphi}{u}
-\duality{\varphi}{b\cdot u} ,\]
for all $b\in\sections{\cB}$; $\varphi\in\jet{\cB}$; and $u\in\enveloping{\cB}$.

It follows from the natural identification of the space of sections
of the Lie algebroid $\cB\to A[1]$ with
$\sections{\Lambda^\bullet A\dual\otimes B}$;
the definition \eqref{eq:anchor} of the
anchor map $\baranchor$ of the Lie algebroid $\cB$;
and the isomorphism of graded associative algebras
\[ \jet{\cB}\isomorphism\wa\otimes_R\jet{B} \]
induced by the identification \eqref{eq:PSU} that
\begin{equation}\label{eq:G}
\nabla^G_{1\otimes b}(\alpha\otimes\varphi)
=(\nabla^{\Bott}_b\alpha)\otimes\varphi
+\alpha\otimes(\nablat^G_b\varphi)
,\end{equation}
for all $b\in\sections{B}$; $\alpha\in\sections{\Lambda A\dual}$;
and $\varphi\in\jet{B}$.

Given $x\in\enveloping{\cB}$,
we think of the multiplication $u\mapsto u\cdot x$ by $x$ from the right in $\enveloping{\cB}$
as an endomorphism $R_x$ of the $\wa$-module
$\enveloping{\cB}\isomorphism\wa\otimes_R\enveloping{B}$,
and we consider the dual endomorphism
\[ R_x^\top:\wa\otimes_R\jet{B}\to\wa\otimes_R\jet{B} .\]
The multiplication in $\enveloping{\cB}\isomorphism
\wa\otimes_R\enveloping{B}$
was defined in Equation~\eqref{eq:AUB}.

For all $b\in\sections{\cB}\isomorphism\sections{\Lambda A\dual\otimes B}$
and $x\in\enveloping{\cB}\isomorphism
\wa\otimes_R\enveloping{B}$, we have
\begin{equation}\label{eq:rlcommute}
\nablab_{b}\circ R_x^\top=R_x^\top\circ\nablab_{b}
,\end{equation}
since
\begin{multline*}
\duality{\nablab_{b}\circ R_x^\top(\varphi)}{u}
=\baranchor_b\duality{R_x^\top(\varphi)}{u}
-\duality{R_x^\top(\varphi)}{b\cdot u}
=\baranchor_b\duality{\varphi}{u\cdot x}
-\duality{\varphi}{b\cdot u\cdot x} \\
=\duality{\nablab_{b}\varphi}{u\cdot x}
=\duality{\nablab_{b}\varphi}{R_x(u)}
=\duality{R_x^\top\circ\nablab_{b}(\varphi)}{u}
\end{multline*}
for all $\varphi\in\jet{\cB}$ and $u\in\enveloping{\cB}$.

\begin{lemma}\label{hedgehog}
For every $x\in\enveloping{\cB}$,
the endomorphism $R_x^\top$ of $\jet{\cB}$
is a differential operator on the algebra $\jet{\cB}$.
Furthermore, the map $x\mapsto R_x^\top$
is a morphism of associative algebras from $\enveloping{\cB}$
to the algebra of $\wa$-linear differential operators
acting on the algebra $\jet{\cB}$.
\end{lemma}

The proof of Lemma~\ref{hedgehog} is similar to the proof of Lemma~\ref{chipmunk}
and is therefore omitted.

The following lemma indicates that $\etendu{\perturbed{\tau}}(x)$
coincides with the operator $R_x^\top$ conjugated by the algebra isomorphism
\[ \id\otimes\pbwb^\top : \wa\otimes_R\jet{B}
\to\sections{\Lambda^{\bullet}A\dual\otimes\hat{S}(B\dual)} .\]

\begin{lemma}\label{lem:tauR}
For all $x\in\enveloping{\cB}\isomorphism
\wa\otimes_R\enveloping{B}$, the diagram
\[ \begin{tikzcd}
\wa\otimes_R\jet{B} \arrow[r, "R_x^\top"] \arrow[d, "\isomorphism", "\id\otimes\pbwb^\top"']
& \wa\otimes_R\jet{B} \arrow[d, "\isomorphism"', "\id\otimes\pbwb^\top"]
\\ \sections{\Lambda^\bullet A\dual\otimes\hat{S}(B\dual)} \arrow[r, "\etendu{\perturbed{\tau}}(x)"']
& \sections{\Lambda^\bullet A\dual\otimes\hat{S}(B\dual)}
\end{tikzcd} \]
commutes.
\end{lemma}

\begin{proof}
Given any element $x$ of $\wa\otimes_R\enveloping{B}$,
let \begin{equation}\label{Yrhs}
\mathcal{Y}_x=(\id\otimes\pbwb^\top)\circ R_x^\top
\circ(\id\otimes\pbwb^\top)\inv
.\end{equation}
According to Lemma~\ref{hedgehog}, the operator $R_x^\top$ acts
on the algebra $\jet{\cB}\cong\wa\otimes_R\jet{B}$
in the manner of a $\wa$-linear differential operator
and, consequently, the operator $\mathcal{Y}_x$ acts on the algebra
$\sections{\Lambda^{\bullet}A\dual\otimes\hat{S}(B\dual)}\cong C^\infty(\cM)$
in the manner of a $\wa$-linear differential operator.
In other words, $\mathcal{Y}_x$ is a formal vertical
differential operator on $\cB$, i.e.\ an element of
\[ \sections{\Lambda A\dual\otimes\hat{S}(B\dual)\otimes SB}
\isomorphism \wa\otimes_R\verticalD(B) .\]

According to Corollary~\ref{cor:Q1001}, for all $b\in\sections{B}$ and
$\alpha\otimes\varsigma\in\sections{\Lambda A\dual\otimes\hat{S}(B\dual)}$,
we have
\[ \interior{b}Q^{1,0}(\alpha\otimes\varsigma)
=\nabla^{\Bott}_b\alpha\otimes\varsigma+\alpha\otimes\cnt_b\varsigma .\]
By commutativity of diagram~\eqref{eq:Gconnection},
we have \[ \cnt_b\varsigma=
\pbwb^\top\circ\nablat^G_b\circ(\pbwb^\top)\inv(\varsigma) ,\]
and it then follows from Equation~\eqref{eq:G} that
\begin{equation}\label{summit}
\interior{b}Q^{1,0}=	
(\id\otimes\pbwb^\top)\circ\nablab_{1\otimes b}
\circ(\id\otimes\pbwb^\top)\inv.
\end{equation}
Equations~\eqref{summit}, \eqref{Yrhs}, and~\eqref{eq:rlcommute} yield
\[ \gerstenhaber{\interior{b}Q^{1,0}}{\mathcal{Y}_x}=(\id\otimes\pbwb^\top)\circ
\gerstenhaber{\nablab_{1\otimes b}}{R_x^\top}\circ(\id\otimes\pbwb^\top)\inv=0 \]
for all $b\in\sections{B}$.
Hence, we obtain \[ \gerstenhaber{Q^{1,0}}{\mathcal{Y}_x}=0 .\]

It is not difficult to check that the subspace
\[ \sections{\Lambda A\dual\otimes\hat{S}(B\dual)\otimes SB}
\isomorphism \sections{p^\top(\Lambda^\bullet A\dual)}\otimes_R\verticalD(B) \]
of $\sections{\Lambda L\dual}\otimes_R\verticalD(B)$
is contained in the kernel of $\etendu{h}$
and is stable under $\gerstenhaber{d^{\Bott}_A+X^{0,1}}{\argument}$.
Therefore, since
\[ Q=Q^{0,1}+Q^{1,0}; \qquad Q^{0,1}=d^{\Bott}_A+X^{0,1};
\qquad\text{and}\qquad \gerstenhaber{Q^{1,0}}{\mathcal{Y}_x}=0, \]
we conclude that
\[ \etendu{h}(\mathcal{Y}_x)=0 \qquad\text{and}\qquad
\etendu{h}(\gerstenhaber{Q}{\mathcal{Y}_x})
=\etendu{h}(\gerstenhaber{d^{\Bott}_A+X^{0,1}}{\mathcal{Y}_x})=0 .\]

Let $\{\chi_i\}_{i=1,\ldots,r}$ and $\{\partial_j\}_{j=1,\ldots,r}$
be a pair dual local frames for the vector bundles $B\dual$ and $B$ respectively.
Then, with the usual multi-indices notations,
$\{\chi^I\}_{I\in\mathbb{N}^r}$ and $\{\partial^J\}_{J\in\mathbb{N}^r}$
are the corresponding dual local frames for $\hat{S}(B\dual)$
and $S(B)$ respectively.
Locally, every formal vertical differential operator
$\xi\in\sections{\Lambda A\dual\otimes\hat{S}(B\dual)\otimes SB}$
on the dg vector bundle $\cB\to A[1]$
can be written as a linear combination
$\xi=\sum_{I,J\in\NN^r}\xi_{I,J}\otimes\chi^I\otimes\partial^J$
with coefficients $\xi_{I,J}=\duality{\xi(\chi^J)}{\partial^I}$
in $\sections{\Lambda A\dual}$.
Furthermore, we have
\begin{equation}\label{eq:Turku}
\etendu{\sigma}(\xi)=\sum_{J\in\NN^r}\xi_{0,J}\otimes\pbwb(\partial^J)
\qquad\text{with}\qquad \xi_{0,J}=\duality{\xi(\chi^J)}{1}
.\end{equation}
In particular, since
\[ \mathcal{Y}_{\alpha\otimes u}
= (\id\otimes\pbwb^\top)\circ R_{\alpha\otimes u}^\top
\circ(\id\otimes\pbwb^\top)\inv
= \big((\id\otimes\pbwb\inv)\circ R_{\alpha\otimes u}
\circ(\id\otimes\pbwb)\big)^\top ,\]
it follows that
\begin{multline*}
\duality{\mathcal{Y}_{\alpha\otimes u}(\chi^J)}{1}
=\duality{\chi^J}{(\id\otimes\pbwb\inv)\circ R_{\alpha\otimes u}
\circ(\id\otimes\pbwb)(1)} \\
=\duality{\chi^J}{(\id\otimes\pbwb\inv)\circ R_{\alpha\otimes u}(1)}
=\alpha\cdot\duality{\chi^J}{\pbwb\inv(1\cdot u)}
=\alpha\cdot\duality{\chi^J}{\pbwb\inv(u)}
\end{multline*}
and, according to Equation~\eqref{eq:Turku},
\begin{multline*} \etendu{\sigma}(\mathcal{Y}_{\alpha\otimes u})
=\sum_{J\in\NN^r} \duality{\mathcal{Y}_{\alpha\otimes u}(\chi^J)}{1}
\otimes\pbwb(\partial^J) =\sum_{J\in\NN^r} \alpha\cdot\duality{\chi^J}{\pbwb\inv(u)}
\otimes\pbwb(\partial^J) \\ =\alpha\otimes\pbwb\left( \sum_{J\in\NN^r}
\duality{\chi^J}{\pbwb\inv(u)}\cdot\partial^J\right)
=\alpha\otimes\pbwb\left(\pbwb\inv(u)\right)=\alpha\otimes u,
\end{multline*}
for all $\alpha\otimes u\in\sections{\Lambda A\dual}\otimes_R\enveloping{B}$.
Hence, we have $\etendu{\sigma}(\mathcal{Y}_x)=x$.

It follows from Proposition~\ref{Mandalay3} that
$\mathcal{Y}_x=\etendu{\perturbed{\tau}}(x)$.
\end{proof}

The following proposition will play a key role in the proof of
Theorem~\ref{thm:Berlin}.

\begin{proposition}\label{pro:Metz}
In the contraction~\eqref{eq:Metz}, the cochain map
\[ \etendu{\perturbed{\tau}} : \wa\otimes_R\enveloping{B} \to
\sections{\Lambda^\bullet L\dual}\otimes_R\verticalD(B) \]
respects the algebra and the coalgebra structures
as well as the counit maps.
Hence it realizes a morphism of dg Hopf algebroids
from $\enveloping{\cB}$ to $\enveloping{\cF}$.
\end{proposition}

\begin{proof}
It follows immediately from Lemma~\ref{lem:tauR} that
$\etendu{\perturbed{\tau}}$ is a morphism of algebras:
\begin{multline*}
\etendu{\perturbed{\tau}}(x_1 x_2)
=(\id\otimes\pbwb^\top)\circ R_{x_1 x_2}^\top
\circ(\id\otimes\pbwb^\top)\inv \\
=(\id\otimes\pbwb^\top)\circ R_{x_1}^\top\circ R_{x_2}^\top\circ(\id\otimes\pbwb^\top)\inv
=\etendu{\perturbed{\tau}}(x_1)\cdot\etendu{\perturbed{\tau}}(x_2)
.\end{multline*}
Proposition~\ref{taucoalgebra} established that
$\etendu{\perturbed{\tau}}$ is a morphism of coalgebras.
It is also clear that $\etendu{\perturbed{\tau}}$ respects the counit maps.
\end{proof}

We are now ready to prove Theorem~\ref{thm:Berlin}.

\begin{proof}[Proof of Theorem~\ref{thm:Berlin}]
Since $L=A\bowtie B$ is a matched pair, as vector spaces,
$\tot\big(\wa\otimes_R\Dpoly{\diamond}\big)$
in Proposition~\ref{Dhaka} are isomorphic to
$\tot\big(\wa\otimes_R\enveloping{B}^{\diamond+1}\big)$.
According to Proposition~\ref{Dhaka}, the cochain maps $\etendu{\perturbed{\tau}}$
respects the cup products.
Therefore, it suffices to prove that $\etendu{\perturbed{\tau}}$ in~\eqref{eq:Dhaka}
respects the Lie algebra structures --- the Lie bracket on
$\tot\big(\wa\otimes_R\enveloping{B}^{\diamond+1}\big)$
is the Gerstenhaber bracket of the dg Lie algebroid $\cB\to A[1]$
as in Proposition~\ref{pro:zurich1}.
We know from the general theory of dg Lie algebroids --- see Section~\ref{Najaf}
--- that the Gerstenhaber bracket of a dg Lie algebroid
is completely determined by its multiplication and comultiplication
as shown by Equations~\eqref{hazmat} and~\eqref{eq:pre-Lie}.
The conclusion thus follows immediately from Proposition~\ref{pro:Metz}.
\end{proof}

\appendix

\section{Semifull algebra contractions}

Let $(V,d_V)$ and $(W,d_W)$ be complexes: recall that a \emph{contraction}
of $(V,d_V)$ onto $(W,d_W)$ is the data of dg morphisms $\tau:W\to V$,
$\sigma:V\to W$ and a contracting (degree minus one) homotopy $h:V\to V$ such that
\[ \sigma\tau=\operatorname{id}_W,\qquad \tau\sigma-\operatorname{id}_V
= d_Vh+hd_V,\qquad h\tau=0,\qquad\sigma h=0,\qquad h^2=0. \]

In the following well known \emph{homological perturbation lemma}
\cite{MR1103672,MR1109665} we assume that $V$ and $W$ are equipped
with complete exhaustive decreasing filtrations $F^\bullet V$ and $F^\bullet W$
(we need this hypothesis to ensure convergence of the infinite sums
in the following Lemma~\ref{HPL}), i.e.\
\[ V=F^0 V\supset F^1 V\supset\cdots\supset F^p V\supset \cdots \]
and the natural map $V\to \underleftarrow{\lim}_p V/F^p V$ is an isomorphism,
similarly for $W$. Furthermore, we assume that the differentials $d_V,d_W$
and the contraction data $(\tau,\sigma,h)$ preserve the filtrations.
Recall that a \emph{perturbation} of the differential $d_V$ on $V$
is a degree one map $\perturbation_V:V\to V$ such that $(d_V+\perturbation_V)^2=0$.

\begin{lemma}[Homological Perturbation]\label{HPL}
Given a perturbation $\perturbation_V:V\to V$ of the differential $d_V$ on $V$
such that $\perturbation_V(F^p V)\subset F^{p+1} V$, for all $p\geqslant 0$,
the endomorphism $\perturbation_W$ of $W$ defined by
\[ \perturbation_W := \sum_{l\geqslant 0} \sigma(\perturbation_V h)^l\perturbation_V\tau
= \sum_{l\geqslant 0} \sigma\perturbation_V (h\perturbation_V)^l\tau \]
is a perturbation of the differential $d_W$ on $W$, and
the triple of maps
\[ \perturbed{\tau} := \sum_{l\geqslant 0}(h\perturbation_V)^l\tau ,\qquad
\perturbed{\sigma}:= \sum_{l\geqslant 0} \sigma(\perturbation_V h)^l ,\qquad
\perturbed{h} := \sum_{l\geqslant 0} h(\perturbation_V h)^l
= \sum_{l\geqslant 0} (h\perturbation_V)^l h \]
is a contraction of $(V,d_V+\perturbation_V)$ onto $(W,d_W+\perturbation_W)$.
\end{lemma}

In the following definition, taken from~\cite{MR1782593},
we shall assume given two associative
(but not necessarily graded commutative)
products $\mu_V:V^{\otimes2}\to V$ and $\mu_W:W^{\otimes 2}\to W$:
we do not require a priori $d_V$ and $d_W$ to be algebra derivations.

\begin{definition}\label{def:semifull}
A contraction $(\tau,\sigma,h)$ of $(V,d_V)$ onto $(W,d_W)$
is a \emph{semifull algebra contraction} if the following identities are satisfied
\begin{gather*} h\mu_V(h\otimes h)=0,\qquad h\mu_V(h\otimes \tau)=0,\qquad
h\mu_V(\tau\otimes h)=0,\qquad h\mu_V(\tau\otimes\tau)=0, \\
\sigma\mu_V(h\otimes h)=0,\qquad \sigma\mu_V(h\otimes \tau)=0,\qquad
\sigma\mu_V(\tau\otimes h)=0,\qquad \sigma\mu_V(\tau\otimes\tau)=\mu_W
.\end{gather*}
\end{definition}

\begin{lemma}\label{lem:semifullstable}
Given (i) a pair of cochain complexes $(V,d_V)$ and $(W,d_W)$
carrying additional graded associative algebra structures;
(ii) a semifull algebra contraction $(\tau,\sigma,h)$ of $(V,d_V)$ onto $(W,d_W)$;
and (iii) a perturbation $\perturbation_V:V\to V$ of the differential $d_V$
satisfying $\perturbation_V(F^p V)\subset F^{p+1}V$ for all $p\geqslant 0$,
the perturbed contraction
$(\perturbed{\tau},\perturbed{\sigma},\perturbed{h})$
is a semifull algebra contraction of
$(V,d_V+\perturbation_V)$ onto $(W,d_W+\perturbation_W)$.
\end{lemma}

\begin{proof} Straightforward.
\end{proof}

\begin{lemma}\label{lem:semifulltau}
Given a semifull algebra contraction $(\tau,\sigma,h)$
of $(V,d_V)$ onto $(W,d_W)$, if $d_V$ is an algebra derivation,
then $\tau:W\to V$ is a morphism of algebras.
\end{lemma}

\begin{proof}
Since $\sigma\mu_V(\tau\otimes\tau)=\mu_W$
and $h\mu_V(\tau\otimes\tau)=0$, we have
\begin{multline*}
\mu_V(\tau\otimes\tau)=(\tau\sigma-d_V h-h d_V)\mu_V(\tau\otimes\tau) \\
=\tau\sigma\mu_V(\tau\otimes\tau)-d_V h\mu_V(\tau\otimes\tau)-h d_V\mu_V(\tau\otimes\tau) \\
=\tau\mu_W-h\mu_V(d_V\otimes\id_V+\id_V\otimes d_V)(\tau\otimes\tau) \\
=\tau\mu_W-h\mu_V(\tau\otimes\tau)(d_W\otimes\id_W+\id_W\otimes d_W)
=\tau\mu_W
\qedhere\end{multline*}
\end{proof}

\printbibliography
%\bibliographystyle{amsplain}
%\bibliography{references}
\end{document}